\documentclass[letterpaper, 11pt,  reqno]{amsart}

\usepackage{amsmath,amssymb,amscd,amsthm,amsxtra, esint}

\usepackage[implicit=true]{hyperref}

\allowdisplaybreaks[2]

\usepackage{color}

\definecolor{gr}{rgb}   {0.,   0.69,   0.23 }
\definecolor{bl}{rgb}   {0.,   0.5,   1. }
\definecolor{mg}{rgb}   {0.85,  0.,    0.85}
\definecolor{yl}{rgb}   {0.8,  0.7,   0.}
\definecolor{or}{rgb}  {0.7,0.2,0.2}

\newtheorem{theorem}{Theorem} [section]

\newtheorem{oldtheorem}{Theorem}

\newtheorem{lemma}[theorem]{Lemma}
\newtheorem{proposition}[theorem]{Proposition}
\newtheorem{remark}[theorem]{Remark}

\newtheorem{definition}[theorem]{Definition}

\newtheorem*{acknowledgment}{Acknowledgments}

\DeclareMathOperator*{\intt}{\int}

\DeclareMathOperator*{\supp}{supp}

\newcommand{\I}{\mathcal{I}}

\newcommand{\noi}{\noindent}
\newcommand{\Z}{\mathbb{Z}}
\newcommand{\R}{\mathbb{R}}
\newcommand{\C}{\mathbb{C}}
\newcommand{\T}{\mathbb{T}}

\let\Re=\undefined\DeclareMathOperator*{\Re}{Re}

\let\P= \undefined
\newcommand{\P}{\mathbf{P}}

\newcommand{\N}{\mathcal{N}}
\newcommand{\NB}{\mathbb{N}}

\newcommand{\RR}{\mathcal{R}}

\newcommand{\F}{\mathcal{F}}

\newcommand{\al}{\alpha}

\newcommand{\dl}{\delta}

\newcommand{\nb}{\nabla}

\newcommand{\Dl}{\Delta}
\newcommand{\eps}{\varepsilon}

\newcommand{\G}{\Gamma}
\newcommand{\ld}{\lambda}

\newcommand{\s}{\sigma}
\newcommand{\Si}{\Sigma}
\newcommand{\ft}{\widehat}

\newcommand{\wt}{\widetilde}
\newcommand{\cj}{\overline}

\newcommand{\dt}{\partial_t}

\renewcommand{\l}{\ell}
\renewcommand{\o}{\omega}
\renewcommand{\O}{\Omega}

\newcommand{\les}{\lesssim}
\newcommand{\ges}{\gtrsim}

\newcommand{\jb}[1]
{\langle #1 \rangle}

\newcommand{\ind}{\mathbf 1}

\newcommand{\lr}[1]{\langle #1 \rangle}

\numberwithin{equation}{section}
\numberwithin{theorem}{section}

\newcommand{\too}{\longrightarrow}

\usepackage{tikz}

\usetikzlibrary{shapes.misc}
\usetikzlibrary{shapes.symbols}
\usetikzlibrary{shapes.geometric}
\tikzset{
	dot/.style={circle,fill=black,draw=black,inner sep=0pt,minimum size=0.5mm},
	>=stealth,
	}
\tikzset{
	ddot/.style={circle,fill=white,draw=black,inner sep=0pt,minimum size=0.8mm},
	>=stealth,
	}

\tikzset{decision/.style={ % requires library shapes.geometric
        draw,
        diamond,
        aspect=1.5
    }}

\tikzset{dia2/.style
={diamond,fill=white,draw=black,inner sep=0pt,minimum size=1mm},
	>=stealth,
	}

\tikzset{dia/.style
={star,fill=black,draw=black,inner sep=0pt,minimum size=1mm},
	>=stealth,
	}

\makeatletter
\def\DeclareSymbol#1#2#3{\expandafter\gdef\csname MH@symb@#1\endcsname{\tikz[baseline=#2,scale=0.15]{#3}}}
\def\<#1>{\csname MH@symb@#1\endcsname}
\makeatother

\DeclareSymbol{1}{-2.7}
 {
  \draw (0,0) node[dot]{};}

\DeclareSymbol{1'}{-2.7}
 {\draw (0,0) node[ddot]{};}

\DeclareSymbol{31}{-3}{
\draw (0,-1)node[dot] {} -- (0,0) node[ddot] {}-- (0.9, -1) node[dot] {}; 
\draw (0,-1)node[dot] {} -- (0,0) node[ddot] {}-- (-0.9, -1) node[dot] {}; 
\draw (0,0)node[ddot] {} -- (1.3,1) node[ddot] {}-- (1.3, 0) node[dot] {}; 
\draw  (1.3,1) node[ddot] {}-- (2.6, 0) node[dot] {}; 
}

\DeclareSymbol{32}{-3}{
\draw (0,-1)node[dot] {} -- (0,0) node[ddot] {}-- (0.9, -1) node[dot] {}; 
\draw (0,-1)node[dot] {} -- (0,0) node[ddot] {}-- (-0.9, -1) node[dot] {}; 
\draw (0,0)node[ddot] {} -- (0,1) node[ddot] {}-- (0.9, 0) node[dot] {}; 
\draw (0,0)node[ddot] {} -- (0,1) node[ddot] {}-- (-0.9, 0) node[dot] {}; 
}

\DeclareSymbol{33}{-3}{
\draw (2.6,-1)node[dot] {} -- (2.6,0) node[ddot] {}-- (3.5, -1) node[dot] {}; 
\draw (2.6,-1)node[dot] {} -- (2.6,0) node[ddot] {}-- (1.7, -1) node[dot] {}; 
\draw (0,0)node[dot] {} -- (1.3,1) node[ddot] {}-- (1.3, 0) node[dot] {}; 
\draw  (1.3,1) node[ddot] {}-- (2.6, 0) node[ddot] {}; 
}

\begin{document}
\baselineskip = 15pt

\title[Higher order expansion for  a.s.~LWP of NLS] 
{Higher order expansions
for 
 the   probabilistic local Cauchy theory of the cubic nonlinear Schr\"odinger equation on 
$\R^3$}

\author[\'A.~B\'enyi, T.~Oh, and O.~Pocovnicu]
{\'Arp\'ad  B\'enyi, Tadahiro Oh, and Oana Pocovnicu}

\address{
\'Arp\'ad  B\'enyi\\
Department of Mathematics\\
 Western Washington University\\
 516 High Street, Bellingham\\
  WA 98225\\ USA}
\email{arpad.benyi@wwu.edu}

\address{
Tadahiro Oh\\
School of Mathematics\\
The University of Edinburgh\\
and The Maxwell Institute for the Mathematical Sciences\\
James Clerk Maxwell Building\\
The King's Buildings\\
 Peter Guthrie Tait Road\\
Edinburgh\\ 
EH9 3FD\\United Kingdom} 

\email{hiro.oh@ed.ac.uk}

\address{
Oana Pocovnicu\\
Department of Mathematics, Heriot-Watt University and The Maxwell Institute for the Mathematical Sciences, Edinburgh, EH14 4AS, United Kingdom}
\email{o.pocovnicu@hw.ac.uk}



\begin{abstract}

We consider the  cubic nonlinear Schr\"odinger equation (NLS) on $\R^3$
with randomized initial data.
In particular, we study an iterative approach based on a partial power series expansion
in terms of the random initial data.
By performing a fixed point argument around 
the second order expansion, 
we improve the regularity threshold for almost sure local well-posedness
from our previous work \cite{BOP2}.
We further investigate  a  limitation
of  this iterative procedure. 
Finally, we introduce an alternative iterative approach, 
based on a modified   expansion of arbitrary length, 
and prove almost sure local well-posedness of  the cubic NLS 
in an almost optimal regularity range
with respect to the original iterative approach based
on a  power series expansion.

\end{abstract}

%\date{\today}

%%
%
\maketitle
\tableofcontents

\baselineskip = 14pt

\section{Introduction}
\subsection{Nonlinear Schr\"odinger equation}
We consider
the Cauchy problem of the  cubic nonlinear Schr\"odinger equation (NLS)
on $\R^3$:\footnote{Our discussion in this paper
can be easily adapted to the cubic NLS on $\R^d$
(and to other nonlinear dispersive PDEs).
For the sake of concreteness, however, we only consider the $d = 3$ case in the following.
See also the comment after Theorem \ref{THM:1}
for our particular interest in the three-dimensional problem.}
\begin{equation}
\begin{cases}\label{NLS1}
i \partial_t u + \Delta u = |u|^2 u  \\
u|_{t = 0} = u_0 \in H^s(\R^3),
\end{cases}
\qquad ( t, x) \in \R \times \R^3.
\end{equation}

\noi
The cubic NLS \eqref{NLS1} has been studied
extensively from both the theoretical and applied points of view.
In this paper, 
we continue our study in \cite{BOP1, BOP2}
and further investigate   the probabilistic well-posedness of \eqref{NLS1} 
with  random and   rough initial data.

Recall that the equation \eqref{NLS1}
is scaling critical in $\dot H^\frac{1}{2}(\R^3)$
in the sense that the scaling symmetry:
%$u(t, x) \mapsto \ld u (\ld^2 t, \ld x)$
\begin{align}
u(t, x) \longmapsto \ld u (\ld^2 t, \ld x)
\label{scaling}
\end{align}
%
%\noi
preserves  the homogeneous $\dot{H}^{\frac 12 }$-norm
(when it is applied to functions only of $x$).
It is known that the Cauchy problem \eqref{NLS1}
is locally well-posed in $H^s(\R^3)$ for $s \geq \frac 12$
\cite{CW}.
See also \cite{CKSTT2, HR1, KM2, Dodson} for 
partial\footnote{Namely, either for smoother initial data %than $H^\frac{1}{2}$-functions
or under an extra hypothesis.} global well-posedness and scattering results.
On the other hand, 
\eqref{NLS1} is 
known to be ill-posed in $H^s(\R^3)$ for $s < \frac 12$~\cite{CCT}.
In \cite{BOP2}, 
we studied the probabilistic well-posedness property
of~\eqref{NLS1} below the scaling critical regularity $s_\text{crit} = \frac 12$
under a suitable randomization of initial data;
see  \eqref{rand} below.
In particular, we 
 proved that \eqref{NLS1} is almost surely locally well-posed
 in $H^s(\R^3)$, $s > \frac 14$.
Our main goal in this paper
is to introduce an iterative procedure to improve this regularity
threshold for almost sure local well-posedness.
Furthermore, we
study  a critical regularity associated with this iterative procedure.
By introducing a modified iterative approach,
we then prove almost sure local well-posedness
of \eqref{NLS1} in an almost optimal range 
with respect to the original iterative procedure (Theorem \ref{THM:2}).
Beyond the concrete results in this paper, we believe 
that the iterative procedures based on
(modified) partial power series expansions are 
themselves of  interest
for further development in the well-posedness theory
of dispersive equations with random initial data
and/or random forcing.

Such a probabilistic construction of solutions to dispersive PDEs
first appeared in the works 
by McKean~\cite{McKean} and Bourgain \cite{BO96}
in the study of invariant Gibbs measures for the cubic NLS on $\T^d$, $d = 1, 2$.
In particular, they established almost sure local well-posedness
with respect to particular random initial data, basically corresponding to the Brownian motion.
(These local-in-time solutions
were then extended globally in time by invariance
of the Gibbs measures.
In the following, however, we restrict our attention to local-in-time solutions.)
In \cite{BT1}, 
Burq-Tzvetkov further elaborated this idea 
and  considered a randomization 
for any rough initial condition 
via the Fourier series expansion.
More precisely, 
they studied the cubic nonlinear wave equation (NLW)
on a three-dimensional compact Riemannian manifold
below the scaling critical regularity.
By introducing  a randomization 
via the multiplication of the Fourier coefficients by independent random variables, 
they  established almost sure local well-posedness 
below the critical regularity. 
Such randomization via the Fourier series expansion
is natural 
on compact domains \cite{CO, NS}
and more generally
in situations where the associated elliptic operators have discrete spectra~\cite{Thomann, PRT}. 
See \cite{BOP2, BOP4} for more references therein.

In the following, 
we  go over a randomization suitable for our problem on $\R^3$.
Recall 
the Wiener decomposition of the frequency space  \cite{W}:
$\R^3_\xi = \bigcup_{n \in \Z^3} (n+ (-\frac 12, \frac12]^3)$.
We employ a randomization adapted to this Wiener decomposition.
Let $\psi \in \mathcal S(\R ^3)$ satisfy
\[
\supp \psi \subset [-1,1]^3 
\qquad \text{and}\qquad \sum _{n \in \mathbb{Z}^3} \psi (\xi -n) =1 \quad \text{for any $\xi \in \R ^3$}.
\]

\noi
Then, given a function $\phi$ on $\R^3$, we have
\[ \phi =  \sum _{n \in \mathbb{Z}^3} \psi (D-n) \phi.\]

\noi
We  define the Wiener randomization $\phi^\o$ of $\phi$ by
\begin{equation} \label{rand}
\phi ^{\omega} := \sum _{n \in \mathbb{Z}^3} g_n (\o) \psi (D-n) \phi,
\end{equation}

\noi
where $\{ g_n \}_{n \in \Z^3} $ is a sequence of independent mean-zero complex-valued random variables 
on a probability space $(\Omega , \mathcal{F} ,P)$.
The randomization \eqref{rand}
based on the Wiener decomposition of the frequency space
is natural in view of 
time-frequency analysis;
 see  \cite{BOP1} for a further discussion.
Almost simultaneously with~\cite{BOP1}, 
L\"uhrmann-Mendelson~\cite{LM} also considered 
a randomization of the form \eqref{rand} 
 (with cubes being substituted by appropriately localized balls)
in the study of NLW on $\R^3$.
For a similar randomization used in the study of the Navier-Stokes equations, 
see also the work of Zhang and Fang \cite{ZF}, preceding \cite{BOP1, LM}.
In \cite{ZF},  the decomposition of the frequency space
is not explicitly provided
but their randomization certainly includes
randomization based on the decomposition of the frequency space
by unit cubes or balls.

In the following, 
we assume that 
the  probability distribution $\mu _n$  of $g_n$ satisfies the following exponential moment bound:
\begin{align}
\int _{\R^2} e^{\kappa \cdot x} d \mu _{n} (x) \le e^{c |\kappa| ^2}
\label{exp1}
\end{align}

\noi
for all $\kappa \in \R^2$ and $n \in \mathbb{Z}^3$.
This condition is satisfied by the standard complex-valued Gaussian random variables
 and the uniform distribution on the circle in the complex plane.

We now recall 
 the almost sure local well-posedness result 
from \cite{BOP2} which is of interest to us.\footnote{For simplicity, 
we only consider positive times.
By the time reversibility of the equation, 
the same analysis applies to negative times.}

\begin{oldtheorem} \label{THM:A}
Let $\frac 14 < s < \frac{1}{2}$.
Given $\phi \in H^s (\R ^3)$, let $\phi^{\omega}$ 
be its Wiener randomization defined in \eqref{rand}.
Then, 
the Cauchy problem \eqref{NLS1} is almost surely locally well-posed
with respect to the random initial data $\phi^\o$.
More precisely,
there exists a set $\Si = \Si(\phi) \subset \O$ 
with $P(\Si) = 1$ such that, for any $\o \in \Si$, there exists
a unique function  $u = u^\o$ in the class:
\begin{align*}
S(t)\phi^\o + X^\frac{1}{2}([0, T]) \subset
 S(t) \phi ^{\omega} + C([0,T] ; H^\frac{1}{2}(\R^3)) \subset C([0,T] ; H^s(\R^3))
\end{align*}

\noi
with $T = T(\phi, \o) >0$ 
such that $u$ is a solution to \eqref{NLS1} on $[0, T]$.
\end{oldtheorem}

Here,  $S(t) := e^{it\Dl}$ denotes the linear Schr\"odinger operator
and the function space $X^\frac{1}{2}([0, T]) \subset C([0, T]; H^\frac{1}{2}(\R^3))$ 
is defined in Section \ref{SEC:2} below.
Note that the uniqueness statement of Theorem \ref{THM:A} in the class: 
\begin{align}
S(t)\phi^\o + X^\frac{1}{2}([0, T])
\label{uniq1}
\end{align}

\noi
is to be interpreted as follows;
by setting $v = u - S(t) \phi^\o$ (see \eqref{v1} below), 
 uniqueness for the residual term $v$ holds in $X^\frac{1}{2}([0, T])$.
 See also Remark \ref{REM:uniq1} below.

Recall that while the Wiener randomization~\eqref{rand} does not improve differentiability, 
 it  improves integrability (see Lemma~4 in \cite{BOP1}).
See \cite{PZ,  Kahane, BT1}
for the corresponding statements in the context of the 
 random Fourier series. 
The main idea for proving Theorem \ref{THM:A} 
is to exploit this gain of integrability.
More precisely, 
let
\begin{align}
z_1(t)  := S(t) \phi ^{\omega}
\label{introZ1}
\end{align}

\noi
 denote the random linear solution
with $\phi^\o$ as initial data
and write
\begin{align}
u = z_1 + v.
\label{v1}
\end{align}
	
\noi
Then, 	
we see that
$v: = u - z_1$
satisfies the following perturbed NLS:
\begin{equation}
\begin{cases}
	 i \dt v + \Dl v =  |v + z_1|^2(v+z_1)\\
v|_{t = 0} = 0,
 \end{cases}
\label{NLS2}
\end{equation}

\noi
where $z_1$ is viewed as a given (random)  source term.
The main point is that the gain of space-time
integrability of the random linear solution $z_1$ (Lemma \ref{LEM:PStr})
makes this problem {\it subcritical}\footnote{The scaling critical Sobolev regularity $s_\text{crit} = \frac 12$
for \eqref{NLS1}
is defined by the fact that the $\dot H^\frac12$-norm remains invariant under 
the scaling symmetry \eqref{scaling}.
Given $1 \leq r \leq \infty$, we can also define  
the scaling critical Sobolev regularity $s_\text{crit}(r)$
in terms of the $L^r$-based Sobolev space $\dot W^{s, r}$.
A direct computation gives 
$s_\text{crit}(r) = \frac 3r - 1\,  (\to -1$ as $r\to \infty$).
In particular,  the gain of integrability of the random linear solution $z_1$
stated in Lemma \ref{LEM:PStr}
implies that $z_1$ in \eqref{NLS2} gives rise only to a subcritical perturbation.
See \cite{BOP4} for a further discussion on this issue.

Note, however, 
that if  we consider a non-zero initial condition for $v$, 
then the critical nature of the equation comes into play through
the initial condition.  See Remark \ref{REM:v_0}. This plays an important role in studying 
the global-in-time behavior of solutions.
See, for example, \cite{OOP}.}
and hence we can solve it by a standard fixed point argument.
Over the last several years, there have been 
many results on probabilistic  well-posedness of nonlinear dispersive PDEs, % on $\R^d$,
using this change of viewpoint.\footnote{In the field
of stochastic parabolic PDEs, this change of viewpoint
and solving the fixed point problem for the residual term $v$ is called
the Da Prato-Debussche trick \cite{DPD, DPD2}.}
See, for example,   \cite{McKean, BO96, BT1,  CO, BT3, NS,  LM, BOP1, BOP2, POC, OP, HO16, 
 LM2,    Bre, KMV, OOP}.
In \cite{BOP2}, 
we studied  the Duhamel formulation for~\eqref{NLS2}:
\begin{equation}
v(t) = - i \int_0^t S(t-t')	  |v + z_1|^2(v+z_1)(t') dt',
\label{NLS3}
\end{equation}

\noi
by carrying  out  case-by-case analysis
on terms of the form $v \cj v v$, $v \cj v z_1$, $v \cj z_1 z_1$, etc.~in $X^\frac{1}{2}([0, T])$.
The main tools were 
(i)  the gain of space-time integrability of $z_1$
and (ii)~the bilinear refinement of the Strichartz estimate (Lemma \ref{LEM:biStr}).
This yields Theorem \ref{THM:A}.

By examining the case-by-case analysis in \cite{BOP2}, 
we see that 
the regularity restriction $s >  \frac 14$ in Theorem \ref{THM:A}
comes from 
the cubic interaction 
of the random linear solution:
\begin{align}
  z_3(t)  := -i  \int_0^t S(t - t') |z_1|^2 z_1(t') dt'. 
\label{introZ3}
\end{align}

\noi
See Proposition \ref{PROP:Z3} below.
In the following, we discuss
an iterative approach to lower this regularity threshold
by studying further expansions
in terms of the random linear solution~$z_1$.
We will also discuss  
the limitation of this iterative procedure.

\begin{remark}\label{REM:uniq1}\rm
(i) 
The proof of Theorem \ref{THM:A}, presented in \cite{BOP2}, is based on a standard contraction argument for the residual term $v = u - S(t) \phi^\o$
 in a ball $B$ of radius  $O(1)$ in $X^\frac{1}{2}([0, T])$.
As such, the argument in \cite{BOP2}
only yields uniqueness of $v$ in the ball $B$.
Noting that 
$v \in C([0, T]; H^{\frac12}(\R^3)) \cap X^\frac{1}{2}([0, T])$,\footnote{By convention, 
our definition of 
the $X^s([0, T])$-space already assumes that functions in 
$X^s([0, T])$ belong to 
$ C([0, T]; H^s(\R^3))$.
See Definition \ref{DEF:X3} below.}
it follows from  Lemma A.8  in \cite{BOP2} that 
the $X^\frac{1}{2}([0, t])$-norm of $v$ is continuous in $t \in [0, T]$.
Then, by possibly shrinking the local existence time $T>0$, 
we can easily upgrade the uniqueness of $v$ in the ball $B$
to uniqueness of $v$ in the entire $X^\frac{1}{2}([0, T])$.
Namely, uniqueness of $u$ holds in the class \eqref{uniq1}.

\smallskip

\noi
(ii) With the exponential moment assumption \eqref{exp1}, 
the proof of Theorem \ref{THM:A}, presented in~\cite{BOP2}, 
allows us to conclude that
the set $\Si$ of full probability in Theorem \ref{THM:A} has the following decomposition:
\[ \Si = \bigcup_{0 < T \ll 1} \Si_T\]

\noi
such that
\begin{itemize}
\item[(a)] there exist $c, C >0$ and $C_0 = C_0 (\|\phi\|_{H^s})>0$ such that 
\[P(\Si_T^c) < C \exp\Big(-\frac{C_0(\|\phi\|_{H^s})}{T^c  }\Big)\]

\noi
 for each $0 < T \ll 1$, 

\smallskip

\item[(b)] for each $\o \in \Si_T\subset \Si$, $0 < T \ll 1$, 
the function $u = u^\o$ constructed in Theorem \ref{THM:A}
is a solution to \eqref{NLS1} on $[0, T]$.

\end{itemize}
\noi
See the statement of Theorem 1.1 in \cite{BOP2}.
A similar decomposition of the set $\Si$ of full probability
 applies to Theorems \ref{THM:1} and \ref{THM:2} below.
We, however, do not state it in an explicit manner in the following.
\end{remark}

\subsection{Improved almost sure local well-posedness}

Let us first state the following proposition on the regularity
property of the cubic term $z_3$ defined above.

\begin{proposition}\label{PROP:Z3}
Given $0 \leq s < 1$, 
 let $\phi^\o$ be the Wiener randomization 
of  $\phi \in H^s(\R^3)$
 defined in \eqref{rand}
 and set $z_1 = S(t) \phi^\o$.

\smallskip

\noi
\textup{(i)}
For any $\s < 2s$, we have
\[z_3   \in X^\s_\textup{loc}\subset C(\R; H^\s(\R^3))\]

\noi
almost surely.
More precisely,  there exists an almost surely finite constant $ C(\o, \|\phi\|_{H^s} ) > 0$ and $\theta > 0$
such that 
\begin{align}
\| z_3 \|_{X^{\s}([0, T])} = \bigg\| \int_0^t S(t - t') |z_1|^2 z_1(t') dt'\bigg\|_{X^\s([0, T])}
\leq T^\theta C(\o,  \|\phi\|_{H^s}) 
\label{Z3_2}
\end{align}

\noi
for any $T > 0$.
In particular, by taking $\s = 2s-\eps$ for small $\eps > 0$, \eqref{Z3_2} shows that 
the second order\footnote{Here, we referred to $z_3$ as the second order term since it corresponds to the second order term
appearing in the (formal) power series expansion of solutions to \eqref{NLS1}
in terms of the random initial data.
For the same reason, we refer to $z_5$ in \eqref{introZ5} and $z_7$ 
in \eqref{introZ7}
as the third and fourth order terms
in the following.  See Subsection \ref{SUBSEC:1.3} below.} term~$z_3$ is smoother 
\textup{(}by $s - \eps$\textup{)} than 
the first order term $z_1$, provided $s > 0$.

\smallskip

\noi
\textup{(ii)}
When $s = 0$, there is no smoothing in the second order
 term $z_3$
in general; 
there exists $\phi \in L^2(\R^3)$ such that the estimate~\eqref{Z3_2} with $\s = \eps > 0$ 
fails  for any $\eps > 0$.

\end{proposition}

In \cite{BOP2}, 
we already proved \eqref{Z3_2} when $\s = \frac 12$ and $s > \frac 14$,
giving the regularity restriction in Theorem \ref{THM:A}.
See Section \ref{SEC:Z3} for the proof of Proposition~\ref{PROP:Z3}
for a general value of $\s$.
In the proof of Theorem \ref{THM:A},
in order to carry out the case-by-case analysis for \eqref{NLS3}, 
we need to have $z_3 \in X^\frac{1}{2}_\text{loc}$,
(where we have a deterministic local well-posedness for \eqref{NLS1}).
In view of 
Proposition~\ref{PROP:Z3}, 
this imposes the regularity restriction $s > \frac 14$. 
Note, however, that even when $s \leq  \frac 14$, 
 $z_3$ is still a well defined space-time function of spatial regularity $2s- < \frac 12 $.
This motivates us 
to consider the following second order expansion:
\begin{align}
u = z_1 + z_3 + v
\label{v2}
\end{align}

\noi
and remove the second order  interaction $z_1 \cj z_1 z_1$.
Indeed,  the residual term $v := u - z_1 - z_3$ satisfies the following equation:
\begin{equation}
\begin{cases}
	 i \dt v + \Dl v =  \N(v + z_1+z_3) - \N(z_1)\\
v|_{t = 0} = 0,
 \end{cases}
\label{NLS4}
\end{equation}

\noi
where $\N(u) = |u|^2 u$.
In terms of the Duhamel formulation, we have
\begin{equation}
v(t) = - i \int_0^t S(t-t')\big\{\N(v + z_1 + z_3) - \N(z_1)\big\}(t') dt'.
\label{NLS5}
\end{equation}

\noi
Then, by studying the fixed point problem \eqref{NLS5} for $v$, 
we have the following improved almost sure local well-posedness of \eqref{NLS1}.

\begin{theorem}
\label{THM:1}
Given $ \frac 12 \leq  \s \leq  1$, 
let   $ \frac 25 \s < s < \frac 12$.
Given $\phi \in H^s(\R^3)$, let $\phi^\o$ be its Wiener randomization defined in \eqref{rand}.
Then, the cubic NLS \eqref{NLS1} on $\R^3$
is almost surely locally well-posed
with respect to the random initial data $\phi^\omega$.
More precisely,
there exists a set $\Si = \Si(\phi, \s) \subset \O$ 
with $P(\Si) = 1$ such that, for any $\o \in \Si$, there exists
a unique function  $u = u^\o$ in the class:
\begin{align*}
z_1  + z_3 
 + X^\s([0, T]) 
&   \subset
 z_1 + z_3 
+ C([0, T]; H^{\s} (\R^3)) 
\notag\\\
&  
\subset C([0, T];H^s(\R^3))
\end{align*}

\noi
with  $T = T(\phi, \o) >0$ 
such that $u$ is a solution to \eqref{NLS1} on $[0, T]$.

\end{theorem}

As in Theorem \ref{THM:A}, 
the uniqueness of $u$ in the class
$z_1  + z_3 
 + X^\s([0, T])$
is to be interpreted as
uniqueness of the residual term $v = u - z_1 - z_3$
in $X^\s([0, T])$.
See also Remark \ref{REM:uniq1}.

By taking $\s = \frac 12$, Theorem \ref{THM:1} states that \eqref{NLS1} is almost surely 
locally well-posed in $H^s(\R^3)$, provided $s > \frac 15$.
This in particular improves 
the almost sure local well-posedness in  Theorem \ref{THM:A}.
On the other hand, 
by taking $\s = 1$,  Theorem \ref{THM:1} allows
us to construct a solution $v = u - z_1 - z_3 
\in X^1([0, T]) \subset C([0,T] ; H^1(\R^3))$, 
provided that $s > \frac 25$.
In particular, we can take random initial data below the scaling critical regularity $s_\text{crit} = \frac 12$,
while we construct the residual part $v$ in $X^1([0, T])$.
This opens up a possibility of studying the global-in-time behavior 
 of $v$, using the
(non-conserved) energy of $v$:
\begin{align*}
 E(v)(t) = \frac 12 \int_{\R^3} |\nb v(t, x)|^2 dx
+\frac {1}{4}  \int_{\R^3} |v(t, x)|^4 dx
\end{align*}

\noi
with random initial data below the scaling critical regularity.
We remark that 
by inspecting the argument in \cite{BOP2}, 
a modification of (the proof of) Theorem~\ref{THM:A}
yields almost sure local well-posedness of \eqref{NLS1}
in the class:
\[ S(t)\phi^\o + X^1([0, T]) \subset
 S(t) \phi ^{\omega} + C([0,T] ; H^1(\R^3))\]

\noi
only for $s > \frac 12$. (This restriction on $s$ can be easily seen
by setting $\s = 1$ in Proposition \ref{PROP:Z3}.)
In particular, Theorem~\ref{THM:A} does not allow
us to take random initial data below the scaling critical regularity $s_\text{crit} = \frac 12$
in studying the global-in-time behavior of $v \in X^1([0, T])$, 
namely at the level of the energy $E(v)$.
As our focus in this paper is the local-in-time analysis, 
we do not pursue further this issue 
on 
almost sure  global well-posedness of \eqref{NLS1} 
below the scaling critical regularity
 in this paper.
We, however, point out 
two  recent results \cite{KMV, OOP} on almost sure global well-posedness
 below the energy space
for  the defocusing energy-critical NLS in higher dimensions.

The main strategy for proving Theorem \ref{THM:1} is
to study the fixed point problem \eqref{NLS5}
by carrying out case-by-case analysis on 
\begin{align}
w_1\cj {w_2} w_3, 
\quad \text{for  $w_i = v,$ $z_1$, or $z_3$, $i = 1, 2, 3$,  but not all $w_i$ equal to $z_1$}
\label{case1}
\end{align}

\noi in $N^\s([0, T])$, where 
the dual norm is defined by 
\begin{align*}
\| F\|_{N^\s([0, T])} = \bigg\|\int_{0}^t S(t - t') F(t') dt'\bigg\|_{X^\s([0, T])}.
\end{align*}

\noi
Note that the number of cases has increased 
from the case-by-case analysis in  the proof of Theorem \ref{THM:A},
where we had $w_i = v$ or $z_1$.
One of the main ingredients is the smoothing on $z_3$
stated  in  Proposition~\ref{PROP:Z3} above.
Note, however, that 
in order to exploit  this smoothing,
we need to measure $z_3$ in the $X^{2s-}([0, T])$-norm,
which imposes a certain rigidity on the space-time integrability.\footnote{Namely, 
we need to  measure $z_3$ in $L^q_t([0, T]; W^{2s-, r}(\R^3))$
for admissible pairs $(q, r)$.  See Lemma \ref{LEM:Str}.}
In order to prove Theorem~\ref{THM:1}, 
we also need to exploit a gain of integrability on $z_3$.
In  Lemma \ref{LEM:Z3_2}, 
we use the dispersive estimate (see \eqref{O9} below) and the gain of integrability
on  each $z_1$ of the three factors in \eqref{introZ3}
and show that $z_3$ also enjoys a gain of integrability by giving up some differentiability.

\begin{remark}\label{REM:v_0}\rm
Given  $\phi \in H^s(\R^3)$, 
let  $\phi^\o$ be its Wiener randomization defined in \eqref{rand}.
Given  $v_0 \in H^\frac{1}{2}(\R^3)$, 
we can also consider  \eqref{NLS1} with the random initial data
of the form $u_0^\o = v_0 + \phi^\o$:
\begin{equation}
\label{NLS5a}
\begin{cases}
 i \dt u + \Dl  u = |u|^2 u\\
 u|_{t = 0} = v_0+ \phi^\o.
\end{cases}
\end{equation}

\noi
Then, by slightly modifying the proofs, 
we see that the analogues  of  Theorems~A and~\ref{THM:1} (with $\s = \frac 12$)
also hold
for \eqref{NLS5a}.
Namely, \eqref{NLS5a} is almost surely locally well-posed, 
provided $s > \frac 15$.
This amounts to considering the following Cauchy problems:
\begin{equation*}
\begin{cases}
	 i \dt v + \Dl v =  |v + z_1|^2(v+z_1)\\
v|_{t = 0} = v_0,
 \end{cases}
\label{NLS5b}
\end{equation*}

\noi
when $\frac 14 < s < \frac 12$
and
\begin{equation*}
\begin{cases}
	 i \dt v + \Dl v =  \N(v + z_1+z_3) - \N(z_1)\\
v|_{t = 0} = v_0,
 \end{cases}
%\label{NLS5c}
\end{equation*}

\noi
when $\frac 15 < s \leq \frac 14$.
In this case, the critical nature of the  problem appears
through the $v \cj v v$ interaction 
in the case-by-case analysis \eqref{case1} 
due to 
 the deterministic (non-zero) initial data $v_0$ at the critical regularity.
The required modification is straightforward 
and thus we omit details.
See Proposition 6.3 in \cite{BOP2}
and Lemma 6.2 in~\cite{OOP}.
We point out that the discussion in the next subsection
also applies to \eqref{NLS5a}.

\end{remark}

\begin{remark}\label{REM:Z3_1}\rm

(i) 
Let $\I(u_1, u_2, u_3)$ denote  the trilinear operator  defined by 
\begin{align}
 \I(u_1, u_2, u_3) (t):= -i \int_0^t S(t - t') u_1 \cj {u_2} u_3 (t') dt'.
 \label{Duhamel1}
\end{align}

\noi
Then, 
for  $\s > 3s-1$, 
there is no finite constant $C > 0$ such that 
\begin{align}
\big\|\I( u_1, u_2, u_3)\big\|_{X^\s([0, 1])}
\leq  C \prod_{j = 1}^3 \| \phi_j  \|_{H^s}
\label{Z3_3}
\end{align}

\noi
for all $\phi_j \in H^s(\R^3)$, 
where  $u_j = S(t) \phi_j$.
See Appendix~\ref{APP:A}.
In particular, this shows that 
when $s \leq \frac 12$,
there is no deterministic smoothing for $\I$,  
i.e.~\eqref{Z3_3} does not hold for $\s > s$.

\smallskip

\noi
(ii)
The proof of Proposition~\ref{PROP:Z3}
only exploits  ``the randomization at the linear level''.
Given $s \in \R$, 
let $\RR^s$ denote the class of functions  defined by 
\begin{align*}
 \RR^s = \big\{ u \text{ on } \R\times \R^3:
\ & i \dt u  + \Dl u = 0 \text{ and } \notag\\
& 
u \in L^q_{t, \text{loc}} W^{s, r}_x(\R^3)
\text{ for any } 2 \leq q, r < \infty \big\}.
\end{align*}

\noi
Under the regularity assumption:\,$\s < 2s$ and $0 \leq s < 1$, 
it follows from the proof of Proposition \ref{PROP:Z3} 
that 
 the left-hand side of \eqref{Z3_3} is finite
  for any $u_1, u_2, u_3 \in \RR^s$.\footnote{Here, 
we only need finiteness of the $A_3^s$-norm defined in \eqref{O1a}
for each $u_j$, $j = 1, 2, 3$.
See Remark \ref{REM:Z3_2}.}
In other words, the proof of 
 Proposition \ref{PROP:Z3}  
 only uses the fact that the random linear solution 
 $z_1 = S(t) \phi^\o$ belongs to $\RR^s$ almost surely;
 see the probabilistic Strichartz estimate (Lemma~\ref{LEM:PStr}).
The multilinear random structure of $z_3$ 
in terms
of the random linear solution $z_1$ yields 
 further cancellation.
See, for example, Lemma 3.6 in \cite{CO}.
Such extra cancellation seems to  improve
only space-time integrability and we do not know how to use it
to improve the regularity threshold (i.e.~differentiability) at this point.
A similar comment applies to the unbalanced higher order terms
$\zeta_{2k-1}$ (including $z_5$ below) studied in Proposition~\ref{PROP:Zk} below.

\end{remark}

\subsection{Partial power series expansion
and the associated critical regularity}
\label{SUBSEC:1.3}

In this subsection,  we  discuss possible improvements 
over Theorem \ref{THM:1} by considering further expansions.
For this purpose, we fix $ \s = \frac 12$ in the following.
By examining the proof of Theorem \ref{THM:1}, 
we see that the regularity restriction $s > \frac 15$ comes from the  following third order term:
\begin{align}
  z_5 (t) := -i  \sum_{\substack{j_1 + j_2 + j_3 = 5\\j_1, j_2, j_3 \in \{1, 3\}  }}\int_0^t S(t - t') 
  z_{j_1} \cj{z_{j_2}}z_{j_3} (t') dt'.
\label{introZ5}
\end{align}

\noi
Namely, we have  $(j_1, j_2, j_3) = (1, 1, 3)$ up to permutations.
In Lemma \ref{LEM:Z5}, we show that 
given $0< s< \frac 12$, 
we have $  z_5  \in X^{\frac 52 s -}_\textup{loc}$.
In particular, 
we have $  z_5  \in X^\frac{1}{2}_\textup{loc}$,
provided $s > \frac 15 $, yielding the regularity threshold
in Theorem \ref{THM:1}.

A natural
next step is to 
remove this non-desirable third order interaction:
\[ \sum_{\substack{j_1 + j_2 + j_3 = 5\\j_1, j_2, j_3 \in \{1, 3\}  }}
  z_{j_1} \cj{z_{j_2}}z_{j_3}\]
  
  \noi
in the case-by-case analysis in \eqref{case1} by 
considering 
 the following  third order expansion:
\begin{align}
u = z_1 + z_3 + z_5+ v.
\label{v3}
\end{align}

\noi
In this case, the residual term $v := u - z_1 - z_3- z_5$ satisfies the following equation:
\begin{equation}
\begin{cases}
\displaystyle
 i \dt v + \Dl v =  \N(v + z_1+z_3+z_5) 
-   \sum_{\substack{j_1 + j_2 + j_3 \in \{3, 5\} \\j_1, j_2, j_3 \in \{1, 3\}  }}
 z_{j_1} \cj{z_{j_2}}z_{j_3}\\
v|_{t = 0} = 0.
 \end{cases}
\label{NLS6}
\end{equation}

\noi
We expect that \eqref{NLS6} is almost surely locally well-posed
for $s > \frac 2{11}$, 
which would be an improvement over Theorem \ref{THM:1}. 
The proof will be once again based on 
case-by-case analysis:
\begin{align*}
w_1\cj{w_2} w_3
\quad & \text{for  $w_i = v, z_1, z_3$, or $z_5$, $i = 1, 2, 3$, such that }  \\
& \text{it is not of the form $z_{j_1}\cj{z_{j_2}}z_{j_3} $ with $j_1 + j_2 + j_3 \in \{3, 5\}$}
\end{align*}

\noi in $N^\frac{1}{2}([0, T])$.
Note the increasing number of combinations.
In the following, however, we do not discuss details of  this particular  improvement over Theorem \ref{THM:1}.
Instead, we consider further iterative steps and discuss a possible limitation of this procedure.

\begin{remark}\rm
We point out that the expansions \eqref{v1}, \eqref{v2}, and \eqref{v3}
correspond to partial power series expansions
of the first, second, and third orders,\footnote{A (partial) power series expansion of a solution to \eqref{NLS1}
in terms of the random initial data
can be expressed as a summation of certain multilinear operators over
ternary trees.
See, for example, \cite{Christ, O17}.
Here, the term ``order'' in our context corresponds to  ``generation (of the associated trees) $+1$''
with the convention that the trivial tree consisting only of the root node
is of the zeroth generation.
For example, the third  order term $z_5$ in~\eqref{introZ5} appears as the summation
over all the multilinear operators associated to the ternary trees of the third  generation.
In terms of the graphical representation in \cite{O17}, we have
\begin{align*}
z_5 \,
 \text{``}\!=\!\text{''} \,\<31>\,+ \,\<32>\,+\,   \<33> \ , 
 \end{align*}

\noi
where 
``\,$\<1>$\,'' 
denotes  the random linear solution $z_1 = S(t) \phi^\o$
and 
``\,$\<1'>$\,'' 
denotes  the trilinear Duhamel integral operator 
$ \I(u_1, u_2, u_3)$ defined in \eqref{Duhamel1}
with its three children as its arguments $u_1, u_2$, and $u_3$.}
 respectively,
 of a solution to \eqref{NLS1}
in terms of the random initial data.
Then, by considering 
 the associated equations \eqref{NLS2}, \eqref{NLS4}, and \eqref{NLS6}
for the residual term $v$, 
we are  recasting  the original problem~\eqref{NLS1}
as a fixed point problem centered 
at the partial power series expansions
of the first, second, and third orders, respectively.

\end{remark}

By drawing an analogy to the previous steps, 
we expect that the worst contribution comes from the following fourth order terms:
\begin{align}
  z_7 (t): = -i  \sum_{\substack{j_1 + j_2 + j_3 = 7\\j_1, j_2, j_3 \in \{1, 3, 5\}  }}\int_0^t S(t - t') z_{j_1} \cj{z_{j_2}}z_{j_3} (t') dt'.
\label{introZ7}
\end{align}

\noi
There are basically two contributions to \eqref{introZ7}:
$(j_1, j_2, j_3) = (1, 3, 3)$ or $(1, 1, 5)$ up to permutations.
In Lemma \ref{LEM:Z7}, we show that the contribution from 
$(j_1, j_2, j_3) = (1, 1, 5)$ is worse, being responsible for the expected regularity restriction $s > \frac{2}{11}$.
In order to remove this term, 
we can consider 
 the following fourth order expansion:
\begin{align*}
u = z_1 + z_3 + z_5+z_7+ v
%\label{v4}
\end{align*}

\noi
as in the previous steps
and 
try to solve the following equation 
for the residual term $v = u - z_1 - z_3- z_5-z_7$:
\begin{equation*}
\begin{cases}
\displaystyle
 i \dt v + \Dl v =  \N(v + z_1+z_3+z_5+z_7) 
-   \sum_{\substack{j_1 + j_2 + j_3 \in \{3, 5, 7\} \\j_1, j_2, j_3 \in \{1, 3, 5\}  }}
 z_{j_1} \cj{z_{j_2}}z_{j_3}\\
v|_{t = 0} = 0.
 \end{cases}
%\label{NLS7}
\end{equation*}

We can obviously iterate this argument
and consider the following $k$th order expansion:
\begin{align}
u = \sum_{\l = 1}^{k} z_{2\l-1}+ v.
\label{vJ}
\end{align}

\noi
In this case, 
we need to consider the following equation 
for the residual term
$v := u - \sum_{\l = 1}^{k} z_{2\l-1}$:
\begin{equation}
\begin{cases}
\displaystyle
 i \dt v + \Dl v =  \N\bigg(v +\sum_{\l= 1}^{k} z_{2\l-1}\bigg) 
- 
 \sum_{\substack{j_1 + j_2 + j_3 \in \{3, 5, \dots, 2k-1\} 
\\j_1, j_2, j_3 \in \{1, 3, \dots, 2k-3\}}} 
 z_{j_1} \cj{z_{j_2}}z_{j_3}\\
v|_{t = 0} = 0
 \end{cases}
\label{NLS8}
\end{equation}

\noi
and hope to construct a solution $v \in X^\frac{1}{2}([0, T])$ 
by carrying out the following case-by-case analysis:
\begin{align}
\begin{split}
w_1\cj{w_2} w_3, 
\quad & \text{for  $w_i = v, z_j,$ $j \in \{1, 3, \dots, 2k-1\}$, 
 $i = 1, 2, 3$, such that }  \\
&  \text{it is not of the form $z_{j_1}\cj{z_{j_2}}z_{j_3} $ with 
$j_1 + j_2 + j_3  
\in  \{3, 5, \dots, 2k-1\}$}
\end{split}
\label{casek}
\end{align}

\noi 
in $N^\frac{1}{2}([0, T])$.
Then, a natural question to ask
is 
{\it``Does this iterative procedure 
work indefinitely, allowing us to arbitrarily lower the regularity threshold
for almost sure local well-posedness for~\eqref{NLS1}?
Or is there any limitation to it?}

We now  
consider a ``critical'' regularity $s_* < \frac 12$
with respect to this iterative procedure
for proving almost sure local well-posedness of \eqref{NLS1}.
We simply define 
the critical regularity $s_* <\frac 12 $ for this problem 
 to be the infimum of  the values of $s< \frac 12$
such that given any  $\phi \in H^s(\R^3)$, 
 the above iterative procedure\footnote{Namely, 
we  solve \eqref{NLS8} for $v \in X^\frac{1}{2}([0, T])$
with some finite (or infinite) number of steps.}
shows that \eqref{NLS1} is almost surely locally well-posed
with respect to the Wiener randomization $\phi^\o$ of $\phi$.
This is an empirical notion of criticality;
unlike the scaling criticality, 
we can not a priori compute this critical regularity~$s_*$.
Moreover, our discussion will be   based on the estimates on the stochastic
multilinear terms (Proposition~\ref{PROP:Z3} and Proposition~\ref{PROP:Zk} below).
In the following, we 
discuss a (possible) lower bound on $s_*$,
presenting a limitation to our iterative procedure
based on partial power series expansions.

Within the framework of the iterative procedure discussed  above, 
a necessary condition for carrying out the case-by-case analysis \eqref{casek}
 to study  \eqref{NLS8}
 in  $X^\frac{1}{2}([0, T])$
is that 
the $(k+1)$st order term 
$  z_{2k+1} =   z_{2(k+1)-1}$
defined in a recursive manner:
\begin{align}
  z_{2k+1}(t)  : = -i  \sum_{\substack{j_1 + j_2 + j_3 = {2k+1}
  \\j_1, j_2, j_3 \in \{1, 3, \dots, 2k-1\}  }}\int_0^t S(t - t') z_{j_1} \cj{z_{j_2}}z_{j_3} (t') dt'
\label{introZk}
\end{align}

\noi
belongs to  $X^\frac{1}{2}([0, T])$.
By the nature of this iterative procedure, we may assume that 
 the lower order terms 
 $z_{2\l-1}$,  $\l \in \{1, 3, \dots, k\}$, 
belong to $X^{ s_\l}([0, T])$ for some $ s_\l < \frac{1}{2}$ but not in $X^\frac{1}{2}([0, T])$,
since if any of 
 the lower order terms, say $z_{2\l-1}$ for some $\l \in \{1, 3, \dots, k\}$, 
 were in $X^\frac{1}{2}([0, T])$, 
then  we would have stopped the iterative procedure at the $(\l-1)$th step.
As in the previous steps, 
we expect 
that  $z_{2k+1}$ is responsible for a regularity restriction
at the $k$th step of this iterative approach.

It could be a cumbersome task to study the regularity
property of $z_{2k+1}$ due to 
the increasing number of combinations
for $z_{j_1}\cj{z_{j_2}} z_{j_3}$, satisfying
$j_1 + j_2 + j_3 = {2k+1}$, $j_1, j_2, j_3 \in \{1, 3, \dots, 2k-1\}$.
In the following,  we instead study the regularity property of the $(k+1)$st order term\footnote{In fact, 
we study the $k$th order term $\zeta_{2k-1}$ in Proposition \ref{PROP:Zk}.} of a particular form,
corresponding to 
  $(j_1, j_2, j_3) = (1, 1, 2k - 1)$ up to permutations
in \eqref{introZk};
given an integer $k \geq 0$, 
define  $\zeta_{2k+1}$ by 
setting $\zeta_1 := z_1 = S(t) \phi^\o$
and 
\begin{align}
\zeta_{2k+1}(t) :=    -i 
 \sum_{\substack{j_1 + j_2 + j_3 = {2k+1}
  \\j_1, j_2, j_3 \in \{1,   2k-1\}  }}
\int_0^t S(t - t') \zeta_{j_1} \cj{\zeta_{j_2}} \zeta_{j_3} (t') dt'.
\label{zeta}
\end{align}

\noi
As mentioned above, 
 there are only three terms
in this sum:   $(j_1, j_2, j_3) = (1, 1, 2k - 1)$ up to permutations.
Hence,  $\zeta_{2k+1}$ consists of the ``unbalanced''\footnote{By associating
the $(2k+1)$-linear terms appearing in the summation in \eqref{introZk}
with ternary trees of the $k$th generation as in \cite{O17}, 
the summands in \eqref{zeta} correspond
to the ``unbalanced'' trees of the $k$th generation,
where two of the three children of the root node are terminal.}
$(2k+1)$-linear terms appearing in the definition~\eqref{introZk}
of $z_{2k+1}$.
We claim that  the $(k+1)$st order term $\zeta_{2k + 1}$ is responsible for the regularity restriction at the $k$th step
of the iterative procedure.
See Theorem \ref{THM:2} below.

In anticipating the alternative expansion \eqref{vk1} below, 
let us study the regularity property of the unbalanced $k$th order term $\zeta_{2k-1}$:
\begin{align}
\zeta_{2k-1}(t) :=    -i 
 \sum_{\substack{j_1 + j_2 + j_3 = {2k-1}
  \\j_1, j_2, j_3 \in \{1,   2k-3\}  }}
\int_0^t S(t - t') \zeta_{j_1} \cj{\zeta_{j_2}} \zeta_{j_3} (t') dt'.
\label{zeta2}
\end{align}

\noi
For  $k = 2, 3, 4$, we have (with appropriate restrictions on the range of $s$)
\begin{align}
 \zeta_3 = z_3 \in X^{2s-}_\text{loc}, 
\qquad
 \zeta_5 = z_5 \in X^{\frac 5 2s-}_\text{loc}, 
 \qquad \text{and}\qquad 
  \zeta_7  \in X^{\frac{11}{4}s-}_\text{loc}.
  \label{intro1}
\end{align}

\noi
See Proposition \ref{PROP:Z3}
and Lemmas \ref{LEM:Z5} and \ref{LEM:Z7}.
In general, we have the following proposition.

\begin{proposition}\label{PROP:Zk}

Define a sequence $\{\al_k\}_{k \in \NB}$ of positive real numbers 
by the following recursive relation:
\begin{align}
\al_k = \frac{\al_{k-1} + 3}{2}
\label{al1}
\end{align}

\noi
with $\al_1 = 1$.
Given $0< s < \al_{k-1}^{-1}$, 
 let $\phi^\o$ be the Wiener randomization 
of $\phi \in H^s(\R^3)$
 defined in \eqref{rand}.
Then,  we have 
\begin{align}
  \zeta_{2k -1} \in X^\s_\textup{loc} 
\label{QQQ3a}
\end{align}

\noi
for $\s < \al_k \cdot  s $, almost surely.

\end{proposition}

By solving the recursive relation \eqref{al1}, 
we have
\begin{align}
\al_k = 2\bigg\{ 1 - \bigg(\frac 12\bigg)^{k-1}\bigg\} + 1
\label{al2}
\end{align}

\noi
and thus we have 
 $\al_2 = 2$, $\al_3 = \frac 52$, and $\al_4 = \frac{11}{4}$.
In particular, Proposition \ref{PROP:Zk}
agrees with~\eqref{intro1}.
Moreover, since  $\al_k$ is increasing and $\lim_{k \to \infty} \al_k =3$, 
the regularity restriction $s < \al_{k-1}^{-1}$
in Proposition \ref{PROP:Zk}
does not cause any issue
since 
our main focus is to study the probabilistic local well-posedness
of \eqref{NLS1} in the range of $s$ that is not covered by 
Theorem~\ref{THM:1}.
Namely, 
we may assume $s \leq \frac 15\, (< \frac 13)$ in the following.

In view of Propositions \ref{PROP:Z3}
and \ref{PROP:Zk}, one obvious  lower bound for the
critical regularity~$s_*$ for this iterative procedure is given by $s_0: = 0$
since there is no gain of regularity when $s = 0$ (even in moving from $z_1$ to $z_3$).
On the other hand, 
in order to prove almost sure local well-posedness of \eqref{NLS1}
by carrying out the case-by-case analysis \eqref{casek} for the equation~\eqref{NLS8}, 
we need to show 
that the $(k+1)$st order term $\zeta_{2k+1}$ belongs to  $X^\frac{1}{2}([0, T])$.
This gives rise to a regularity restriction
$s_k := \frac{1}{2\al_{k+1}}$
at the $k$th step of the iterative procedure.
By taking $k \to \infty$, we obtain another ``lower'' bound\footnote{This ``lower'' bound 
is based on the upper bounds obtained in Propositions \ref{PROP:Z3}
and \ref{PROP:Zk}.
In other words, 
if one can improve the bounds
in Propositions \ref{PROP:Z3}
and \ref{PROP:Zk}, then one can lower the value of $s_\infty$.
See Remark~\ref{REM:multi}.}
$s_\infty := \frac 16$
on this critical regularity $s_*$.

As mentioned above, 
the case-by-case analysis \eqref{casek}
for general $k \in \NB$
may be a combinatorially overwhelming task
due to (i) the number of the increasing combinations in~\eqref{casek}
and (ii) the random multilinear  terms
$z_j,$ $j \in \{3, 5, \dots, 2k-1\}$, 
themselves having non-trivial combinatorial structures
which makes it difficult to 
establish nonlinear estimates;
see \eqref{introZk}.
In the following, we instead consider an alternative  iterative procedure
based on   the following expansion:
\begin{align}
u = \sum_{\l = 1}^k \zeta_{2\l-1} + v
\label{vk1}
\end{align}

\noi
in place of \eqref{vJ}.
This expansion allows us to prove  the following almost sure local well-posedness
of \eqref{NLS1} for $s > s_\infty=  \frac 16$.

\begin{theorem}\label{THM:2}
Let $\frac 16 < s < \frac 12$.
Given $\phi \in H^s(\R^3)$, let $\phi^\o$ be its Wiener randomization defined in \eqref{rand}.
Then, the cubic NLS \eqref{NLS1} on $\R^3$
is almost surely locally well-posed
with respect to the random initial data $\phi^\omega$.
More precisely,
there exists a set $\Si = \Si(\phi) \subset \O$ 
with $P(\Si) = 1$ such that, for any $\o \in \Si$, there exists
a unique function  $u = u^\o$ in the class:
\begin{align*}
\zeta_1  + \zeta_3 + \cdots+ \zeta_{2k-1} 
 + X^\frac{1}{2}([0, T]) 
&   \subset
\zeta_1  + \zeta_3 + \cdots+ \zeta_{2k-1} 
+ C([0, T]; H^\frac{1}{2} (\R^3)) 
\notag\\\
&  
\subset C([0, T];H^s(\R^3))
%\label{class3}
\end{align*}

\noi
with  $T = T(\phi, \o) >0$ 
such that $u$ is a solution to \eqref{NLS1} on $[0, T]$.
Here, 
 $k \in \NB $ is a unique positive integer such that $\frac{1}{2\al_{k+1}} < s \leq \frac{1}{2\al_{k}}$.

\end{theorem}

As before, 
the uniqueness of $u$ in the class:
\begin{align*}
\zeta_1  + \zeta_3 + \cdots+ \zeta_{2k-1} 
 + X^\frac{1}{2}([0, T]) 
\end{align*}

\noi
is to be interpreted as
uniqueness of the residual term $v = u -\sum_{\l = 1}^k \zeta_{2\l-1} $
in $X^\frac{1}{2}([0, T])$.
See also Remark \ref{REM:uniq1}.

In view of the discussion above, 
Theorem \ref{THM:2} proves
almost sure local well-posedness of \eqref{NLS1}
 in an almost ``optimal''\footnote{Once again, this is based
 on the estimates in Propositions \ref{PROP:Z3}
and \ref{PROP:Zk}. 
In particular, the ``optimality''
of the regularity threshold in Theorem \ref{THM:2}
is with respect to 
 Propositions \ref{PROP:Z3}
and \ref{PROP:Zk}.
 If one can improve the bounds
in Propositions \ref{PROP:Z3}
and \ref{PROP:Zk}, then one can lower 
the regularity threshold in Theorem \ref{THM:2}.}
regularity range $s > s_\infty = \frac 16$
with respect to the original iterative procedure 
based on the partial power series expansion \eqref{vJ}.
The proof of Theorem \ref{THM:2}
is analogous
to that of Theorem \ref{THM:1}.
Given $ \frac 16 <  s<\frac 12$, 
fix $k \in \NB$ such that $\frac{1}{2\al_{k+1}} < s \leq \frac{1}{2\al_{k}}$.\footnote{In view of Proposition \ref{PROP:Zk}, 
the lower bound on $s$
guarantees that $\zeta_{2k+1} \in X^\frac{1}{2}([0, T])$ almost surely, 
while the upper  bound on $s$
states that we can not use 
 Proposition \ref{PROP:Zk} to conclude 
$\zeta_{2k-1} \in X^\frac{1}{2}([0, T])$ almost surely.}
Write a solution $u$ as in \eqref{vk1}.
Note that as $s$ gets closer  and closer to the critical value $s_\infty = \frac{1}{6}$, 
the expansion~\eqref{vk1} gets  arbitrarily long. 
In view of \eqref{zeta2},  the residual term
$v := u - \sum_{\l = 1}^{k} \zeta_{2\l-1}$ satisfies the following equation:
\begin{equation}
\begin{cases}
\displaystyle
 i \dt v + \Dl v =  \N\bigg(v +\sum_{\l= 1}^{k} \zeta_{2\l-1}\bigg) 
-  
\sum_{\l = 2}^{k}
\sum_{\substack{j_1 + j_2 + j_3 = 2\l - 1\\j_1, j_2, j_3 \in \{1,  2\l-3\}}} 
 \zeta_{j_1} \cj{\zeta_{j_2}}\zeta_{j_3}\\
v|_{t = 0} = 0.
 \end{cases}
\label{NLS9}
\end{equation}

\noi
Hence, we need to carry out 
 the following case-by-case analysis:
\begin{align}
\begin{split}
w_1\cj{w_2} w_3, 
\quad & \text{for  $w_i = v, \zeta_j,$ $j \in \{1, 3, \dots, 2k-1\}$, 
 $i = 1, 2, 3$, such that }  \\
&  \text{it is not of the form $\zeta_{j_1}\cj{\zeta_{j_2}}\zeta_{j_3} $ with 
$(j_1,  j_2,  j_3 )   
=  (1, 1,  2\l-3)$}\\
& \text{(up to permutations) for $\l = 2, 3 \dots, k$} 
\end{split}
\label{casek2}
\end{align}

\noi 
in $N^\frac{1}{2}([0, T])$.
While Theorem \ref{THM:2} yields
almost sure local well-posedness
in an almost optimal range 
with respect to the original iterative procedure 
based on the partial power series expansion~\eqref{vJ}, 
the required analysis is much simpler
than that required 
for the original iterative procedure based on the expansion \eqref{vJ}.
First,  note that while the case-by-case analysis~\eqref{casek}
based on the original iterative procedure
involves combinatorially non-trivial $z_{2k-1}$, 
 the case-by-case analysis~\eqref{casek2}
only involves 
the  unbalanced $k$th order term $\zeta_{2k-1}$
which has a much simpler structure than $z_{2k-1}$.
In particular, Proposition~\ref{PROP:Zk} shows that 
$\zeta_{2k-1}$, $k \geq 3$,  
has a better regularity property than the second order term $\zeta_3 = z_3$
in \eqref{introZ3}.
In terms of space-time integrability, 
we show that $\zeta_{2k-1}$ also 
 enjoys a gain of integrability by giving up a control on derivatives 
 (Lemma \ref{LEM:Zk_2}).
 Finally, by inspecting the proof of Theorem~\ref{THM:1}
 (see Lemma \ref{LEM:Z7} and Proposition \ref{PROP:NL1} below), 
 we see that,
 except for  $z_{j_1} \cj{z_{j_2}}z_{j_3}$ with $(j_1, j_2, j_3) = (1, 1, 3)$ up to permutations,\footnote{Namely, the terms
  constituting the third order term $\zeta_5 = z_5$
in~\eqref{introZ5}.} 
we can bound 
all the terms $w_1 \cj{w_2} w_3$ 
appearing  in the case-by-case analysis~\eqref{case1}
 in $N^\frac{1}{2}([0, T])$. 
 Hence, we can basically 
 apply the result of the case-by-case analysis~\eqref{case1} to our problem at hand.
More precisely, by rewriting the case-by-case analysis \eqref{casek2}
as 
\[ w_i = v,  \zeta_1 = z_1 , \text{ or }\zeta_j, \, j \in \{3, 5, \dots, 2k-1\}\]

\noi
(with the restriction in \eqref{casek2})
and using the fact that the terms $\zeta_{2\l - 1}$, $\l = 3, \dots, k$,  behave better than $\zeta_3 = z_3$, 
the proof of Theorem \ref{THM:1} (in particular, Proposition \ref{PROP:NL1})
can be used to control  all the terms $w_1 \cj{w_2} w_3$ in \eqref{casek2} {\it except for}
\begin{align}
 \zeta_{j_1}\cj{\zeta_{j_2}}\zeta_{j_3} 
\quad \text{with }
(j_1,  j_2,  j_3 )   
=  (1, 1,  2 k-1)
\label{casek2a}
\end{align}

\noi
 (up to permutations).
Note that the contribution to \eqref{casek2a} under the Duhamel integral
is precisely given by $\zeta_{2k+1}$ in \eqref{zeta}.
In particular, Proposition \ref{PROP:Zk} with $\s = \frac 12$ yields
the regularity restriction $\al_{k+1}\cdot s > \frac 12$, 
i.e.~the lower bound 
$s > \frac{1}{2\al_{k+1}}$ stated in Theorem \ref{THM:2}.
 This allows us to 
  construct a solution $v \in X^\frac{1}{2}([0, T])$ 
  by the standard fixed point argument.
See Section \ref{SEC:THM2} for details.

We previously conjectured that the $(k+1)$st order term $z_{2k+1}$ in \eqref{introZk}
would be responsible for a regularity 
restriction at the $k$th step of the original iterative procedure.
Combining Proposition \ref{PROP:Zk}
and Theorem \ref{THM:2}, 
we confirmed this claim;
 the regularity restriction
indeed comes only from the unbalanced
$(k+1)$st order term $\zeta_{2k+1}$ in \eqref{zeta}.

We conclude this introduction with several remarks.

\begin{remark}\rm
Let $v_0 \in H^\frac{1}{2}(\R^3)$.
Then, by slightly modifying the proof of Theorem \ref{THM:2}
based on the modified iterative approach \eqref{vk1},
we can prove almost sure local well-posedness
of  the cubic NLS 
\eqref{NLS5a}
with the random initial data of the form $v_0 + \phi^\o$
for the same range of $s$.
See Remark \ref{REM:v_0}

\end{remark}

\begin{remark}\label{REM:multi}\rm
In this paper, 
we exploit randomness only at the linear level
in estimating the stochastic terms.
See Remarks \ref{REM:Z3_1} and \ref{REM:Z3_2}.
It may be possible to lower the regularity thresholds
in Theorems \ref{THM:1} and \ref{THM:2} 
by exploiting randomness at the multilinear level.
We, however, do not pursue this direction in this paper since
(i) our main purpose is to present the iterative procedures in their simplest forms
and 
(ii) estimating the higher order stochastic terms by exploiting randomness
at the multilinear level
would require a significant amount of additional work,
which would blur the main focus of this paper.

\end{remark}

\begin{remark}\rm

The ill-posedness result in \cite{CCT} show that 
the solution map 
\[\Phi: u_0 \in H^s(\R^3) \longmapsto u \in C([-T, T]; H^s(\R^3))\]

\noi
is not continuous for \eqref{NLS1}
when $s < \frac{1}{2}$.
In proving Theorem \ref{THM:A}, we studied 
the perturbed NLS~\eqref{NLS2} for $v = u-z_1$.
In particular, the proof shows that 
we can  factorize the solution map for~\eqref{NLS1} as\footnote{Similarly, 
we can factorize the solution map for \eqref{NLS5a} 
as 
\begin{align*}
u |_{t = 0} =  v_0 + \phi^\o \in H^s(\R^3) 
\longmapsto  (v_0, z_1)   \stackrel{\Psi}{\longmapsto} v \in
X^\frac{1}{2}([0, T]) \subset  C([0, T]; H^\frac{1}{2}(\R^3))
\end{align*}

\noi
such that the second map $\Psi$ is continuous
in  $(v_0, z_1) \in H^\frac{1}{2}(\R^3)\times S^s_1([0, T])$.
}

\begin{align}
\phi^\o \in H^s(\R^3) 
\longmapsto  z_1   \stackrel{\Psi}{\longmapsto} v \in 
X^\frac{1}{2}([0, T]) \subset C([0, T]; H^\frac{1}{2}(\R^3)),
\label{dia1}
\end{align}

\noi
where the first map can be viewed as a universal lift map
and the second map $\Psi$ is the solution map 
to \eqref{NLS2},
which is in fact continuous in $z_1 \in S^s_1([0, T])$.\footnote{Here,  $S^s_1([0, T]) \subset C([0, T]; H^s(\R^3))$ denotes the intersection of 
suitable space-time function spaces of differentiability at most $s$.
In the following, we use $S^{\s}_j([0, T])$ in a similar manner.}

In the case of Theorem \ref{THM:1} (with $\s = \frac 12$), we have the following factorization 
of the solution map for~\eqref{NLS1}:
\begin{align} \phi^\o \in H^s(\R^3)
\longmapsto (z_1, z_3 ) \stackrel{\Psi}{\longmapsto} v \in 
X^\frac{1}{2}([0, T]) \subset 
C([0, T]; H^\frac{1}{2}(\R^3)),
\label{dia2}
\end{align}

\noi
where  the second map $\Psi$ is the solution map 
to \eqref{NLS4},
which is  continuous from $(z_1, z_3)  \in S^s_1([0, T])\times S^{2s-}_3([0, T])$
to $v \in X^\frac{1}{2}([0, T])$. % \subset C([0, T]; H^\frac{1}{2}(\R^3))$.

In the case of Theorem \ref{THM:2}, we create $k$ stochastic objects in the first step,
where $k = k(s) \in \NB$:
\begin{align} 
\phi^\o \in H^s(\R^3)
\longmapsto (\zeta_1, \zeta_3, \dots, \zeta_{2k-1} ) \stackrel{\Psi}{\longmapsto} v \in 
X^\frac{1}{2}([0, T]) \subset 
C([0, T]; H^\frac{1}{2}(\R^3)).
\label{dia3}
\end{align}

\noi
Once again, the  second map $\Psi$ (which is the solution map 
to~\eqref{NLS9})
is  continuous from 
$(\zeta_1, \zeta_3, \dots, \zeta_{2k-1} )  \in S^s_1([0, T]) \times 
\prod_{\l = 2}^ k S_\l^{\al_\l s-} ([0, T])$
to $v \in 
X^\frac{1}{2}([0, T])$.  
%C([0, T]; H^\frac{1}{2}(\R^3))$.
We point out  an analogy between these factorizations \eqref{dia1}, \eqref{dia2}, and \eqref{dia3} of the ill-posed solution maps
into 
\begin{itemize}
\item[(i)] the first step, involving stochastic analysis
and 

\item [(ii)] the second step, where purely deterministic analysis is performed
in constructing a continuous map $\Psi$

\end{itemize}

\noi
and similar factorizations  for studying rough differential equations via the rough path theory~\cite{FH}
and 
singular stochastic parabolic PDEs~\cite{GP, Hairer}.

In the proof of Theorem \ref{THM:2}, 
we could  consider the following expansion of an infinite order:
\begin{align}
u = \sum_{\l = 1}^\infty \zeta_{2\l-1} + v.
\label{vk2}
\end{align}

\noi
This would allow us to present a single argument that works
for all $\frac16 < s < \frac 13$.
In this case, the residual part 
$v := u - \sum_{\l = 1}^\infty \zeta_{2\l-1}$ satisfies 
\begin{equation*}
\begin{cases}
\displaystyle
 i \dt v + \Dl v =  \N\bigg(v +\sum_{\l= 1}^{\infty} \zeta_{2\l-1}\bigg) 
-  
\sum_{\l = 2}^{\infty}
\sum_{\substack{j_1 + j_2 + j_3 = 2\l - 1\\j_1, j_2, j_3 \in \{1,  2\l-3\}}} 
 \zeta_{j_1} \cj{\zeta_{j_2}}\zeta_{j_3}\\
v|_{t = 0} = 0.
 \end{cases}
%\label{NLS10}
\end{equation*}

\noi
In particular, we would need to worry about the convergence issue
of infinite series
and hence there seems to be no simplification in considering the infinite order expansion \eqref{vk2}.

Another strategy would be to treat $\zeta_\infty: =  \sum_{\l = 1}^\infty \zeta_{2\l-1} $ as one 
stochastic object
and write 
\begin{align*}
u =\zeta_\infty + v.
%\label{vk3}
\end{align*}

\noi
 It follows from \eqref{zeta2} that 
$\zeta_\infty$ satisfies the following equation:
\begin{equation}
\begin{cases}
\displaystyle
 i \dt \zeta_\infty + \Dl \zeta_\infty = 
\sum_{\l = 2}^{\infty}
\sum_{\substack{j_1 + j_2 + j_3 = 2\l - 1\\j_1, j_2, j_3 \in \{1,  2\l-3\}}} 
 \zeta_{j_1} \cj{\zeta_{j_2}}\zeta_{j_3}\\
\zeta_\infty|_{t = 0} = \phi^\o.
 \end{cases}
\label{NLS11}
\end{equation}

\noi
Noting that 
\[\sum_{\substack{j_1 + j_2 + j_3 = 2\l - 1\\j_1, j_2, j_3 \in \{1,  2\l-3\}}} 
 \zeta_{j_1} \cj{\zeta_{j_2}}\zeta_{j_3}
= \begin{cases}
|z_1|^2 z_1& \text{when }\l = 2,\\
2|z_1|^2\zeta_{2\l-3} + z_1^2 \cj{\zeta_{2\l-3}},
& \text{when }\l \geq 3,
\end{cases}
\]	

\noi
we can rewrite \eqref{NLS11} as 
\begin{equation}
\begin{cases}
\displaystyle
 i \dt \zeta_\infty + \Dl \zeta_\infty = 
 2|z_1|^2 \zeta_\infty + z_1^2 \cj{\zeta_\infty}
 - 2 |z_1|^2 z_1\\
\zeta_\infty|_{t = 0} = \phi^\o.
 \end{cases}
\label{NLS12}
\end{equation}

\noi
and thus $v : = u - \zeta_\infty$ satisfies
\begin{equation}
\begin{cases}
\displaystyle
 i \dt v + \Dl v =  \N(v +\zeta_\infty)
-   2|z_1|^2 \zeta_\infty - z_1^2 \cj{\zeta_\infty}
 + 2 |z_1|^2 z_1\\
v|_{t = 0} = 0.
 \end{cases}
\label{NLS13}
\end{equation}

\noi
The equations \eqref{NLS12} and \eqref{NLS13}
do not particularly appear to be in such a friendly format.
Namely,  studying 
\eqref{NLS12} and \eqref{NLS13}
does not seem to provide a simplification
over the case-by-case analysis \eqref{casek2}
for the equation 
\eqref{NLS9} for each fixed $k \in \NB$.

In a recent work \cite{OTW}, the second author (with Tzvetkov and Y.\,Wang)
proved invariance of the white noise for the (renormalized) cubic fourth order NLS on the circle.
One novelty of this work is
that we introduced 
 an infinite sequence  $\{ z^\text{res}_{2\l-1}\}_{\l \in \NB}$
of stochastic $(2\l-1)$-linear objects  (depending only on the random initial data)
and considered the following expansion of an infinite order:
\[ u =  \sum_{\l = 1}^\infty z^\text{res}_{2\l-1} + v.\]

\noi
For this problem, 
it turned out that 
 $z^\text{res}_\infty: = \sum_{\l = 1}^\infty z^\text{res}_{2\l-1}$
satisfies a particularly simple equation.\footnote{In fact, the series
 $z^\text{res}_\infty = \sum_{\l = 1}^\infty z^\text{res}_{2\l-1}$
corresponds to 
the power series expansion of the resonant  cubic fourth order NLS.
We point out that  $z^\text{res}_\infty$ does not belong to 
the span of Wiener
homogeneous chaoses of any finite order.}
Then by treating $z^\text{res}_\infty$ as one stochastic object,
we wrote a solution $u$ as $u = z^\text{res}_\infty +v$,
which led to 
the  following factorization:
\[\phi^\o \in H^s(\T) 
\longmapsto 
z^\text{res}_\infty = \sum_{\l = 1}^\infty z^\text{res}_{2\l-1}
 \longmapsto v \in C(\R; L^2(\T))\]

\noi
for $s < -\frac 12$,
where $\phi^\o$ denotes the Gaussian white noise on the circle.

\end{remark}

\begin{remark}	\rm 
In \cite{PW}, 
the third author (with Y.\,Wang) 
recently studied probabilistic local well-posedness of NLS on $\R^d$ 
within the framework of  the $L^p$-based Sobolev spaces,
using the dispersive estimate.
In the context of the cubic NLS \eqref{NLS1} on $\R^3$, 
their result 
yields  almost sure  local existence of 
a unique
solution $u$
for the randomized initial data $\phi^\o \in H^s(\R^3)$,
provided $s\geq 0$. 
In particular, the argument in \cite{PW}
allows us 
 to consider 
random initial data of lower regularities than
Theorems~\ref{THM:A},~\ref{THM:1}, and~\ref{THM:2}.
Note that, in \cite{PW}, 
it was shown that  
the solution $u$  only belongs to 
$C([0, T]; W^{s, 4}(\R^3))$, almost surely.
We point out that  a slight adaptation of the work \cite{OPW} by the second and third authors (with Y.\,Wang) 
shows that 
the solution $u$ indeed lies in $C([0, T]; H^s(\R^3))$.

As compared to \cite{PW}, 
 the argument presented in this paper
 provides extra regularity information, namely
 the decomposition~\eqref{vk1}  of the solution $u$
with the terms of increasing regularities:
$  \zeta_{2\l -1} \in X^{\al_\l\cdot s-}_\textup{loc} $
and $v \in X^\frac{1}{2}([0, T])$.
In particular, the residual term $v$ lies in the (sub)critical regularity,\footnote{By slightly tweaking the argument, 
we can easily place $v$ in $X^{\frac{1}{2}+}([0, T])$.}
leaving us a possibility of adapting deterministic techniques to study 
its further properties.

\end{remark}

This paper is organized as follows.
In Section \ref{SEC:2}, 
we recall probabilistic and
deterministic lemmas along with the definitions of the basic  function spaces.
In Section \ref{SEC:Z3}, 
we study the regularity properties of the second order term $z_3$ in \eqref{introZ3}.
In Section \ref{SEC:high}, 
we further investigate the regularity properties
of the higher order terms $z_5$ and $z_7$
and the unbalanced higher order terms $\zeta_{2k-1}$.
Note that the analysis on $z_5$ and $z_7$ contains
part\footnote{Note that $z_7$ defined in \eqref{introZ7} contains 
the contribution from $z_{j_1}\cj{z_{j_2}}z_{j_3}$
with $(j_1, j_2, j_3) = (1, 3, 3)$ up to permutations.} of the case-by-case analysis \eqref{case1}
needed for proving Theorem \ref{THM:1}.
In Section \ref{SEC:NL1},
we then carry out the rest of 
the case-by-case analysis \eqref{case1}
and prove Theorem \ref{THM:1}.
In Section \ref{SEC:THM2}, 
we briefly describe   the proof of Theorem \ref{THM:2}
by indicating how the analysis in the previous sections
can lead to the proof. 
In Appendix \ref{APP:A}, we 
prove the deterministic non-smoothing
of the Duhamel integral operator
discussed  in 
Remark \ref{REM:Z3_1}.

\medskip

\noi
{\bf Notations:}
We  use $a+$ (and $a-$) to denote $a + \eps$ (and $a - \eps$, respectively) for arbitrarily small $\eps \ll 1$,
where an implicit constant is allowed to depend on $\eps > 0$
(and it usually diverges as $\eps \to 0$).

Given a Banach space $B$ of temporal functions, 
we use the following short-hand notation: $B_T := B([0, T])$.
For example, $L^q_T L^r_x = L^q_t([0, T]; L^r_x(\R^3))$.

Let $\eta: \R \to [0, 1]$ be an even, smooth cutoff function
supported on $[-\frac{8}{5}, \frac{8}{5}]$
such that $\eta \equiv 1$ on $[-\frac{5}{4}, \frac{5}{4}]$.
Given a dyadic number $N \geq 1$, we
set
$\eta_1(\xi) = \eta(|\xi|)$
and
\[\eta_N(\xi) = \eta\bigg(\frac{|\xi|}{N}\bigg) - \eta\bigg(\frac{2|\xi|}{N}\bigg)\]

\noi
for $N \geq 2$.
Then, we define the Littlewood-Paley projection operator
$\P_N$ as the Fourier multiplier operator with symbol $\eta_N$.
In the following, we use the convention that capital letters denote
dyadic numbers.
For example, $N  = 2^n$ for some $n \in \NB_0 : = \NB \cup\{0\}$.

Given dyadic numbers $N_1, \dots, N_4 \in 2^{\NB_0}$, 
we set $N_{\max} :=  \max_{j = 1, \dots, 4} N_j$.
We also use the following shorthand notation:
$f_N = \P_N f$.
For example, we have 
$z_{1, N_j} = \P_{N_j} z_1$.

\section{Strichartz estimates and function spaces}
\label{SEC:2}

\subsection{Probabilistic Strichartz estimates}
First, we recall the usual Strichartz estimates on $\R^3$ for readers' convenience.
We say that
a pair  $(q, r)$ is Schr\"odinger admissible
if it satisfies
\begin{equation*}
\frac{2}{q} + \frac{3}{r} = \frac{3}{2}
%\label{Admin}
\end{equation*}

\noi
 with $2\leq q, r \leq \infty$.
Then, the following Strichartz estimates
are known to hold \cite{Strichartz, Yajima, GV, KeelTao}:
\begin{equation}
\| S(t) \phi\|_{L^q_t L^r_x (\R\times \R^3)} \lesssim \|\phi\|_{L^2_x(\R^3)}.
\label{Str0}
\end{equation}

\noi
It follows from \eqref{Str0} and Sobolev's inequality that 
\begin{equation}
\| S(t) \phi\|_{L^p_{t, x} (\R\times \R^3)}
\lesssim \big\||\nb|^{\frac 32 - \frac{5}{p}}\phi\big\|_{L^2_x(\R^3)}
\label{Str0a}
\end{equation}

\noi
for $p \geq \frac{10}{3}$.
We will use  the following admissible pairs in this paper:
\begin{align*}
(\infty, 2), \ %\big(10, \tfrac{30}{13}\big), \ 
\big(5, \tfrac{30}{11}\big), \ 
\big(\tfrac{10}{3}\tfrac{10}{3}\big),
\ (2+, 6-).
\end{align*}

\noi
In particular, by Sobolev's inequality, we have
\begin{align}
W^{\frac 12, \frac{30}{11}}(\R^3) \hookrightarrow L^5(\R^3).
\label{Sobolev}
\end{align}

One of the important key ingredients
for probabilistic well-posedness is
the probabilistic Strichartz estimates.
Such probabilistic estimates were first exploited by 
McKean \cite{McKean} and Bourgain \cite{BO96}.
In the following, we state
the probabilistic Strichartz estimates under the Wiener randomization \eqref{rand}.
See \cite{BOP1} for the proofs.

\begin{lemma}\label{LEM:PStr}
Given $\phi$ on $\R^3$, 
let $\phi^\o$ be its Wiener randomization defined in \eqref{rand}.
Then,
given finite  $q \geq 2$ and $2 \leq  r \leq \infty$,
there exist $C, c>0$ such that
\begin{align*}
P\Big( \|S(t) \phi^\omega\|_{L^q_t L^r_x([0, T]\times \R^3)}> \ld\Big)
\leq C\exp \bigg(-c \frac{\ld^2}{ T^\frac{2}{q}\|\phi\|_{H^s}^{2}}\bigg)
%\label{PStr1}
\end{align*}
	
\noi
for all  $ T > 0$ and $\lambda>0$
with  \textup{(i)} $s = 0$ if $r < \infty$
and  \textup{(ii)}  $s > 0$ if $r = \infty$.

\end{lemma}

A similar estimate holds when $q = \infty$ (with $s > 0$)
but we will not need it in this paper.  See \cite{OP}.
We also need the following lemma on the control of the size of $H^s$-norm of $\phi^\o$.

\begin{lemma} \label{LEM:Hs}
Given  $\phi \in H^s(\R^3)$, let $\phi^\o$ be its Wiener randomization
defined in \eqref{rand}.
Then, we have
\begin{align*}
P\Big( \| \phi^\omega \|_{ H^s(  \R^3)} > \ld\Big)
\leq  
C\exp \bigg(-c \frac{\ld^2}{ \|\phi\|_{H^s}^{2}}\bigg)
\end{align*}

\noi
for all   $\lambda>0$.

\end{lemma}

\subsection{Function spaces and their properties}
\label{SUBSEC:Up}

In this subsection, we go over the basic definitions 
and properties of the functions spaces
used for the Fourier restriction norm method 
(i.e.~analysis involving the $X^{s, b}$-spaces introduced in \cite{Bo2})
adapted
to the space  of  functions of bounded $p$-variation and its pre-dual,
introduced and developed by 
Tataru, Koch, and their collaborators \cite{KochT, HHK, HTT}.
We refer readers to   \cite{HHK, HTT}
for the proofs of the basic properties. See also~\cite{BOP2}.

Let $\mathcal{Z}$ be the set of finite partitions $-\infty <t_{0}<t_{1}<\dots <t_{K} \le \infty $ of the real line.
By convention, we set $u(t_K):=0$  if $t_K=\infty$.
We use $\ind_I$ to denote the sharp characteristic function
of a set $I \subset \R$.

 \begin{definition} \label{DEF:X1}\rm
Let $1\leq p < \infty$.

\smallskip
\noi
\textup{(i)}
We define a $U^p$-atom to be 
a step function $a: \R \to L^2(\R^3)$
of the form: 
\begin{equation*}
a=\sum_{k=1}^{K}\phi _{k-1} \ind _{[t_{k-1},t_{k})}, 
\end{equation*}

\noi
where
  $\{t_{k}\}_{k=0}^{K}\in \mathcal{Z}$ and $\{\phi _{k}\}_{k=0}^{K-1}\subset L^2 (\R^3)$ with 
$\sum_{k=0}^{K-1} \| \phi _{k} \|_{L^{2}}^{p}=1$.
Furthermore, we define the atomic space $U^p = U^p(\R; L^2(\R^3))$
by 
\begin{align*}
U^{p}:= \bigg\{ u: \R \rightarrow L^2 (\R^3) : u=\sum_{j=1}^{\infty }\lambda _{j}a_{j}\ & \text{for $U^{p}$-atoms $a_{j}$},
\ \{\ld_j\}_{j \in \NB} \in \l^1(\NB; \C)\bigg\}
\end{align*}

\noi
with the norm 
\begin{align*}
\| u \| _{U^{p}}:= \inf \bigg\{ \sum_{j=1}^{\infty }|\lambda _{j}|:u=\sum_{j=1}^{\infty } \lambda _{j}a_{j}
 \text{ for $U^{p}$-atoms } a_{j},  \
 \{\ld_j\}_{j \in \NB} \in \l^1(\NB; \C)\bigg\},
\end{align*}

\noi
where the infimum is taken over all possible representations for $u$.

\smallskip

\noi
(ii)  We define $V^{p}= V^p(\R; L^2(\R^3))$ to be 
the space of functions $u:\R\to L^2(\R^3)$ of bounded $p$-variation
with the standard $p$-variation norm:
\begin{equation*}
\| u \| _{V^{p}}:=\sup_{\{t_{k}\}_{k=0}^{K}\in \mathcal{Z}}
\bigg( \sum_{k=1}^{K}\| u(t_{k})-u(t_{k-1})\| _{L^{2}}^{p}\bigg) ^{\frac 1p}.
\end{equation*}

\noi
By convention, we impose that  the limits $\lim _{t \to \pm \infty}u(t)$ exist in $L^2 (\R^3)$.

\smallskip

\noi
(iii) Let $V_{\text{rc}}^{p}$ be the closed subspace of $V^{p}$ 
of all right-continuous functions $u \in V^p$ with  $\lim_{t\rightarrow -\infty }u(t)=0$.

\smallskip

\noi
(iv)
We define  $U_{\Delta}^{p}:=S (t) U^{p}$ (and $V_{\Delta}^{p}:=S(t)V^{p}$, respectively)
to be the space of all functions $u:\R \to L^2 (\R^3)$ such that 
$t \to S ( -t ) u(t)$ is in $U^p$ (and in $V^p$, respectively)
with the norms 
\begin{equation*}
\| u \|_{U_{\Delta}^{p}} :=\| S(-t )u \|_{U^{p}}
\qquad \text{and} \qquad
\| u \|_{V_{\Delta}^p} :=\| S(-t)u \|_{V^{p}}.
\end{equation*}

\noi
 The closed subspace $V_{\text{rc}, \Delta}^{p}$ is defined in an analogous manner.

 \end{definition}

Recall the following inclusion relation; 
for $1\leq p<q<\infty $, 
\begin{align*}
U^p \hookrightarrow V_{\text{rc}}^{p}\hookrightarrow U^{q} \hookrightarrow L^{\infty }(\R ;L^2 (\R^3)).
%\label{incl}
\end{align*}

\noi
The space $V^p$ is the classical space of functions of bounded $p$-variation
and the space $U^p$ appears as the pre-dual of $V^{p'}$ with $\frac 1p + \frac 1{p'} = 1$,
$1 < p < \infty$.
Their duality relation and the atomic structure of the $U^p$-space
turned out to be very effective in studying dispersive PDEs in critical settings.

We are now ready to define the solution spaces.

\begin{definition} \label{DEF:X3}
\rm
\textup{(i)}
Let $s\in \R$.
We define $X^s(\R)$ to be 
 the closure of 
 $C(\R ;H^s(\R ^3)) \cap U_{\Delta}^2 $ 
with respect to the $X^s$-norm defined by
\[ \|u \|_{X^s (\R)} : = \bigg(
\sum_{\substack{N\geq 1\\\textup{dyadic}}} N^{2s}
\| \P_N u \|_{U^2_{\Dl}L^2}^2
\bigg)^\frac{1}{2}.
\]

\smallskip
\noi
\textup{(ii)}
Let $s\in \R$.
We define $Y^s(\R)$ to be 
 the space of all functions $u \in C(\R ;H^s(\R ^3))$
such that the map $t \mapsto \P_N u$ lies in $ V_{\text{rc},\Delta}^2 H^s$
for any $N \in 2^{\NB_0}$
and $\| u \|_{Y^s(\R)}< \infty$, 
where  
  the $Y^s$-norm  is defined by
\[ \|u \|_{Y^s(\R)} : = \bigg( \sum_{\substack{N\geq 1\\\textup{dyadic}}} N^{2s}
\| \P_N u\|_{V^2_\Dl L^2}^2\bigg)^\frac{1}{2}.
\]

\end{definition}

Recall the following embeddings:
\begin{equation*}%\label{inclusions}
U^2_\Dl H^s \hookrightarrow X^s \hookrightarrow
Y^s \hookrightarrow V^2_\Dl H^s \hookrightarrow U^p_\Dl H^s,
\end{equation*}
for $p>2$.

Given an  interval $I \subset \R$,
we define the local-in-time versions $X^s(I)$
and $Y^s(I)$ of these spaces
as restriction norms.
For example, we define the $X^s(I)$-norm by
\[ \|u \|_{X^s(I)} = \inf\big\{ \|v\|_{X^s(\R)}: \, v|_I = u\big\}.\]

\noi
We also define the norm for the nonhomogeneous term
on an interval $I = [t_0, t_1)$:
\begin{align*}
\| F\|_{N^s(I)} = \bigg\|\int_{t_0}^t S(t - t') F(t') dt'\bigg\|_{X^s(I)}.
\end{align*}

We conclude this section by presenting
some basic
estimates involving these function spaces.	
See \cite{HHK, HTT, BOP2} for the proofs.

\begin{lemma} \label{LEM:lin}
Let $s \in \R$ and $T\in ( 0,\infty ]$.
Then,  the following linear estimates hold: 
\begin{align*}
\|   S(t) \phi \|_{X^s([0, T])}
& \leq \|\phi\|_{H^s}, \\
\|  F \|_{N^{s}([0, T])}
& \le \sup_{\substack{w \in Y^{-s}([0, T])\\ \| w \|_{Y^{-s}([0, T])}=1}} \left| \int_{0}^{T} \lr{F(t), w(t)}_{L_{x}^{2}} dt \right|
\end{align*}

\noi
 for any $\phi \in H^s(\R^3)$
 and $F \in L^1([0,T];H^{s}(\R^3))$.
\end{lemma}

The transference principle \cite[Proposition 2.19]{HHK}
and the interpolation lemma 
\cite[Proposition 2.20]{HHK}
applied on the Strichartz estimates \eqref{Str0} and \eqref{Str0a}
imply the following estimates.

\begin{lemma} \label{LEM:Str}
Given any  admissible pair  $(q,r)$  with $q>2$
and $p \geq \frac{10}{3}$, 
 we have
\begin{align*}
\| u \|_{L_t^q L_x^r} & \lesssim \| u \|_{Y^0}, \\
\|  u \|_{L^p_{t,x}}
&  \les  \big\||\nb|^{\frac 32 - \frac {5}p} u\big\|_{Y^0}.
\end{align*}
\end{lemma}

Similarly, the bilinear refinement of the Strichartz estimate \cite{Bo98, OzawaT, CKSTT}
implies the following bilinear estimate.

\begin{lemma} \label{LEM:biStr}
Let $N_1, N_2 \in 2^{\NB _0}$ with $N_1 \le N_2$.
Then,   we  have 
\begin{align}
\| \P_{N_1}u_1 \P_{N_2}u_2\|_{L^2_{t} ([0, T]; L^2_x)}
\les T^{0+}N_1^{1-} N_2^{-\frac{1}{2}+}
\| \P_{N_1} u_1 \|_{Y^0([0, T])} \| \P_{N_2} u_2 \|_{Y^0([0, T])} 
\label{biStr2}
\end{align}

\noi
for any $T > 0$ and  $u_1, u_2 \in Y^0([0, T])$.

\end{lemma}

\begin{proof}
From  the bilinear refinement of the Strichartz estimate \cite{Bo98, CKSTT}
and the transference principle, we have
\begin{align}
\| \P_{N_1}u_1 \P_{N_2}u_2\|_{L^2_{t,x}} 
\les N_1^{1-} N_2^{-\frac{1}{2}+} \| \P_{N_1} u_1 \|_{Y^0} \| \P_{N_2} u_2 \|_{Y^0} 
\label{biStr1}
\end{align}

\noi
for all $u_1, u_2 \in Y^0$.
See \cite{BOP2} for the proof of  \eqref{biStr1}.
On the other hand, it follows from H\"older's  and Sobolev's inequalities that
\begin{align}
\| \P_{N_1}u_1 \P_{N_2}u_2\|_{L^2_TL^2_x}
& \les T^\frac{1}{2} \| \P_{N_1} u_1 \|_{L^\infty_TL^4_x} \| \P_{N_2} u_2 \|_{L^\infty_TL^4_x} 
\notag \\
& \les T^\frac{1}{2} N_1^\frac{3}{4}N_2^\frac{3}{4} \| \P_{N_1} u_1 \|_{Y^0_T} \| \P_{N_2} u_2 \|_{Y^0_T} .
\label{biStr3}
\end{align}

\noi
Then, the  estimate \eqref{biStr2} follows
from interpolating \eqref{biStr1} and \eqref{biStr3}.
\end{proof}

\section{On the second order term $z_3$}\label{SEC:Z3}

In this and the next sections, 
we study the regularity properties
of the various stochastic terms that appear
in the iterative procedures.
Given  $\phi \in H^s(\R^3)$, let 
 $\phi^\o$ be 
 the Wiener randomization of $\phi$ defined in \eqref{rand}
and   set 
 \begin{align*}
z_1 = S(t) \phi^\o.
% \label{O-2}
 \end{align*}

\noi
In this section,  we study the regularity properties of  the second order term:
\begin{align}
  z_3  = -i  \int_0^t S(t - t') |z_1|^2 z_1(t') dt',
\label{O-1}
\end{align}

\noi
We first present the proof of  Proposition \ref{PROP:Z3}.
We follow closely the argument in \cite{BOP2}.

\begin{proof}[Proof of  Proposition \ref{PROP:Z3}]

(i) 		
By Lemma \ref{LEM:lin}, 
the estimate \eqref{Z3_2} follows once we prove
\begin{align}
\bigg| \int_0^T \int_{\R^3}
\jb{\nb}^\s ( z_1 \cj{z_1} z_1 )\cj w  dx dt\bigg|
\leq T^\theta C(\o,  \|\phi\|_{H^s}) 
\label{O1}
\end{align}

\noi
for some almost surely finite constant $ C(\o, \|\phi\|_{H^s} ) > 0$ and $\theta > 0$,  
where $\|w\|_{Y^{0}_T} \leq 1$.
In the following, we drop the complex conjugate when it does not play any role.

Define $A^s_3(T)$ by 
\begin{align} 
\|z_1\|_{A^s_3(T)}: = \max\Big(
\|\jb{\nb}^s z_{1} \|_{L^{\frac{30}{7}+}_{T}L^\frac{30}{7}_x}, 
\| \jb{\nb}^sz_{1} \|_{L^{4}_{T, x}}, 
\| z_1(0)\|_{H^s} \Big). 
\label{O1a}
\end{align}

\noi
Then, 
by applying the dyadic decomposition, 
it suffices to prove
\begin{align}
\bigg| \int_0^T \int_{\R^3}
\P_{N_1} z_1 \cdot  \P_{N_2} z_1 \cdot N_3^\s \P_{N_3}  z_1 \cdot  \P_{N_4}  w  \, dx dt\bigg|
\leq  T^\theta N_{\max}^{0-} C(\|z_1\|_{A^s_3(T)})
\label{O2}
\end{align}

\noi
for all $N_1, \dots, N_4 \in 2^{\NB_0}$
with  $N_3 \geq N_2 \geq N_1$. 
Once we prove~\eqref{O2}, the desired estimate~\eqref{O1}
follows from summing~\eqref{O2} over dyadic blocks and applying Lemmas \ref{LEM:PStr} and  \ref{LEM:Hs}.
Recall our  shorthand notation:
$z_{1, N_j} = \P_{N_j} z_1$
and 
$w_{N_4} = \P_{N_4} w$. 

In the following, we assume
\begin{align}
 \s  < 2s
 \qquad \text{and}\qquad 
0 \leq  s < 1.
\label{O3}
\end{align}

\smallskip

\noi
{\bf $\bullet$  Case (1):} $N_2 \sim N_3$.
\newline
\indent
By H\"older's inequality and Lemma \ref{LEM:Str},
we have
\begin{align*}
\text{LHS of } \eqref{O2}
& \les \| z_{1, N_1} \|_{L^\frac{30}{7}_{T, x}}
\bigg(\prod_{j = 2}^3  \|\jb{\nb}^\frac{\s}{2} z_{1, N_j} \|_{L^{\frac{30}{7}}_{T, x}}\bigg)
 \| w_{N_4} \|_{L^\frac{10}{3}_{T, x}}\notag\\
& \leq T^{0+} N_{\max}^{0-} C(\|z_1\|_{A^s_3(T)}) \|w_{N_4}\|_{Y^0_T},  
\end{align*}

\noi
provided that \eqref{O3} holds.

\smallskip

\noi
{\bf $\bullet$ Case (2):} $N_3 \sim N_4 \gg N_1, N_2$.
\newline
\noi
$\circ$ \underline{Subcase (2.a):} 
$N_1, N_2 \ll N_3^\frac{1}{2}$.
\quad
\noi
By Cauchy-Schwarz' inequality and Lemma \ref{LEM:biStr} 
followed by Lemma \ref{LEM:lin},\footnote{In the remaining part of this paper, 
we repeatedly apply this argument when there is a frequency separation.
We shall simply refer to it as the ``bilinear Strichartz estimate'' argument.}
 we have
\begin{align*}
\text{LHS of } \eqref{O2}
&    \leq N_3^\s \|z_{1, N_1}   z_{1, N_3} \|_{L^2_{T, x}} \|z_{1, N_2}  w_{N_4}\|_{L^2_{T, x}}\\
& \les T^{0+} N_1^{1-s-} N_2^{1-s-} N_3^{\s - s - 1 +}
\bigg(\prod_{j = 1}^3
\|\P_{N_j} \phi^\o\|_{H^{s}}\bigg)
 \|w_{N_4}\|_{Y^0_T}\\
& \les T^{0+}  N_3^{\s - 2s  +}
\prod_{j = 1}^3
\|\P_{N_j} \phi^\o\|_{H^{s}}\\
& \leq 
T^{0+} N_{\max}^{0-} C(\|z_1\|_{A^s_3(T)}), 
\end{align*}

\noi
provided that \eqref{O3} holds.

\smallskip

\noi
$\circ$ \underline{Subcase (2.b):} $N_1, N_2\ges N_3^\frac{1}{2} $.
\quad 
By H\"older's inequality and Lemma \ref{LEM:Str}, 
we have
\begin{align*}
\text{LHS of } \eqref{O2}
& \le N_3^\s \bigg(\prod_{j = 1}^3 \| z_{1, N_j} \|_{L^\frac{30}{7}_{T, x}}\bigg)
 \| w_{N_4} \|_{L^\frac{10}{3}_{T, x}}\notag\\
& \le T^{0+}
N_1^{-s} N_2^{-s}N_3^{\s-s} 
 C(\|z_1\|_{A^s_3(T)})\notag\\
& \leq T^{0+} N_{\max}^{0-} C(\|z_1\|_{A^s_3(T)}), 
\end{align*}

\noi
provided that \eqref{O3} holds.

\smallskip

\noi
$\circ$ \underline{Subcase (2.c):} $N_2\ges N_3^\frac{1}{2} \gg N_1$.
\quad 
By the bilinear Strichartz estimate, 
 we have
\begin{align*}
\text{LHS of } \eqref{O2}
&  \le N_3^\s 
 \|z_{1, N_1}  w_{N_4} \|_{L^2_{T, x}}
\prod_{j = 2}^3 \|z_{1, N_j}\|_{L^4_{T, x}}
\\
& \les
T^{0+} N_1^{1-s-} N_2^{-s}
N_3^{\s-s-\frac 12+}
\|\P_{N_1}\phi^\o\|_{H^s}
\bigg(\prod_{j = 2}^3 \|\jb{\nb}^s z_{1, N_j} \|_{L^4_{T, x}}\bigg)
\|w_{N_4}\|_{Y^0_T}\\
& \le T^{0+}
N_3^{\s-2s+} 
 C(\|z_1\|_{A^s_3(T)})\\
 & \leq T^{0+} N_{\max}^{0-} C(\|z_1\|_{A^s_3(T)}), 
\end{align*}

\noi
provided that \eqref{O3} holds.

\smallskip

Therefore, putting all the cases together, we obtain \eqref{O2}.

\medskip

\noi
(ii) 
Given $N \gg 1$ and small $\l > 0$, consider the following deterministic initial condition
$\phi$ whose Fourier transform is given by 
\begin{align*}
 \ft \phi(\xi) = \ind_{ N e_1 + \l Q}(\xi) + \ind_{(N+100\l) e_1 + \l Q}(\xi) 
+ \ind_{ N e_1 +  100\l e_2 + \l Q}(\xi),
%\label{O5}
\end{align*}

\noi
where $Q = (-\frac 12, \frac 12]^3$, $e_1 = (1, 0, 0)$, and $e_2 = (0, 1, 0)$.
By taking $\l > 0$ sufficiently small, 
we have $\supp \ft \phi \subset N e_1 + Q$
and thus we can neglect the effect of the randomization in \eqref{rand}
since all the three terms on the right-hand side will be multiplied by 
a common random number $g_{N e_1}$.
Without loss of generality, we assume that 
$g_{N e_1} = 1$ in the following.

We estimate from below the contribution to 
$\ind_{Q_{N, \l}} (\xi) \ft z_3 (t, \xi) $,
where  \[Q_{N, \l}  = Ne_1 + 100 \l (e_1 + e_2) + \l Q.\] 
From \eqref{O-1}, we have
\begin{align}
\ind_{Q_{N, \l}} & (\xi) \ft z_3 (t, \xi) \notag \\
& = -i \ind_{Q_{N, \l}} (\xi) e^{-it |\xi|^2} \int_0^t \intt_{\xi = \xi_1 - \xi_2 + \xi_3} 
e^{i t' \Phi(\bar \xi)}  \ft{\phi^\o} (\xi_1) \cj{\ft{\phi^\o} (\xi_2)} \ft{\phi^\o} (\xi_3) d\xi_1 d\xi_2 dt',
\label{O6}
\end{align}

\noi
where the phase function $\Phi(\bar \xi)$ is given by 
\begin{align}
\Phi(\bar \xi) = \Phi(\xi, \xi_1, \xi_2, \xi_3)
= |\xi|^2 - |\xi_1|^2 + |\xi_2|^2 - |\xi_3|^2
= 2\jb{\xi - \xi_1, \xi - \xi_3}_{\R^3}.
\label{O7}
\end{align}

\noi
Then, it follows that the only non-trivial contribution to \eqref{O6}
appears if 
\begin{align}
\xi_1  \in (N+100\l) e_1 + \l Q, 
\qquad \xi_2  \in N e_1 + \l Q,
\qquad \xi_3  \in N e_1 +  100\l e_2 + \l Q
\label{O7a}
\end{align}

\noi
(up to the permutation $\xi_1 \leftrightarrow \xi_3$).
In this case, we have 
$|\Phi(\bar \xi)| \ll1 $
and thus  
\begin{align}
\Re e^{i t' \Phi(\bar \xi)} \geq \frac 12
\label{O8}
\end{align}

\noi
for all $t' \in [0, 1]$.

Now, recall the following lemma on the convolution.

\begin{lemma}\label{LEM:conv}
There exists $c>0$ such that 
\begin{align*}
\ind_{a + \l Q}* \ind_{b + \l Q} (\xi)\geq c \l^3 \ind_{a+b+\l Q}(\xi)
\end{align*}
		
\noi
for all $a, b, \xi \in \R^3$.
\end{lemma}

By applying Lemma \ref{LEM:conv} to \eqref{O6}
with \eqref{O7a} and \eqref{O8}, we obtain
\begin{align*}
|\ind_{Q_{N, \l}} (\xi) \ft z_3 (t, \xi) |
\ges t \l^6  \ind_{Q_{N, \l}} (\xi).
\end{align*}

\noi
Therefore, for any $\s > 0$, we have
\begin{align*}
\| z_3\|_{X^\s([0, 1])} \ges \| z_3 \|_{L^\infty_t([0, 1];  H^\s)}
\ges  \l^\frac{15}{2} N^\s \too \infty
\end{align*}

\noi
as $N \to \infty$, 
while $\| \phi\|_{L^2} \sim \l^\frac{3}{2}$ remains bounded.	
This in particular implies that
 when $s = 0$, 
 the estimate \eqref{Z3_2}
can not hold for any $\s > 0$.
This proves Part (ii).
\end{proof}

\begin{remark}\label{REM:Z3_2}
\rm

It follows from the proof of Proposition \ref{PROP:Z3} (i)
that 
\begin{align*}
\big\|\I( u_1, u_2, u_3)\big\|_{X^\s([0, 1])}
\les \prod_{j = 1}^3 \| u_j  \|_{A^s_3(T)},
\end{align*}

\noi
where $\I( u_1, u_2, u_3)$ is as in \eqref{Duhamel1}.
In particular, 
the left-hand side of \eqref{Z3_3} is finite
for $u_1, u_2, u_3 \in \mathcal{R}^s$.
The only probabilistic component
in the proof of Proposition \ref{PROP:Z3} (i)
appears in applying Lemmas~\ref{LEM:PStr}
and \ref{LEM:Hs}
to control the $A^s_3(T)$-norm of $z_1$
in terms of the $H^s$-norm of $\phi$.
In this sense, we exploit the randomization only at the linear level.

\end{remark}

\medskip

On the one hand, 
Proposition \ref{PROP:Z3} 
shows that $z_3$ controls  almost $2s$ derivatives.
On the other hand, we need to measure
$z_3$ in the $X^{2s-}$-norm,
which controls only the admissible space-time Lebesgue norms
(with $2s-$ derivatives) via Lemma \ref{LEM:Str}.
The following lemma breaks this rigidity by giving up a control on derivatives.
In particular, it allows us to control a wider range
of space-time Lebesgue norms of $z_3$.
The main idea is to use the dispersive estimate for the linear Schr\"odinger operator:
\begin{align}
\| S(t) f \|_{L^r_x} \les |t|^{-\frac{3}{2}(1 - \frac 2 r)} \| f\|_{L^{r'}_x}.
\label{O9}
\end{align}
	 
\noi
This allows us to reduce the analysis to a product of the random linear solution $z_1 = S(t) \phi^\o$
and apply Lemma \ref{LEM:PStr}.

\begin{lemma}\label{LEM:Z3_2}
Let $s \geq 0$.
Given  $\phi \in H^s(\R^3)$,  let $\phi^\o$ be its Wiener randomization
defined in~\eqref{rand}.
Then, for any finite $q, r \geq 1$, we have
\begin{align}
\| \P_N z_3\|_{L^q_T L^r_x}
\les \begin{cases}
\rule[-6mm]{0pt}{15pt} T^{\frac 3 r - \frac 12}\| z_1\|_{L^{3q}_T L^{3r'}_x}^3 & \text{when } 1 \leq r < 6, \\
 T^{0+} N^{\frac{1}{2} - \frac 3 r+}  \| z_1\|_{L^{3q}_T L^{\frac{18}{5}+}_x}^3 
&  \text{when }  r \geq  6, \\
\end{cases}
\label{O10}
\end{align}

\noi
for any $T > 0$ and $N \in 2^{\NB_0}$.	
Note that 
 the right-hand side of \eqref{O10} is almost surely finite thanks 
to the probabilistic Strichartz estimate (Lemma \ref{LEM:PStr}).
\end{lemma}

\begin{proof}

We first consider the case $r < 6$.
From \eqref{O-1} and \eqref{O9}, we have
\begin{align}
\| \P_N z_3\|_{L^q_T L^r_x}
& \leq \bigg\|\int_0^t \|\P_N S(t - t') |z_1|^2 z_1(t')\|_{L^r_x} dt'\bigg\|_{L^q_T}\notag \\
& \les \bigg\|\int_0^t \frac{1}{|t - t'|^{\frac{3}{2} - \frac 3 r}} \| z_1(t')\|_{L^{3r'}_x}^3 dt'\bigg\|_{L^q_T}\notag \\
& \les T^{\frac 3 r - \frac 12}\| z_1\|_{L^{3q}_T L^{3r'}_x}^3. 
\label{O11}
\end{align}

\noi
When $r \geq 6$, we proceed as in \eqref{O11}
but we apply Sobolev's inequality before applying~\eqref{O9}:
\begin{align*}
\| \P_N z_3\|_{L^q_T L^r_x}
& \leq \bigg\|\int_0^t \|\P_N S(t - t') |z_1|^2 z_1(t')\|_{L^r_x} dt'\bigg\|_{L^q_T}\notag \\
& \les N^{\frac{1}{2} - \frac 3 r+} \bigg\|\int_0^t \|\P_N S(t - t') |z_1|^2 z_1(t')\|_{L^{6-}_x} dt'\bigg\|_{L^q_T}\notag \\
& \les N^{\frac{1}{2} - \frac 3 r+} 
\bigg\|\int_0^t \frac{1}{|t - t'|^{1-}} \| z_1(t')\|_{L^{\frac{18}{5}+}_x}^3 dt'\bigg\|_{L^q_T}\notag \\
& \les T^{0+} N^{\frac{1}{2} - \frac 3 r+}  \| z_1\|_{L^{3q}_T L^{\frac{18}{5}+}_x}^3 .
\end{align*}

\noi
This completes the proof of Lemma \ref{LEM:Z3_2}.
\end{proof}

\section{On the higher order terms}\label{SEC:high}

In this section, we study the regularity properties of the higher order terms.

\subsection{On the third order term $z_5$}\label{SUBSEC:Z5}

In this subsection,  
we study the third order term:
\begin{align}
  z_5  = -i  \sum_{\substack{j_1 + j_2 + j_3 = 5\\j_1, j_2, j_3 \in \{1, 3\}  }}\int_0^t S(t - t') z_{j_1} \cj{z_{j_2}}z_{j_3} (t') dt'.
\label{P1}
\end{align}

\noi
The following lemma shows that 
the third order term $z_5$ enjoys a gain of extra $\frac 12s$ derivative
 as compared to the second order term
$z_3$ (Proposition~\ref{PROP:Z3}).

\begin{lemma}\label{LEM:Z5}
Given $0 < s < \frac 12 $, 
 let $\phi^\o$ be the Wiener randomization 
of $\phi \in H^s(\R^3)$
 defined in \eqref{rand}.
Then, for any $\s < \frac 52 s$, we have 
\begin{align*}
  z_5  \in X^\s_\textup{loc} , 
\end{align*}

\noi
almost surely.
In particular,  there exists an almost surely finite constant $ C(\o, \|\phi\|_{H^s} ) > 0$ and $\theta > 0$
such that 
\begin{align*}
\| z_5 \|_{X^{\s}([0, T])} \leq T^\theta C(\o,  \|\phi\|_{H^s}) 
%\label{P2}
\end{align*}

\noi
for any $T > 0$.
\end{lemma}

\begin{proof}
First, note that the only possible combination 
for $(j_1, j_2, j_3)$ in \eqref{P1} is $(1, 1, 3)$ up to permutations.
The complex conjugate does not play any role in the subsequent analysis
and hence we drop the complex conjugate sign 
and simply study
\begin{align}
  z_5  \sim   \int_0^t S(t - t') z_{1} z_1 z_{3} (t') dt'.
\label{P3}
\end{align}

\noi
By Lemma \ref{LEM:lin}, 
it suffices to prove 
\begin{align}
\bigg| \int_0^T \int_{\R^3}
\jb{\nb}^\s ( z_1 z_1 z_3 )w  dx dt\bigg|
\leq T^\theta C(\o,  \|\phi\|_{H^s}) 
\label{P4}
\end{align}

\noi
for some almost surely finite constant $ C(\o, \|\phi\|_{H^s} ) > 0$ and $\theta > 0$,  
where $\|w\|_{Y^{0}_T} \leq 1$.
Define $A^s_5(T)$ by 
\[ \|z_1\|_{A^s_5(T)}: = \max\Big(
\| \jb{\nb}^s z_{1} \|_{L^{5+}_T L^{5}_x}, 
\| z_1(0)\|_{H^s}, 
\Big). \]

\noi
Then, 
by applying the dyadic decomposition, 
it suffices to prove
\begin{align}
\bigg| \int_0^T \int_{\R^3}
N_{\max}^\s \P_{N_1} z_1 \cdot  \P_{N_2} z_1 
& \cdot \P_{N_3}  z_3 \cdot  
\P_{N_4}  w  \, dx dt\bigg| \notag\\
& \leq  T^\theta N_{\max}^{0-} C(\|z_1\|_{A^s_5(T)}, \|z_3\|_{X^{2s-}_T})
\label{P5}
\end{align}

\noi
for all $N_1, \dots, N_4 \in 2^{\NB_0}$.
Once we prove \eqref{P5}, the desired estimate \eqref{P4}
follows from summing over dyadic blocks and applying Lemmas \ref{LEM:PStr} and  \ref{LEM:Hs}
and Proposition \ref{PROP:Z3}.
In the following, we fix 
\begin{align*}
0 < s < \tfrac 12.
\end{align*}

\noi
Without loss of generality, we assume that $N_1 \geq N_2$.

\smallskip

\noi
{\bf $\bullet$  Case (1):} $N_1 \sim N_2 \sim N_{\max}$.
\newline
\noi
$\circ$ \underline{Subcase (1.a):} $N_3 \ll  N_1^\frac{1}{2}$.
\quad 
By the bilinear Strichartz estimate and Lemma \ref{LEM:Str}, 
 we have
\begin{align*}
\text{LHS of } \eqref{P5}
&  \les N_1^\s  \|z_{1, N_1}  z_{3, N_3} \|_{L^2_{T, x}}
  \| z_{1, N_2} \|_{L^{5}_{T, x}}
 \| w_{N_4} \|_{L^\frac{10}{3}_{T, x}}\notag\\
& \les T^{0+}
N_1^{\s - s - \frac 12+} N_2^{-s}
N_3^{1 - 2s-}
\|\P_{N_1}\phi^\o\|_{H^s}
 \|\jb{\nb}^s z_{1, N_2} \|_{L^5_{T, x}}
\|z_{3, N_3}\|_{Y^{2s-}_T}\\
& \le T^{0+}
N_1^{\s-3s+} 
 C(\|z_1\|_{A^s_5(T)}, \|z_{3}\|_{X^{2s-}_T}). 
\end{align*}

\noi
Hence, we obtain  
 \eqref{P5}, 
provided that $ \s  < 3s$.

%$\circ$ \underline{Subcase (1.a):} 
%$N_3, N_4 \ll N_1^\frac{1}{2}$.
%\quad 
%By the bilinear Strichartz estimate, 
% we have
%\begin{align*}
%\text{LHS of } \eqref{P5}
%&    \les N_1^\s\|z_{1, N_1}   z_{3, N_3} \|_{L^2_{T, x}} \|z_{1, N_2}  w_{N_4}\|_{L^2_{T, x}}\\
%& \les T^{0+} N_1^{\s-2s-1 + }  N_3^{1-  2s  -}N_4^{1-}
%\bigg(\prod_{j = 1}^2
%\|\P_{N_j} \phi^\o\|_{H^{s}}\bigg)
% \|z_{3, N_3}\|_{Y^{2s-}_T}
% \|w_{N_4}\|_{Y^0_T}\\
%& \les T^{0+}  N_1^{\s - 3s  +}
%\bigg(\prod_{j = 1}^2
%\|\P_{N_j} \phi^\o\|_{H^{s}}\bigg) \|z_{3, N_3}\|_{Y^{2s-}_T}\\
%& \leq 
%T^{0+} N_{\max}^{0-} C(\|z_1\|_{A^s_5(T)}, \|z_3\|_{X^{2s-}_T}), 
%\end{align*}
%
%
%\noi
%provided that  
%$ \s  < 3s$.

\smallskip

\noi
$\circ$ \underline{Subcase (1.b):} $N_3\ges N_1^\frac{1}{2} $.
\quad 
By H\"older's inequality and Lemma \ref{LEM:Str}, 
we have
\begin{align*}
\text{LHS of } \eqref{P5}
& \les N_1^\s \| z_{1, N_1} \|_{L^{5}_{T, x}}
  \| z_{1, N_2} \|_{L^{5}_{T, x}}
  \| z_{3, N_3} \|_{L^{\frac{10}{3}}_{T, x}}
 \| w_{N_4} \|_{L^\frac{10}{3}_{T, x}}\notag\\
& \le T^{0+}
N_1^{\s -s} N_2^{-s}N_3^{-2s+} 
 C(\|z_1\|_{A^s_5(T)}, \|z_3\|_{X^{2s-}_T})\notag\\
& \le T^{0+}
N_1^{\s-3s+} 
 C(\|z_1\|_{A^s_5(T)}, \|z_{3}\|_{X^{2s-}_T}).
\end{align*}

\noi
Hence, we obtain  
 \eqref{P5}, 
provided that $ \s  < 3s$.

\smallskip

\noi
{\bf $\bullet$  Case (2):} $N_1 \sim N_3 \sim N_{\max} \gg N_2$.
\newline
\noi
$\circ$ \underline{Subcase (2.a):} $N_2 \ll N_1^\frac{1}{2}$.
\quad 
By the bilinear Strichartz estimate
and Lemma \ref{LEM:Str},  
 we have
\begin{align*}
\text{LHS of } \eqref{P5}
&  \les N_1^\s    \| z_{1, N_1} \|_{L^{5}_{T, x}} \|z_{1, N_2}  z_{3, N_3} \|_{L^2_{T, x}}
 \| w_{N_4} \|_{L^\frac{10}{3}_{T, x}}\notag\\
& \les
T^{0+}
N_1^{\s - s } N_2^{1-s-}
N_3^{- 2s-\frac 12 +}
 \|\jb{\nb}^s z_{1, N_1} \|_{L^5_{T, x}}
\|\P_{N_2}\phi^\o\|_{H^s}
\|z_{3, N_3}\|_{Y^{2s-}_T}\\
& \le T^{0+}
N_1^{\s-\frac 72s+} 
 C(\|z_1\|_{A^s_5(T)}, \|z_{3}\|_{X^{2s-}_T}).
\end{align*}

\noi
Hence, we obtain \eqref{P5}, 
provided that $ \s  < \frac 72s$.

%$\circ$ \underline{Subcase (2.a):} 
%$N_2, N_4  \ll N_1^\frac{1}{2} $.
%\quad 
%By the bilinear Strichartz estimate, 
% we have
%\begin{align*}
%\text{LHS of } \eqref{P5}
%&    \les N_1^\s\|z_{1, N_1}   z_{1, N_2} \|_{L^2_{T, x}} \|z_{3, N_3}  w_{N_4}\|_{L^2_{T, x}}\\
%& \les T^{0+} N_1^{\s-s-\frac 12 + }N_2^{1-s - }  N_3^{-  2s -\frac 12+}N_4^{1-}
%\bigg(\prod_{j = 1}^2
%\|\P_{N_j} \phi^\o\|_{H^{s}}\bigg)
% \|z_{3, N_3}\|_{Y^{2s-}_T}\\
%& \le T^{0+}  N_1^{\s - \frac 72 s  +}
%C(\|z_1\|_{A^s_5(T)}, \|z_3\|_{X^{2s-}_T}).
%\end{align*}
%
%
%\noi
%Hence, we obtain  \eqref{P5}, 
%provided that 
%$ \s  < \frac 72s$.

\smallskip

\noi
$\circ$ \underline{Subcase (2.b):} $N_2\ges N_1^\frac{1}{2} $.
\quad 
By H\"older's inequality and Lemma \ref{LEM:Str}, 
we have
\begin{align*}
\text{LHS of } \eqref{P5}
& \les N_1^\s \| z_{1, N_1} \|_{L^{5}_{T, x}}
  \| z_{1, N_2} \|_{L^{5}_{T, x}}
  \| z_{3, N_3} \|_{L^{\frac{10}{3}}_{T, x}}
 \| w_{N_4} \|_{L^\frac{10}{3}_{T, x}}\notag\\
& \le T^{0+}
N_1^{\s -s} N_2^{-s}N_3^{-2s+} 
 C(\|z_1\|_{A^s_5(T)}, \|z_3\|_{X^{2s-}_T})\notag\\
& \le T^{0+}
N_1^{\s-\frac 72 s+} 
 C(\|z_1\|_{A^s_5(T)}, \|z_{3}\|_{X^{2s-}_T}).
\end{align*}

\noi
Hence, we obtain \eqref{P5}, 
provided that $ \s  < \frac 72s$.

\smallskip

\noi
{\bf $\bullet$  Case (3):} $N_1 \sim N_4 \sim N_{\max} \gg N_2, N_3$.
\newline
\noi
$\circ$ \underline{Subcase (3.a):} 
$N_2, N_3  \ll N_1^\frac{1}{2} $.
\quad 
By the bilinear Strichartz estimate, 
 we have
\begin{align*}
\text{LHS of } \eqref{P5}
&    \les N_1^\s\|z_{1, N_1}   z_{1, N_2} \|_{L^2_{T, x}} \|z_{3, N_3}  w_{N_4}\|_{L^2_{T, x}}\\
& \les T^{0+} N_1^{\s-s-\frac 12 + }N_2^{1-s - }  N_3^{1-  2s -}N_4^{-\frac{1}{2}+}
\bigg(\prod_{j = 1}^2
\|\P_{N_j} \phi^\o\|_{H^{s}}\bigg)
 \|z_{3, N_3}\|_{Y^{2s-}_T}
 \|w_{N_4}\|_{Y^0_T}\\
& \le T^{0+}  N_1^{\s - \frac 52 s  +}
C(\|z_1\|_{A^s_5(T)}, \|z_3\|_{X^{2s-}_T}).
\end{align*}

\noi
Hence, we obtain  \eqref{P5}, 
provided that 
$ \s  < \frac 52s$.

\noi
$\circ$ \underline{Subcase (3.b):} $N_2, N_3\ges N_1^\frac{1}{2} $.
\quad 
By H\"older's inequality and Lemma \ref{LEM:Str}, 
we have
\begin{align*}
\text{LHS of } \eqref{P5}
& \les N_1^\s \| z_{1, N_1} \|_{L^{5}_{T, x}}
  \| z_{1, N_2} \|_{L^{5}_{T, x}}
  \| z_{3, N_3} \|_{L^{\frac{10}{3}}_{T, x}}
 \| w_{N_4} \|_{L^\frac{10}{3}_{T, x}}\notag\\
& \le T^{0+}
N_1^{\s -s} N_2^{-s}N_3^{-2s+} 
 C(\|z_1\|_{A^s_5(T)}, \|z_3\|_{X^{2s-}_T})\notag\\
& \le T^{0+}
N_1^{\s-\frac 52 s+} 
 C(\|z_1\|_{A^s_5(T)}, \|z_{3}\|_{X^{2s-}_T}).
\end{align*}

\noi
Hence, we obtain  \eqref{P5}, 
provided that 
$ \s  < \frac 52s$.

\smallskip

\noi
$\circ$ \underline{Subcase (3.c):} $N_2\ges N_1^\frac{1}{2} \gg N_3$.
\quad 
By the bilinear Strichartz estimate
and Lemma \ref{LEM:Str},  
 we have
\begin{align*}
\text{LHS of } \eqref{P5}
&  \les N_1^\s  \|z_{1, N_1}  z_{3, N_3} \|_{L^2_{T, x}}  \| z_{1, N_2} \|_{L^{5}_{T, x}} 
 \| w_{N_4} \|_{L^\frac{10}{3}_{T, x}}\notag\\
& \le T^{0+}
N_1^{\s - s- \frac 12 +  } N_2^{-s}
N_3^{1- 2s-}
 C(\|z_1\|_{A^s_5(T)}, \|z_{3}\|_{X^{2s-}_T})\\
& \le T^{0+}
N_1^{\s-\frac 52s+} 
 C(\|z_1\|_{A^s_5(T)}, \|z_{3}\|_{X^{2s-}_T}).
\end{align*}

\noi
Hence, we obtain \eqref{P5}, 
provided that 
$ \s  < \frac 52s$.

\smallskip

\noi
$\circ$ \underline{Subcase (3.d):} $N_3\ges N_1^\frac{1}{2} \gg N_2$.
\quad 
By the bilinear Strichartz estimate
and Lemma \ref{LEM:Str},  
 we have
\begin{align*}
\text{LHS of } \eqref{P5}
&  \les N_1^\s    \| z_{1, N_1} \|_{L^{5}_{T, x}} 
 \| z_{3, N_3} \|_{L^\frac{10}{3}_{T, x}}
\|z_{1, N_2}  w_{N_4} \|_{L^2_{T, x}}
\notag\\
& \le T^{0+}
N_1^{\s - s } N_2^{1-s-}
N_3^{- 2s+} N_4^{-\frac 12 +}
 C(\|z_1\|_{A^s_5(T)}, \|z_{3}\|_{X^{2s-}_T})\\
& \le T^{0+}
N_1^{\s-\frac 52s+} 
 C(\|z_1\|_{A^s_5(T)}, \|z_{3}\|_{X^{2s-}_T}).
\end{align*}

\noi
Hence, we obtain \eqref{P5}, 
provided that 
$ \s  < \frac 52s$.

\smallskip

\noi
{\bf $\bullet$ Case (4):} $N_3 \sim N_4 \gg N_1 \geq  N_2$.
\newline
\noi
$\circ$ \underline{Subcase (4.a):} 
$N_1, N_2  \ll N_3^\frac{1}{2} $.
\quad 
By the bilinear Strichartz estimate, 
 we have
\begin{align*}
\text{LHS of } \eqref{P5}
&    \les N_3^\s\|z_{1, N_1}   z_{3, N_3} \|_{L^2_{T, x}} \|z_{1, N_2}  w_{N_4}\|_{L^2_{T, x}}\\
& \les T^{0+} N_1^{1-s- }N_2^{1-s - }  N_3^{\s-  2s -\frac 12 +}N_4^{-\frac{1}{2}+}
C(\|z_1\|_{A^s_5(T)}, \|z_3\|_{X^{2s-}_T})\\
& \le T^{0+}  N_1^{\s - 3 s  +}
C(\|z_1\|_{A^s_5(T)}, \|z_3\|_{X^{2s-}_T}).
\end{align*}

\noi
Hence, we obtain  \eqref{P5}, 
provided that $ \s  < 3s$.

\smallskip

\noi
$\circ$ \underline{Subcase (4.b):} $N_1, N_2\ges N_3^\frac{1}{2} $.
\quad 
By H\"older's inequality and Lemma \ref{LEM:Str}, 
we have
\begin{align*}
\text{LHS of } \eqref{P5}
& \les N_3^\s \| z_{1, N_1} \|_{L^{5}_{T, x}}
  \| z_{1, N_2} \|_{L^{5}_{T, x}}
  \| z_{3, N_3} \|_{L^{\frac{10}{3}}_{T, x}}
 \| w_{N_4} \|_{L^\frac{10}{3}_{T, x}}\notag\\
& \le T^{0+}
N_1^{ -s} N_2^{-s}N_3^{\s-2s+} 
 C(\|z_1\|_{A^s_5(T)}, \|z_3\|_{X^{2s-}_T})\notag\\
& \le T^{0+}
N_1^{\s-3 s+} 
 C(\|z_1\|_{A^s_5(T)}, \|z_{3}\|_{X^{2s-}_T}).
\end{align*}

\noi
Hence, we obtain \eqref{P5}, 
provided that $ \s  < 3s$.

\smallskip

\noi
$\circ$ \underline{Subcase (4.c):} $N_1\ges N_3^\frac{1}{2} \gg N_2$.
\quad 
By the bilinear Strichartz estimate
and Lemma \ref{LEM:Str}, 
 we have
\begin{align*}
\text{LHS of } \eqref{P5}
&  \les N_3^\s  \| z_{1, N_2} \|_{L^{5}_{T, x}}   \|z_{1, N_2}  z_{3, N_3} \|_{L^2_{T, x}} 
 \| w_{N_4} \|_{L^\frac{10}{3}_{T, x}}\notag\\
& \les T^{0+}
N_1^{ - s} N_2^{1-s-}
N_3^{\s- 2s- \frac 12+}
 C(\|z_1\|_{A^s_5(T)}, \|z_{3}\|_{X^{2s-}_T})\\
& \le T^{0+}
N_1^{\s- 3s+} 
 C(\|z_1\|_{A^s_5(T)}, \|z_{3}\|_{X^{2s-}_T}).
\end{align*}

\noi
Hence, we obtain \eqref{P5}, 
provided that $ \s  < 3s$.

\smallskip

Putting all the cases together, 
we conclude that \eqref{P5} holds, 
provided that  $\s < \frac 52 s$.
This completes the proof of Lemma \ref{LEM:Z5}.
\end{proof}

\subsection{On the fourth order term $z_7$}\label{SUBSEC:Z7}

Next,  we study the following fourth order term:
\begin{align}
  z_7  = -i  \sum_{\substack{j_1 + j_2 + j_3 = 7\\j_1, j_2, j_3 \in \{1, 3, 5\}  }}\int_0^t S(t - t') z_{j_1} \cj{z_{j_2}}z_{j_3} (t') dt'.
\label{Q1}
\end{align}

\noi
In this case, 
there are two possibilities
for $(j_1, j_2, j_3)$ in \eqref{Q1}:
$(1, 3, 3)$ and $(1, 1, 5)$ up to permutations.
We denote by $\wt z_7$ 
the contribution to $z_7$ 
from 
$(j_1, j_2, j_3) = (1, 3, 3)$  (up to permutations).
Note that the contribution to $z_7$ from 
$(j_1, j_2, j_3) = (1, 1, 5)$  (up to permutations)
corresponds to $\zeta_7$ defined in \eqref{zeta2}.
Dropping 
the complex conjugate, we have
\begin{align}
  \wt z_7 &  \sim   \int_0^t S(t - t') z_{1} z_3 z_{3} (t') dt',   \label{Q2}\\
\zeta_7  & \sim   \int_0^t S(t - t') z_{1} z_1 z_{5} (t') dt'.
\label{Q3}
\end{align}

\noi
The following lemma shows that 
$\wt z_7$ and $\zeta_7$ enjoy a further gain of derivatives
as compared to $z_1$, $z_3$, and $z_5$.

\begin{lemma}\label{LEM:Z7}
Given   $\phi \in H^s(\R^3)$, 
let  $\phi^\o$ be the Wiener randomization  of $\phi$
 defined in \eqref{rand}.

\smallskip

\noi
\textup{(i)}
Let $ 0 < s < \frac 12$.
Then, given any $\s < 3s$, we have 
\begin{align*}
   \wt z_7  \in X^\s_\textup{loc} ,
\end{align*}

\noi
 almost surely.

\smallskip

\noi
\textup{(ii)}
Let $ 0 < s < \frac 25$.
Then, given any $\s < \frac{11}{4}s$, we have 
\begin{align}
\zeta_7  \in X^\s_\textup{loc} ,
\label{Q4}
\end{align}

\noi
 almost surely.

\end{lemma}

\begin{proof}
(i) We first estimate $\wt z_7$ in \eqref{Q2}.
Fix $0 < s < \frac 12$.
We proceed as in  the proofs of Proposition \ref{PROP:Z3} and Lemma~\ref{LEM:Z5}.
In view of 
 Lemmas \ref{LEM:PStr} and  \ref{LEM:Hs}
and Proposition \ref{PROP:Z3}, 
it suffices to prove that there exists $\theta > 0$ such that 
\begin{align}
\bigg| \int_0^T \int_{\R^3}
N_{\max}^\s \P_{N_1} z_1 \cdot  \P_{N_2} z_3 
& \cdot \P_{N_3}  z_3 \cdot  \P_{N_4}  w  \, dx dt\bigg|\notag\\
& \leq  T^\theta N_{\max}^{0-} C(\|z_1\|_{A^s_7(T)}, \|z_3\|_{X^{2s-}_T}).
\label{Q5}
\end{align}

\noi
for all $N_1, \dots, N_4 \in 2^{\NB_0}$
and  $\|w\|_{Y^{0}_T} \leq 1$, 
where  $A^s_7(T)$ is given by 
\[ \|z_1\|_{A^s_7(T)}: = \max\Big(
\| \jb{\nb}^s z_{1} \|_{L^{5}_{T, x}}, 
\| \jb{\nb}^s z_{1} \|_{L^{10+}_T L^{10}_x}, 
\| \jb{\nb}^s z_1\|_{L^{15}_T L^{\frac{15}{4}}_x}, 
\| z_1(0)\|_{H^s}
\Big). \]

\noi
Without loss of generality, 
we assume that $N_2 \geq N_3$.

\smallskip

\noi
{\bf $\bullet$  Case (1):} $N_1 \sim N_2 \sim N_{\max}$.
\newline
\noi
$\circ$ \underline{Subcase (1.a):} $N_3 \ll N_1^\frac{1}{2}$.
\quad 
By the bilinear Strichartz estimate and 
 Lemma \ref{LEM:Str}, 
 we have
\begin{align*}
\text{LHS of } \eqref{Q5}
&  \les N_1^\s   \| z_{1, N_1} \|_{L^{5}_{T, x}}  \|z_{3, N_2}  z_{3, N_3} \|_{L^2_{T, x}}
 \| w_{N_4} \|_{L^\frac{10}{3}_{T, x}}\notag\\
& \les
T^{0+}
N_1^{\s - s } N_2^{-2s - \frac 12 + }
N_3^{1 - 2s-}
 \|\jb{\nb}^s z_{1, N_1} \|_{L^5_{T, x}}
\bigg(\prod_{j = 2}^3 \|z_{3, N_j}\|_{Y^{2s-}_T}\bigg)
\\
& \le T^{0+}
N_1^{\s-4s+} 
 C(\|z_1\|_{A^s_7(T)}, \|z_{3}\|_{X^{2s-}_T}).
\end{align*}

\noi
Hence, we obtain  \eqref{Q5}, 
provided that $\s < 4s$. 

%$\circ$ \underline{Subcase (1.a):} 
%$N_3, N_4 \ll N_1^\frac{1}{2}$.
%\quad 
%By the bilinear Strichartz estimate, 
% we have
%\begin{align*}
%\text{LHS of } &  \eqref{Q5}
%    \les N_1^\s\|z_{1, N_1}   z_{3, N_3} \|_{L^2_{T, x}} \|z_{3, N_2}  w_{N_4}\|_{L^2_{T, x}}\\
%& \les T^{0+} N_1^{\s-s-\frac 12 + }  N_2^{-2s - \frac 12+}N_3^{1-  2s  -}N_4^{1-}
%\|\P_{N_1} \phi^\o\|_{H^{s}}
%\bigg(\prod_{j = 2}^3
% \|z_{3, N_j}\|_{Y^{2s-}_T}\bigg)
% \|w_{N_4}\|_{Y^0_T}\\
%& \le T^{0+}  N_1^{\s - 4s  +}
% C(\|z_1\|_{A^s_7(T)}, \|z_3\|_{X^{2s-}_T}).
%\end{align*}
%
%
%\noi
%Hence, we obtain  \eqref{Q5}, 
%provided that $\s < 4s$.

\smallskip

\noi
$\circ$ \underline{Subcase (1.b):} $N_3\ges N_1^\frac{1}{2} $.
\quad 
By H\"older's inequality and Lemma \ref{LEM:Str}, 
we have
\begin{align*}
\text{LHS of } \eqref{Q5}
& \les N_1^\s \|  z_{1, N_1} \|_{L^{10}_{T, x}}
  \|   z_{3, N_2} \|_{L^\frac{10}{3}_{T, x}}
  \| z_{3, N_3} \|_{L^{\frac{10}{3}}_{T, x}}
 \| w_{N_4} \|_{L^\frac{10}{3}_{T, x}}\notag\\
& \le T^{0+}
N_1^{\s -s} N_2^{-2s+}N_3^{-2s+} 
 C(\|z_1\|_{A^s_7(T)}, \|z_3\|_{X^{2s-}_T})\notag\\
& \le T^{0+}
N_1^{\s-4s+} 
 C(\|z_1\|_{A^s_7(T)}, \|z_{3}\|_{X^{2s-}_T}).
\end{align*}

\noi
Hence, we obtain  \eqref{Q5}, 
provided that $\s < 4s$.

\smallskip

\noi
{\bf $\bullet$  Case (2):} $N_1 \sim N_4 \sim N_{\max} \gg N_2 \geq  N_3$.
\newline
\noi
$\circ$ \underline{Subcase (2.a):} 
$N_2, N_3  \ll N_1^\frac{1}{2} $.
\quad 
By the bilinear Strichartz estimate, 
 we have
\begin{align*}
\text{LHS of }   \eqref{Q5}
&     \les N_1^\s\|z_{1, N_1}   z_{3, N_2} \|_{L^2_{T, x}} \|z_{3, N_3}  w_{N_4}\|_{L^2_{T, x}}\\
& \les T^{0+} N_1^{\s-s-\frac 12 + }N_2^{1-2s - }  N_3^{1-  2s -}N_4^{-\frac{1}{2}+}
\|\P_{N_1} \phi^\o\|_{H^{s}}
\bigg(\prod_{j = 2}^3
 \|z_{3, N_j}\|_{Y^{2s-}_T}
\bigg)
\\
& \le T^{0+}  N_1^{\s -  3 s  +}
C(\|z_1\|_{A^s_7(T)}, \|z_3\|_{X^{2s-}_T}).
\end{align*}

\noi
Hence, we obtain 
 \eqref{Q5}, 
provided that $\s < 3s$.

\smallskip

\noi
$\circ$ \underline{Subcase (2.b):} $N_2, N_3\ges N_1^\frac{1}{2} $.
\quad 
By $L^{10}_{T, x}, L^{\frac{10}{3}}_{T, x}, 
L^{\frac{10}{3}}_{T, x}, L^{\frac{10}{3}}_{T, x}$-H\"older's
 inequality and Lemma~\ref{LEM:Str}, 
we have
\begin{align*}
\text{LHS of } \eqref{Q5}
& \le T^{0+}
N_1^{\s -s} N_2^{-2s+}N_3^{-2s+} 
 C(\|z_1\|_{A^s_7(T)}, \|z_3\|_{X^{2s-}_T})\notag\\
& \le T^{0+}
N_1^{\s-3 s+} 
 C(\|z_1\|_{A^s_7(T)}, \|z_{3}\|_{X^{2s-}_T}).
\end{align*}

\noi
Hence, we obtain  \eqref{Q5}, 
provided that $\s < 3s$.

\smallskip

\noi
$\circ$ \underline{Subcase (2.c):} $N_2\ges N_1^\frac{1}{2} \gg N_3$.
\quad 
By the bilinear Strichartz estimate
and Lemma \ref{LEM:Str},  
 we have
\begin{align*}
\text{LHS of } \eqref{Q5}
&  \les N_1^\s    \| z_{1, N_1} \|_{L^{5}_{T, x}} 
 \| z_{3, N_2} \|_{L^\frac{10}{3}_{T, x}}
 \|z_{3, N_3}  w_{N_4} \|_{L^2_{T, x}}
 \notag\\
& \le T^{0+}
N_1^{\s - s  } N_2^{-2s+}
N_3^{1- 2s-}N_4^{-\frac{1}{2}+}
 C(\|z_1\|_{A^s_7(T)}, \|z_{3}\|_{X^{2s-}_T})\\
& \le T^{0+}
N_1^{\s- 3s+} 
 C(\|z_1\|_{A^s_7(T)}, \|z_{3}\|_{X^{2s-}_T}).
\end{align*}

\noi
Hence, we obtain \eqref{Q5}, 
provided that $\s < 3s$.

\smallskip

\noi
{\bf $\bullet$  Case (3):} $N_2 \sim N_3 \sim N_{\max} \gg N_1$.
\newline
\noi
$\circ$ \underline{Subcase (3.a):} $N_1 \ll N_2^\frac{1}{2}$.
\quad 
By the bilinear Strichartz estimate 
and Lemma \ref{LEM:Str}, 
 we have
\begin{align*}
\text{LHS of } \eqref{Q5}
&  \les N_2^\s  
 \|z_{1, N_1}  z_{3, N_2} \|_{L^2_{T, x}}
 \| z_{3, N_3} \|_{L^{5}_{T, x}} 
 \| w_{N_4} \|_{L^\frac{10}{3}_{T, x}}\notag\\
& \les
T^{0+}
N_1^{ 1- s -  } N_2^{\s -2s - \frac 12 + }
N_3^{-s}
\| \P_{N_1} \phi^\o\|_{H^s}
 \|z_{3, N_2}\|_{Y^{2s-}_T}
\| \jb{\nb}^s z_{3, N_3} \|_{L^{5}_{T, x}}\\
\intertext{By applying Lemma \ref{LEM:Z3_2}
and the fractional Leibniz rule,} 
& \les
T^{\frac{1}{10}+}
N_2^{\s-\frac 72 s+} 
\| \P_{N_1} \phi^\o\|_{H^s}
 \|z_{3, N_2}\|_{Y^{2s-}_T}
\| \jb{\nb}^s z_1\|_{L^{15}_T L^{\frac{15}{4}}_x}
\| z_1\|_{L^{15}_T L^{\frac{15}{4}}_x}^2\\
& \le T^{\frac{1}{10}+}
N_2^{\s-\frac 72s+} 
 C(\|z_1\|_{A^s_7(T)}, \|z_{3}\|_{X^{2s-}_T}).
\end{align*}

\noi
Hence, we obtain 
 \eqref{Q5}, 
provided that $\s < \frac 72s$.

%$\circ$ \underline{Subcase (3.a):} 
%$N_1, N_4 \ll N_2^\frac{1}{2}$.
%\quad 
%By the bilinear Strichartz estimate, 
% we have
%\begin{align*}
%\text{LHS of }   \eqref{Q5}
%&     \les N_2^\s\|z_{1, N_1}   z_{3, N_2} \|_{L^2_{T, x}} \|z_{3, N_3}  w_{N_4}\|_{L^2_{T, x}}\\
%& \le T^{0+} N_1^{1-s-}  N_2^{\s-2s - \frac 12+}N_3^{-  2s  - \frac 12 + }N_4^{1-}
% C(\|z_1\|_{A^s_7(T)}, \|z_3\|_{X^{2s-}_T})\\
%& \le T^{0+}  N_1^{\s - \frac 92s  +}
% C(\|z_1\|_{A^s_7(T)}, \|z_3\|_{X^{2s-}_T}).
%\end{align*}
%
%
%\noi
%Hence, we obtain  \eqref{Q5}, 
%provided that 
%$\s < \frac 92s$.

\smallskip

\noi
$\circ$ \underline{Subcase (3.b):} $N_1\ges N_2^\frac{1}{2} $.
\quad 
By $L^{10}_{T, x}, L^{\frac{10}{3}}_{T, x}, 
L^{\frac{10}{3}}_{T, x}, L^{\frac{10}{3}}_{T, x}$-H\"older's
 inequality and Lemma~\ref{LEM:Str}, 
we have
\begin{align*}
\text{LHS of } \eqref{Q5}
& \le T^{0+}
N_1^{ -s} N_2^{\s -2s+}N_3^{-2s+} 
 C(\|z_1\|_{A^s_7(T)}, \|z_3\|_{X^{2s-}_T})\notag\\
& \le T^{0+}
N_1^{\s-\frac 92s+} 
 C(\|z_1\|_{A^s_7(T)}, \|z_{3}\|_{X^{2s-}_T}).
\end{align*}

\noi
Hence, we obtain  \eqref{Q5}, 
provided that $\s < \frac 92s$.

\smallskip

\noi
{\bf $\bullet$ Case (4):} $N_2 \sim N_4 \gg N_1, N_3$.
\newline
\noi
$\circ$ \underline{Subcase (4.a):} 
$N_1, N_3  \ll N_2^\frac{1}{2} $.
\quad 
By the bilinear Strichartz estimate, 
 we have
\begin{align*}
\text{LHS of } \eqref{Q5}
&    \les N_2^\s\|z_{1, N_1}   z_{3, N_2} \|_{L^2_{T, x}} \|z_{3, N_3}  w_{N_4}\|_{L^2_{T, x}}\\
& \le T^{0+} N_1^{1-s- }N_2^{\s -2s -\frac 12+ }  N_3^{1-  2s-}N_4^{-\frac{1}{2}+}
C(\|z_1\|_{A^s_7(T)}, \|z_3\|_{X^{2s-}_T})\\
& \le T^{0+}  N_1^{\s - \frac 72  s  +}
C(\|z_1\|_{A^s_7(T)}, \|z_3\|_{X^{2s-}_T}).
\end{align*}

\noi
Hence, we obtain  \eqref{Q5}, 
provided that $\s < \frac 72s$.

\smallskip

\noi
$\circ$ \underline{Subcase (4.b):} $N_1, N_3\ges N_2^\frac{1}{2} $.
\quad 
By $L^{10}_{T, x}, L^{\frac{10}{3}}_{T, x}, 
L^{\frac{10}{3}}_{T, x}, L^{\frac{10}{3}}_{T, x}$-H\"older's
 inequality and Lemma~\ref{LEM:Str}, 
we have
\begin{align*}
\text{LHS of } \eqref{Q5}
& \le T^{0+}
N_1^{ -s} N_2^{\s -2s+}N_3^{-2s+} 
 C(\|z_1\|_{A^s_7(T)}, \|z_3\|_{X^{2s-}_T})\notag\\
& \le T^{0+}
N_2^{\s-\frac 72 s+} 
 C(\|z_1\|_{A^s_7(T)}, \|z_{3}\|_{X^{2s-}_T}).
\end{align*}

\noi
Hence, we obtain  \eqref{Q5}, 
provided that $\s < \frac 72s$.

\smallskip

\noi
$\circ$ \underline{Subcase (4.c):} $N_1\ges N_2^\frac{1}{2} \gg N_3$.
\quad 
By the bilinear Strichartz estimate and Lemma \ref{LEM:Str}, 
 we have
\begin{align*}
\text{LHS of } \eqref{Q5}
&  \les N_2^\s  \| z_{1, N_1} \|_{L^{5}_{T, x}}   \|z_{3, N_2}  z_{3, N_3} \|_{L^2_{T, x}} 
 \| w_{N_4} \|_{L^\frac{10}{3}_{T, x}}\notag\\
& \le T^{0+}
N_1^{ - s} N_2^{\s -2s-\frac 12 + }
N_3^{1- 2s-}
 C(\|z_1\|_{A^s_7(T)}, \|z_{3}\|_{X^{2s-}_T})\\
& \le T^{0+}
N_1^{\s-\frac 72s+} 
 C(\|z_1\|_{A^s_7(T)}, \|z_{3}\|_{X^{2s-}_T}).
\end{align*}

\noi
Hence, we obtain  \eqref{Q5}, 
provided that $\s < \frac 72s$.

\noi
$\circ$ \underline{Subcase (4.d):} $N_3\ges N_2^\frac{1}{2} \gg N_1$.
\quad 
By the bilinear Strichartz estimate and Lemma \ref{LEM:Str}, 
 we have
\begin{align*}
\text{LHS of } \eqref{Q5}
&  \les N_2^\s 
  \|z_{1, N_1}  z_{3, N_2} \|_{L^2_{T, x}} 
 \| z_{3, N_3} \|_{L^{5}_{T, x}} 
 \| w_{N_4} \|_{L^\frac{10}{3}_{T, x}}\notag\\
& \les T^{0+}N_1^{ 1- s -  } N_2^{\s -2s - \frac 12 + }
N_3^{-s}
\| \P_{N_1} \phi^\o\|_{H^s}
 \|z_{3, N_2}\|_{Y^{2s-}_T}
\| \jb{\nb}^s z_{3, N_3} \|_{L^{5}_{T, x}}\\
\intertext{Proceeding as in Subcase (3.a) with Lemma \ref{LEM:Z3_2},}
& \le T^{\frac{1}{10}}
N_1^{\s- 3s+} 
 C(\|z_1\|_{A^s_7(T)}, \|z_{3}\|_{X^{2s-}_T}).
\end{align*}

\noi
Hence, we obtain  \eqref{Q5}, 
provided that $\s < 3s$.

\smallskip

Putting all the cases together, 
we conclude that \eqref{Q5} holds
when 
$ \s  < 3s$.

\medskip

\noi
(ii) 
Next, we estimate $\zeta_7$ in \eqref{Q3}.
This term has a similar structure
to $z_5$ in~\eqref{P3};
the only difference appears in the third factor.
Hence,  we can estimate $\zeta_7$ 
simply by replacing the regularity $2s-$ (for $z_3$; see Proposition \ref{PROP:Z3}) with 
$\frac{5}{2} s-$ (for $z_5$; see Lemma \ref{LEM:Z5})
in the proof  of Lemma~\ref{LEM:Z5}.
In the following, 
we only indicate the necessary modifications
on the powers of dyadic parameters
in the proof of Lemma \ref{LEM:Z5}.

\smallskip

\noi
{\bf $\bullet$  Case (1):} $N_1 \sim N_2 \sim N_{\max}$.
\newline
\noi
$\circ$ \underline{Subcase (1.a):} 
$N_3 \ll N_1^\frac{1}{2}$.
\quad It suffices to note that 
\begin{align*}
N_1^{\s - s - \frac 12+} N_2^{-s}
N_3^{1 - \frac52 s-}
%
%N_1^{\s-2s-1 + }  N_3^{1-  \frac 52 s  -}N_4^{1-}
& \les   N_1^{\s - \frac{13}{4} s  +}
\les  N_{\max}^{0-} , 
\end{align*}

\noi
provided that 
$ \s  < \frac{13}{4} s$
and 
$0 <  s < \frac 25$.
The modification for Subcase (1.b)
is straightforward
under the same 
regularity restriction.

\smallskip

\noi
{\bf $\bullet$  Case (2):} $N_1 \sim N_3 \sim N_{\max} \gg N_2$.
\newline
\noi
$\circ$ \underline{Subcase (2.a):} 
$N_2  \ll N_1^\frac{1}{2} $.
\quad 
It suffices to note that 
\begin{align*}
N_1^{\s - s } N_2^{1-s-}
N_3^{- \frac 52s-\frac 12 +}
%
% N_1^{\s-s-\frac 12 + }N_2^{1-s - }  N_3^{-  \frac 52s -\frac 12+}N_4^{1-}
 \les  N_1^{\s - 4 s  +}
\les  N_{\max}^{0-} , 
\end{align*}

\noi
provided that 
$\s < 4s$ and $0 \leq  s < 1$.\footnote{The lower bound on $s$ is needed only for Subcase (2.b).
Similar comments apply to the following cases and also to  the proof of Proposition \ref{PROP:Zk}.}
The modification for Subcase (2.b) 
is straightforward
under the same regularity restriction.

\smallskip

\noi
{\bf $\bullet$  Case (3):} $N_1 \sim N_4 \sim N_{\max} \gg N_2, N_3$.
\newline
\noi
$\circ$ \underline{Subcase (3.a):} 
$N_2, N_3  \ll N_1^\frac{1}{2} $.
\quad 
It suffices to note that  
\begin{align*}
N_1^{\s-s-\frac 12 + }N_2^{1-s - }  N_3^{1-  \frac 52s -}N_4^{-\frac{1}{2}+}
 \les  N_1^{\s - \frac {11}4 s  +}
\les  N_{\max}^{0-} , 
\end{align*}

\noi
provided that 
\begin{align}
 \s  < \frac {11}4s
 \qquad \text{and}\qquad 
0 <  s < \frac 25.
\label{QQ3}
\end{align}

\noi
The modifications for Subcases (3.b), (3.c),  and (3.d)
are straightforward
under the regularity restriction \eqref{QQ3}.

\smallskip

\noi
{\bf $\bullet$ Case (4):} $N_3 \sim N_4 \gg N_1 \geq  N_2$.
\newline
\noi
$\circ$ \underline{Subcase (4.a):} 
$N_1, N_2  \ll N_3^\frac{1}{2} $.
\quad 
It suffices to note that 
\begin{align*}
N_1^{1-s- }N_2^{1-s - }  N_3^{\s-  \frac 52 s -\frac 12 +}N_4^{-\frac{1}{2}+}
 \les   N_1^{\s - \frac 72 s  +}
\les  N_{\max}^{0-} , 
\end{align*}

\noi
provided that $\s < \frac 72 s$ and $0 \leq s < 1$.
The modifications for Subcases (4.b) and  (4.c)
are straightforward
under the same regularity restriction.

\smallskip

Putting all the cases together, 
we conclude that \eqref{Q4} holds
under the regularity restriction~\eqref{QQ3}.
This completes the proof of Lemma \ref{LEM:Z7}.
\end{proof}

\subsection{On the ``unbalanced'' higher order terms $\zeta_{2k-1}$}\label{SUBSEC:unbalanced}

In this subsection, 
we study the regularity properties of the
unbalanced higher order terms $\zeta_{2k-1}$ defined in \eqref{zeta2}.
Note that $(j_1, j_2, j_3) = (1, 1, 2k-3)$ up to permutations.
Then, by dropping the complex conjugate, we have
\begin{align}
\zeta_{2k-1}\sim    \int_0^t S(t - t') z_{1} z_1 \zeta_{2k-3} (t') dt'.
\label{QQQ2}
\end{align}

We first present the proof of Proposition \ref{PROP:Zk}.

\begin{proof}[Proof of Proposition \ref{PROP:Zk}]

We proceed by induction.
From \eqref{intro1}, 
we see that the recursive relation \eqref{al1} is satisfied
when $k = 2, 3, 4$.
In the following, by assuming
\begin{align*}
  \zeta_{2k -3} \in X^\s_\textup{loc} 
\end{align*}

\noi
for $\s < \al_{k-1} \cdot s $  almost surely,
where $\al_{k-1}$ satisfies \eqref{al2}, 
we  prove \eqref{al1} and \eqref{QQQ3a}. 
As in the proof of Lemma \ref{LEM:Z7} (ii), the proof follows from the proof of Lemma \ref{LEM:Z5}
by replacing $z_3$ and $2s-$ with  $\zeta_{2k-3}$ and  $\al_{k-1} \cdot s-$, respectively.
Therefore, 
we only indicate the necessary modifications
on the powers of dyadic parameters
in the proof of Lemma \ref{LEM:Z5}.

\smallskip

\noi
{\bf $\bullet$  Case (1):} $N_1 \sim N_2 \sim N_{\max}$.
\newline
\noi
$\circ$ \underline{Subcase (1.a):} 
$N_3 \ll N_1^\frac{1}{2}$.
\quad It suffices to note that 
\begin{align*}
N_1^{\s - s - \frac 12+} N_2^{-s}
N_3^{1 -  \al_{k-1} s-}
%
%N_1^{\s-2s-1 + }  N_3^{1-  \al_{k-1} s  -}N_4^{1-}
& \les   N_1^{\s - \frac{\al_{k-1}+4}{2} s  +}
\les  N_{\max}^{0-} , 
\end{align*}

\noi
provided that 
\begin{align}
 \s  < \frac{\al_{k-1}+4}{2} s
 \qquad \text{and}
 \qquad 0 < s < \al_{k-1}^{-1}.
\label{QQQ6}
\end{align}

\noi
The modification for Subcase (1.b)
is straightforward,
giving the regularity restriction~\eqref{QQQ6}.

\smallskip

\noi
{\bf $\bullet$  Case (2):} $N_1 \sim N_3 \sim N_{\max} \gg N_2$.
\newline
\noi
$\circ$ \underline{Subcase (2.a):} 
$N_2 \ll N_1^\frac{1}{2} $.
\quad 
It suffices to note that 
\begin{align*}
N_1^{\s - s } N_2^{1-s-}
N_3^{- \al_{k-1}s-\frac 12 +}
%
% N_1^{\s-s-\frac 12 + }N_2^{1-s - }  N_3^{-  \al_{k-1} s -\frac 12+}N_4^{1-}
 \les  N_1^{\s - \frac{2 \al_{k-1} + 3}{2} s  +}
\les  N_{\max}^{0-} , 
\end{align*}

\noi
provided that 
\begin{align}
 \s  <  \frac{2\al_{k-1}+3}{2}s
  \qquad \text{and}
 \qquad 0 \leq s < 1.
\label{QQQ7}
\end{align}

\noi
The modification for Subcase (2.b)
is straightforward,
giving the regularity restriction~\eqref{QQQ7}.

\smallskip

\noi
{\bf $\bullet$  Case (3):} $N_1 \sim N_4 \sim N_{\max} \gg N_2, N_3$.
\newline
\noi
$\circ$ \underline{Subcase (3.a):} 
$N_2, N_3  \ll N_1^\frac{1}{2} $.
\quad 
It suffices to note that  
\begin{align*}
N_1^{\s-s-\frac 12 + }N_2^{1-s - }  N_3^{1-  \al_{k-1}s -}N_4^{-\frac{1}{2}+}
 \les  N_1^{\s - \frac {\al_{k-1}+3}2 s  +}
\les  N_{\max}^{0-} , 
\end{align*}

\noi
provided that 
\begin{align}
 \s  < \frac {\al_{k-1}+3}2s
  \qquad \text{and}
 \qquad 0 < s < \al_{k-1}^{-1}.
\label{QQQ8}
\end{align}

\noi
The modifications for Subcases (3.b), (3.c),  and (3.d)
are straightforward,
giving the regularity restriction \eqref{QQQ8}.

\smallskip

\noi
{\bf $\bullet$ Case (4):} $N_3 \sim N_4 \gg N_1 \geq  N_2$.
\newline
\noi
$\circ$ \underline{Subcase (4.a):} 
$N_1, N_2  \ll N_3^\frac{1}{2} $.
\quad 
It suffices to note that 
\begin{align*}
N_1^{1-s- }N_2^{1-s - }  N_3^{\s-  \al_{k-1} s -\frac 12 +}N_4^{-\frac{1}{2}+}
 \les   N_1^{\s - (\al_{k-1}+1) s  +}
\les  N_{\max}^{0-} , 
\end{align*}

\noi
provided that 
\begin{align}
 \s  < (\al_{k-1}+1)s 
 \qquad \text{and}
 \qquad 0 \leq s < 1.
\label{QQQ9}
\end{align}

\noi
The modifications for Subcases (4.a) and  (4.b)
are straightforward,
giving the regularity restriction \eqref{QQQ9}.

Since $\al_{k-1}$ satisfies \eqref{al2}, 
 we have $\al_{k-1} \geq 1$,
 which in turn implies 
 $\frac {\al_{k-1}+3}2 \leq \al_{k-1} + 1$.
 Therefore, 
we conclude \eqref{al1} and \eqref{QQQ3a}.
\end{proof}

We conclude this section by 
proving a gain of space-time  integrability for $\zeta_{2k-1}$
analogous to Lemma \ref{LEM:Z3_2}.

\begin{lemma}\label{LEM:Zk_2}
Let $s \geq 0$.
Given  $\phi \in H^s(\R^3)$,  let $\phi^\o$ be its Wiener randomization
defined in~\eqref{rand}.
Fix an integer $k \geq 3$.
Then, for any finite $q, r > 2$, 
there exist $\theta_k = \theta_k(q, r) > 0$, 
$\dl_k  = \dl_k(q, r) \in (0, 1)$, and  
$1\leq q_k< \infty$ such that 
\begin{align}
\| \P_N \zeta_{2k-1}\|_{L^q_T L^r_x}
\les \begin{cases}
\rule[-6mm]{0pt}{15pt} T^{\theta_k} \| z_1\|_{L^{q_k}_T L^{2}_x}^{(2k-1)(1-\dl_k)} 
\| z_1\|_{L^{q_k}_T L^{6}_x}^{(2k-1)\dl_k} 
& \text{when } 1 \leq r < 6, \\
 T^{\theta_k} N^{\frac{1}{2} - \frac 3 r+}   \| z_1\|_{L^{q_k}_T L^{2}_x}^{(2k-1)(1-\dl_k)} 
\| z_1\|_{L^{q_k}_T L^{6}_x}^{(2k-1)\dl_k} 
&  \text{when }  r \geq  6, \\
\end{cases}
\label{QQQ10}
\end{align}

\noi
for any $T > 0$ and $N \in 2^{\NB_0}$.	
Note that 
 the right-hand side of \eqref{QQQ10} is almost surely finite thanks 
to the probabilistic Strichartz estimate (Lemma \ref{LEM:PStr}).
\end{lemma}

\begin{proof}
We prove \eqref{QQQ10} by induction.
We first consider the case $r < 6$.
In Lemma \ref{LEM:Z3_2}, 
we proved \eqref{QQQ10}
for $k = 2$.
Now, suppose that \eqref{QQQ10} holds for  $k - 1$.
Then, 
from \eqref{QQQ2} and   the dispersive estimate \eqref{O9}, 
we have
\begin{align}
\| \P_N \zeta_{2k-1}\|_{L^q_T L^r_x}
& \leq \bigg\|\int_0^t \|\P_N S(t - t') z_1^2 \zeta_{2k-3} (t')\|_{L^r_x} dt'\bigg\|_{L^q_T}\notag \\
& \les \bigg\|\int_0^t \frac{1}{|t - t'|^{\frac{3}{2} - \frac 3 r}}
 \| z_1\|_{L^{3r'}_x}^2 \| \zeta_{2k-3}\|_{L^{3r'}_x} dt'\bigg\|_{L^q_T}\notag \\
\intertext{Noting  $3r' < 6$
and applying the inductive hypothesis
with $\theta_{k-1} = \theta_{k-1}(3q, 3r')$, 
 $\dl_{k-1} = \dl_{k-1}(3q, 3r')$,
and $q_{k-1} = q_{k-1}(3q, 3r')$,}
& \les T^{\frac 3 r - \frac 12 + \theta_{k-1} }\| z_1\|_{L^{3q}_T L^{3r'}_x}^2
 \| z_1\|_{L^{q_{k-1}}_T L^{2}_x}^{(2k-3)(1-\dl_{k-1})} 
\| z_1\|_{L^{q_{k-1}}_T L^{6}_x}^{(2k-3)\dl_{k-1}} \notag\\
\intertext{By interpolation in $x$ and H\"older's inequality in $t$,}
& \les T^{\theta_k }
 \| z_1\|_{L^{q_k}_T L^{2}_x}^{(2k-1)(1-\dl_{k})} 
\| z_1\|_{L^{q_k}_T L^{6}_x}^{(2k-1)\dl_{k}} .
\label{QQQ11}
\end{align}

\noi
When $r \geq 6$, 
\eqref{QQQ10} follows from 
 Sobolev's inequality and \eqref{QQQ11}
 as in the proof of Lemma~\ref{LEM:Z3_2}.
\end{proof}

\section{Proof of Theorem \ref{THM:1}}

\label{SEC:NL1}

In this section, 
we study the fixed point problem \eqref{NLS5} around the second order expansion \eqref{v2}
and present the proof of Theorem \ref{THM:1}.
Given $ \frac 12 \leq  \s\leq  1$, 
let   $ \frac 25 \s  < s < \frac 12 $.
Given $\phi \in H^s(\R^3)$, let $\phi^\o$
be its Wiener randomization defined in \eqref{rand}
and let $z_1$ and $z_3$ be as in \eqref{introZ1} and~\eqref{introZ3}.
Define $\G$ by 
\begin{equation*}
\G v(t) = - i \int_0^t S(t-t')\big\{\N(v + z_1 + z_3) - \N(z_1)\big\}(t') dt'.
%\label{NLSx1}
\end{equation*}

\noi
Then, we have the following nonlinear estimates.

\begin{proposition}\label{PROP:NL1}
Given $ \frac 12 \leq  \s\leq 1$, 
let   $\frac 25 \s < s < \frac{1}{2}$.
Then,  there exist $\theta > 0$,  $C_1, C_2> 0$, 
and an almost surely finite constant $R = R(\o) >0$ such that 
\begin{align}
\|\G v\|_{X^\s([0, T])}
& \leq C_1
\big(\|v\|_{X^\s([0, T])} ^3 + T^\theta R(\o)\big),  \label{nl1a}\\
\|\G v_1 - \G v_2  \|_{X^\s([0, T])}
& \leq C_2
\Big(\sum_{j = 1}^2 \|v_j\|_{X^\s([0, T])} ^2
+ T^\theta R(\o)\Big)
\|v_1 -v_2 \|_{X^\s([0, T])},
\label{nl1b}
\end{align}

\noi
for all $v, v_1, v_2 \in X^{\s}([0, T])$
and $0 < T\leq 1$.

\end{proposition}

Once we prove Proposition \ref{PROP:NL1}, 
Theorem \ref{THM:1}
immediately follows from a standard argument
and thus we omit details.
See Section 5 in \cite{BOP2}.

\begin{proof}[Proof of Proposition \ref{PROP:NL1}]
Let $0<T \leq 1$.
We only prove \eqref{nl1a} since  \eqref{nl1b} follows in a  similar manner.
Arguing as in the proof of Proposition 4.1 in \cite{BOP2}, 
it suffices to perform 
a case-by-case analysis of expressions of the form:
\begin{align*}
\bigg| \int_0^T \int_{\R^3}
\jb{\nb}^\s ( w_1 w_2 w_3 )w dx dt\bigg|,
%\label{nl2}
\end{align*}

\noi
where $\|w\|_{Y^{0}_T} \leq 1$
and $w_j=  v$, $z_1$, or $z_3$,  $j = 1, 2, 3$, 
but not all $z_1$.
Note that we have dropped  the complex conjugate sign
on $w_2$ 
since it does not play any essential role.
Then,  we need to consider the following cases:
\begin{center}
\begin{tabular}{clclcl}
(A)& $z_1 z_1 z_3$
\qquad\qquad  
& \text{(D)}& $v v z_1$             
\qquad \qquad 
& \text{(G)}& $v z_3 z_3$            
\rule[-3mm]{0pt}{0pt} 
 \\
 \text{(B)}& $z_1 z_3 z_3$     
\qquad \qquad
& \text{(E)}& $v v z_3$
\qquad \qquad
& \text{(H)}& $v z_3 z_1$
\rule[-3mm]{0pt}{0pt} 

\\
 \text{(C)}& $z_3 z_3 z_3$
\qquad \qquad
& \text{(F)}& $v z_1 z_1$
\qquad \qquad
& \text{(\,I\,)}& $v v v$
%\rule[-3mm]{0pt}{0pt} 

\end{tabular}
\end{center}

\noi
We already treated 
Cases (A) and (B) in 
Lemmas \ref{LEM:Z5} and \ref{LEM:Z7}.
In particular, Case (A) imposes the regularity restriction:
\[ \s < \frac 52 s.\]

\noi
As we see below, 
 Cases (B) - (I) 
only impose  a milder regularity restriction:
$\s < 3s$.

Given $ \frac 12 \leq  \s\leq 1$, 
let  $\frac 25 \s < s < \frac{1}{2}$.
Define the $B^s(T)$-norm by 
\begin{align*}
 \|z_1\|_{B^s(T)}: = \max\Big(
 \| \jb{\nb}^s z_{1} \|_{L^{5+}_T L^{5}_x}, 
& \| \jb{\nb}^s z_{1} \|_{L^{10+}_T L^{10}_x}, \\
&  \|\jb{\nb}^s z_1\|_{L^{3q}_T L^{\frac{18}{5}+}_x}, 
\| \jb{\nb}^s z_1\|_{L^{15}_T L^{\frac{15}{4}}_x}, 
\| z_1(0)\|_{H^s}
\Big), 
\end{align*}

\noi
where $q \gg 1$ is defined in \eqref{XX2} below.
Then, by applying the dyadic decomposition, we prove 
the following estimate:\footnote{In Case (I), we do not perform the dyadic decomposition
and hence there is no need to have the factor $N_{\max}^{0-}$ on the right-hand side of \eqref{XX1}.}
\begin{align}
\bigg| \int_0^T \int_{\R^3}
N_{\max}^\s ( \P_{N_1} w_1 & \P_{N_2}  w_2 \P_{N_3} w_3 ) \P_{N_4} w \, dx dt\bigg|\notag \\
& \leq C_1
 N_{\max}^{0-}
\big(\|v\|_{X^\s([0, T])} ^3 + T^\theta
C( \|z_1\|_{B^s(T)}, \|z_3\|_{X^{2s-}(T)})\big) 
\label{XX1}
\end{align}

\noi
for all $N_1, \dots, N_4 \in 2^{\NB_0}$.
Once we prove~\eqref{XX1}, the desired estimate~\eqref{nl1a}
follows from summing~\eqref{XX1} over dyadic blocks and applying Lemmas \ref{LEM:PStr} and  \ref{LEM:Hs}
and Proposition \ref{PROP:Z3}.

\smallskip
\noi
{\bf  Case (C):} $z_3z_3z_3 $ case.

By symmetry, we assume $N_1 \leq N_2 \leq N_3$.
Note that we have $N_3 \sim N_{\max}$.

\smallskip

\noi
{\bf $\bullet$  Subcase (C.1):} $N_2 \sim N_3 \sim N_{\max}$.

\quad 
By H\"older's inequality with
\begin{align}
\begin{cases}
\text{time: } \frac{1}{q} + \frac{1}{2+} + \frac{1}{2+} + \frac{1}{\infty} =1  
\text{ for some }q \gg 1, \rule[-3mm]{0pt}{0pt}\\
\text{space: } 
\frac{1}{6+} + \frac{1}{6-} + \frac{1}{6-} + \frac{1}{2} =1
\end{cases}
\label{XX2}
\end{align}

\noi
such that $(2+, 6-)$ is admissible and 
applying Lemmas \ref{LEM:Str} and  \ref{LEM:Z3_2}, 
we have
\begin{align*}
& \text{LHS of }  \eqref{XX1}
 \les N_3^\s \|  z_{3, N_1} \|_{L^{q}_T L^{6+}_x}
\bigg(\prod_{j = 2}^3   \|   z_{3, N_j} \|_{L^{2+}_T L^{6-}_x}\bigg)
 \| w_{N_4} \|_{L^\infty_T L^2_x}\notag\\
& \hphantom{X}
\les T^{0+}
N_1^{ -s +} N_2^{-2s+}N_3^{\s -2s+} 
\|\jb{\nb}^s z_1\|_{L^{3q}_T L^{\frac{18}{5}+}_x}
\| z_1\|_{L^{3q}_T L^{\frac{18}{5}+}_x}^2
\bigg(\prod_{j = 2}^3
 \|z_{3, N_j}\|_{Y^{2s-}_T}\bigg)
 \|w_{N_4}\|_{Y^0_T}\\
 \notag\\
 & \hphantom{X}
\le T^{0+}
N_3^{\s-4s+} 
 C(\|z_1\|_{B^s(T)}, \|z_{3}\|_{X^{2s-}_T}).
\end{align*}

\noi
This yields  \eqref{XX1}, 
provided that 
$\s < 4s$ and $ s > 0$.

\smallskip

\noi
{\bf $\bullet$  Subcase (C.2):} $N_3 \sim N_4 \sim N_{\max} \gg N_2 \geq N_1$.

\noi
$\circ$ \underline{Subsubcase (C.2.a):} 
$N_1, N_2 \ll N_3^\frac{1}{2}$.
\quad 
By the bilinear Strichartz estimate, 
 we have
\begin{align*}
\text{LHS of }   \eqref{XX1}
&    \les N_3^\s\|z_{3, N_1}   z_{3, N_3} \|_{L^2_{T, x}} \|z_{3, N_2}  w_{N_4}\|_{L^2_{T, x}}\\
& \les T^{0+} N_1^{1-2s- }  N_2^{1-2s -}N_3^{\s-  2s  -\frac 12+}N_4^{-\frac{1}{2}+}
\bigg(\prod_{j = 1}^3
 \|z_{3, N_j}\|_{Y^{2s-}_T}\bigg)\\
& \les T^{0+}  N_3^{\s - 4s  +}
 C( \|z_3\|_{X^{2s-}_T}).
\end{align*}

\noi
This yields \eqref{XX1}, 
provided that $\s < 4s$ and $s < \frac 12$.

\smallskip

\noi
$\circ$ \underline{Subsubcase (C.2.b):} $N_1, N_2 \ges N_3^\frac{1}{2} $.
\quad 
Proceeding as in Subcase (C.1), 
we have
\begin{align*}
\text{LHS of } \eqref{XX1}
& \les N_3^\s \|  z_{3, N_1} \|_{L^{q}_T L^{6+}_x}
\bigg(\prod_{j = 2}^3   \|   z_{3, N_2} \|_{L^{2+}_T L^{6-}_x}\bigg)
 \| w_{N_4} \|_{L^\infty_T L^2_x}\notag\\
& \les  T^{0+}
N_1^{ -s +} N_2^{-2s+}N_3^{\s -2s+} 
\| \jb{\nb}^s z_1\|_{L^{3q}_T L^{\frac{18}{5}+}_x}
\| z_1\|_{L^{3q}_T L^{\frac{18}{5}+}_x}^2
\bigg(\prod_{j = 2}^3
 \|z_{3, N_j}\|_{Y^{2s-}_T}\bigg)\\
& \le T^{0+}
N_3^{\s-\frac{7}{2} s+} 
 C(\|z_1\|_{B^s(T)}, \|z_{3}\|_{X^{2s-}_T}).
\end{align*}

\noi
This yields \eqref{XX1}, 
provided that $\s < \frac 72 s$ and $s > 0$.

\smallskip

\noi
$\circ$ \underline{Subsubcase (C.2.c):} $N_2\ges N_3^\frac{1}{2} \gg N_1$.
\quad 
By 
Lemmas~\ref{LEM:Str}
and \ref{LEM:biStr}, 
 we have
\begin{align*}
\text{LHS of } \eqref{XX1}
&  \les N_1^\s     \|z_{3, N_1}  z_{3, N_3} \|_{L^2_{T, x}}
\| z_{3, N_2} \|_{L^{5}_{T, x}}
 \| w_{N_4} \|_{L^\frac{10}{3}_{T, x}}\notag\\
& \le
T^{0+}
N_1^{1 - 2s -} N_2^{-s}
N_3^{\s - 2s- \frac 12+}
C( \|z_{3}\|_{X^{2s-}_T})
 \|\jb{\nb}^s z_{3, N_2} \|_{L^5_{T, x}}
\\
\intertext{By applying Lemma   \ref{LEM:Z3_2},}
& \le
T^{\frac{1}{10}+}
N_1^{1 - 2s -} N_2^{-s}
N_3^{\s - 2s- \frac 12+}
C( \|z_{3}\|_{X^{2s-}_T})
\| \jb{\nb}^s z_1\|_{L^{15}_T L^{\frac{15}{4}}_x}
\| z_1\|_{L^{15}_T L^{\frac{15}{4}}_x}^2\\
& \le T^{\frac{1}{10}+}
N_3^{\s-\frac 72 s+} 
 C(\|z_1\|_{B^s(T)}, \|z_{3}\|_{X^{2s-}_T}).
\end{align*}

\noi
This yields \eqref{XX1}, 
provided that $\s < \frac 72 s$
and $ 0 \leq s < \frac 12$.

\medskip
\noi
{\bf  Case (D):} $v v z_1$ case.
\newline
\indent
By symmetry, we assume $N_1 \geq N_2$.

\smallskip

\noi
{\bf $\bullet$  Subcase (D.1):} $N_1 \ges N_3$.
\newline
\indent
In this case, we have $N_1 \sim N_{\max}$.

\noi
$\circ$ \underline{Subsubcase (D.1.a):} 
$\max(N_2, N_3) \geq N_1^\frac{1}{10}$.
\quad
By H\"older's inequality and Lemma \ref{LEM:Str},
we have
\begin{align*}
\text{LHS of } \eqref{XX1}
& \les N_1^\s \| v_{N_1} \|_{L^\frac{10}{3}_{T, x}}
\| v_{N_2} \|_{L^\frac{10}{3}_{T, x}}
\| z_{1, N_3} \|_{L^{10}_{T, x}}
 \| w_{N_4} \|_{L^\frac{10}{3}_{T, x}}\notag\\
& \leq T^{0+} N_2^{-\s} N_3^{-s} C(\|z_1\|_{B^s(T)}) \| v\|_{X^\s_T}^2\\
& \leq T^{0+} N_{\max}^{0-} C(\|z_1\|_{B^s(T)}) \| v\|_{X^\s_T}^2, 
\end{align*}

\noi
provided that $\s > 0$ and $ s > 0$.

\smallskip

\noi
$\circ$ \underline{Subsubcase (D.1.b):} 
$\max(N_2, N_3) \ll N_1^\frac{1}{10}$.
\quad
In this case, we have $N_1\sim N_4 \sim N_{\max}$.
By  Lemmas \ref{LEM:Str} and \ref{LEM:biStr},
 we have
\begin{align*}
\text{LHS of } \eqref{XX1}
&  \les N_1^\s 
 \|v_{N_1}\|_{L^\frac{10}{3}_{T, x}}
 \|z_{1, N_3}\|_{L^5_{T, x}}
 \|v_{ N_2}  w_{N_4} \|_{L^2_{T, x}}
\\
& \le
T^{0+}
 N_2^{1-\s-}
N_3^{-s}N_4^{-\frac{1}{2}+}
C(\|z_1\|_{B^s(T)}) \| v\|_{X^\s_T}^2\\
& \leq 
T^{0+} N_{\max}^{0-} C(\|z_1\|_{B^s(T)}) \| v\|_{X^\s_T}^2,
\end{align*}

\noi
provided that $ \s \geq 0$ and $ s \geq 0$.

\smallskip

\noi
{\bf $\bullet$ Subcase (D.2):} $N_3 \sim N_4 \gg N_1\geq  N_2$.
\newline
\noi
$\circ$ \underline{Subsubcase (D.2.a):} 
$N_1, N_2 \ll N_3^\frac{1}{2}$.
\quad
By   the bilinear Strichartz estimate, 
 we have
\begin{align*}
\text{LHS of } \eqref{XX1}
&    \les N_3^\s \|v_{N_1}   z_{1, N_3} \|_{L^2_{T, x}} \|v_{N_2}  w_{N_4}\|_{L^2_{T, x}}\\
& \les T^{0+} N_1^{1-\s-} N_2^{1-\s-} N_3^{\s - s - \frac{1}{2} +}N_4^{-\frac{1}{2}+}
\bigg(\prod_{j = 1}^2 \| v_{N_j}\|_{Y^\s_T}\bigg)
\|\P_{N_3} \phi^\o\|_{H^{s}}\\
& \le T^{0+}  N_3^{ - s  +}
C(\|z_1\|_{B^s(T)}) \| v\|_{X^\s_T}^2\\
& \leq 
T^{0+} N_{\max}^{0-} C(\|z_1\|_{B^s(T)}) \| v\|_{X^\s_T}^2,
\end{align*}

\noi
provided that $\s \leq 1$ and  $s>0$.

\smallskip

\noi
$\circ$ \underline{Subsubcase (D.2.b):} $N_1, N_2\ges N_3^\frac{1}{2} $.
\quad 
Proceeding as in Subsubcase (D.1.a) with 
$L^\frac{10}{3}_{T, x}, 
L^\frac{10}{3}_{T, x}, 
L^{10}_{T, x}, L^\frac{10}{3}_{T, x}$-H\"older's inequality, it suffices to note that 
\begin{align*}
N_1^{-\s} N_2^{-\s}N_3^{\s-s} 
\les N_3^{ - s  }
\les  N_{\max}^{0-},
\end{align*}

\noi
provided that $\s \geq 0$ and $ s > 0$.

\smallskip

\noi
$\circ$ \underline{Subsubcase (D.2.c):} $N_1\ges N_3^\frac{1}{2} \gg N_2$.
\quad 
By  Lemmas \ref{LEM:Str} and \ref{LEM:biStr},
 we have
\begin{align*}
\text{LHS of } \eqref{XX1}
&  \les N_3^\s 
 \|v_{N_1}\|_{L^\frac{10}{3}_{T, x}}
 \|z_{1, N_3}\|_{L^5_{T, x}}
 \|v_{ N_2}  w_{N_4} \|_{L^2_{T, x}}
\\
& \le
T^{0+}
N_1^{-\s} N_2^{1-\s-}
N_3^{\s-s}N_4^{-\frac{1}{2}+}
C(\|z_1\|_{B^s(T)}) \| v\|_{X^\s_T}^2\\
& \leq 
T^{0+} N_{\max}^{0-} C(\|z_1\|_{B^s(T)}) \| v\|_{X^\s_T}^2,
\end{align*}

\noi
provided that $0\leq \s \leq 1$ and $ s > 0$.

\medskip
\noi
{\bf  Case (E):} $v v z_3$ case.
\newline
\indent
By symmetry, we assume $N_1 \geq N_2$.

\smallskip

\noi
{\bf $\bullet$  Subcase (E.1):} $N_1 \ges N_3$.
\newline
\indent
In this case, we have $N_1 \sim N_{\max}$.

\noi
$\circ$ \underline{Subsubcase (E.1.a):} 
$\max(N_2, N_3) \geq N_1^\frac{1}{10}$.
\quad 
By H\"older's inequality with \eqref{XX2} as in Subcase (C.1) 
and Lemmas~\ref{LEM:Str} and  \ref{LEM:Z3_2}, 
we have
\begin{align*}
\text{LHS of } \eqref{XX1}
& \les N_1^\s 
\bigg(\prod_{j = 1}^2   \|   v_{N_j} \|_{L^{2+}_T L^{6-}_x}\bigg)
\|  z_{3, N_3} \|_{L^{q}_T L^{6+}_x}
 \| w_{N_4} \|_{L^\infty_T L^2_x}\notag\\
& \les T^{0+}
 N_2^{-\s}N_3^{-s+} 
\| v\|_{X^\s_T}^2
\| \jb{\nb}^s z_1\|_{L^{3q}_T L^{\frac{18}{5}+}_x}
\| z_1\|_{L^{3q}_T L^{\frac{18}{5}+}_x}^2\\
 \notag\\
& \leq 
T^{0+} N_{\max}^{0-} C(\|z_1\|_{B^s(T)}) \| v\|_{X^\s_T}^2,
\end{align*}

\noi
provided that 
$\s > 0$ and $s > 0$.

\smallskip

\noi
$\circ$ \underline{Subsubcase (E.1.b):} 
$\max(N_2, N_3) \ll N_1^\frac{1}{10}$.
\quad
In this case, we have $N_1\sim N_4 \sim N_{\max}$.
By  Lemmas \ref{LEM:Str}, \ref{LEM:biStr}, and \ref{LEM:Z3_2}, 
 we have
\begin{align*}
\text{LHS of } \eqref{XX1}
&  \les N_1^\s 
 \|v_{N_1}\|_{L^\frac{10}{3}_{T, x}}
 \|z_{3, N_3}\|_{L^5_{T, x}}
 \|v_{ N_2}  w_{N_4} \|_{L^2_{T, x}}
\\
& \les
T^{\frac{1}{10}+}
 N_2^{1-\s-}
N_3^{-s}N_4^{-\frac{1}{2}+}
 \| v\|_{X^\s_T}^2\| 
 \jb{\nb}^s z_1\|_{L^{15}_T L^{\frac{15}{4}}_x}
\| z_1\|_{L^{15}_T L^{\frac{15}{4}}_x}^2\\
& \leq 
T^{\frac{1}{10}+} N_{\max}^{0-} C(\|z_1\|_{B^s(T)}) \| v\|_{X^\s_T}^2,
\end{align*}

\noi
provided that $ \s \geq 0$ and $ s \geq 0$.

\smallskip

\noi
{\bf $\bullet$ Subcase (E.2):} $N_3 \sim N_4 \gg N_1\geq  N_2$.
\newline
\noi
$\circ$ \underline{Subsubcase (E.2.a):} 
$N_1, N_2 \ll N_3^\frac{1}{2}$.
\quad
\noi
In this case, we proceed as in Subsubcase (D.2.a).
It suffices to note that 
\begin{align*}
N_1^{1-\s-} N_2^{1-\s-} N_3^{\s - 2s - \frac{1}{2} +}N_4^{-\frac{1}{2}+}
\les   N_3^{ - 2s  +}
\les  N_{\max}^{0-}, 
\end{align*}

\noi
provided that $\s \leq 1$ and  $s>0$.

\smallskip

\noi
$\circ$ \underline{Subsubcase (E.2.b):} $N_1, N_2\ges N_3^\frac{1}{2} $.
\quad 
By H\"older's inequality with~\eqref{XX2} as in Subcase~(C.1) 
and Lemmas~\ref{LEM:Str} and  \ref{LEM:Z3_2}, 
we have
\begin{align*}
\text{LHS of } \eqref{XX1}
& \les N_3^\s 
\bigg(\prod_{j = 1}^2   \|   v_{N_j} \|_{L^{2+}_T L^{6-}_x}\bigg)
\|  z_{3, N_3} \|_{L^{q}_T L^{6+}_x}
 \| w_{N_4} \|_{L^\infty_T L^2_x}\notag\\
& \les T^{0+}
N_1^{-\s}N_2^{-\s} N_3^{\s-s+}
\| v\|_{X^\s_T}^2
\| \jb{\nb}^s z_1\|_{L^{3q}_T L^{\frac{18}{5}+}_x}
\| z_1\|_{L^{3q}_T L^{\frac{18}{5}+}_x}^2\\
 \notag\\
& \leq 
T^{0+} N_{\max}^{0-} C(\|z_1\|_{B^s(T)}) \| v\|_{X^\s_T}^2,
\end{align*}

\noi
provided that 
$\s \geq  0$ and $s > 0$.

\smallskip

\noi
$\circ$ \underline{Subsubcase (E.2.c):} $N_1\ges N_3^\frac{1}{2} \gg N_2$.
\quad 
By  Lemmas \ref{LEM:Str},   \ref{LEM:biStr}, and \ref{LEM:Z3_2}, 
 we have
\begin{align*}
\text{LHS of } \eqref{XX1}
&  \les N_3^\s 
 \|v_{N_1}\|_{L^\frac{10}{3}_{T, x}}
 \|z_{3, N_3}\|_{L^5_{T, x}}
 \|v_{ N_2}  w_{N_4} \|_{L^2_{T, x}}
\\
& \le
T^{\frac{1}{10}}
N_1^{-\s} N_2^{1-\s-}
N_3^{\s-s}N_4^{-\frac{1}{2}+}
 \| v\|_{X^\s_T}^2
\| \jb{\nb}^s z_1\|_{L^{15}_T L^{\frac{15}{4}}_x}
\| z_1\|_{L^{15}_T L^{\frac{15}{4}}_x}^2\\
& \leq 
T^{\frac{1}{10}} N_{\max}^{0-} C(\|z_1\|_{B^s(T)}) \| v\|_{X^\s_T}^2,
\end{align*}

\noi
provided that $0\leq \s \leq 1$ and $ s > 0$.

\medskip
\noi
{\bf  Case (F):} $v z_1 z_1$ case.
\newline
\indent
By symmetry, we assume $N_3 \geq N_2$.

\smallskip

\noi
{\bf $\bullet$  Subcase (F.1):} $N_1 \ges N_3$.
\newline
\indent
In this case, we have $N_1 \sim N_{\max}$.

\noi
$\circ$ \underline{Subsubcase (F.1.a):} 
$\max(N_2, N_3) \geq N_1^\frac{1}{10}$.
\quad 
By  Lemma \ref{LEM:Str},
we have
\begin{align*}
\text{LHS of } \eqref{XX1}
& \les N_1^\s \| v_{N_1} \|_{L^\frac{10}{3}_{T, x}}
\| z_{1, N_2} \|_{L^5_{T, x}}
\| z_{1, N_3} \|_{L^{5}_{T, x}}
 \| w_{N_4} \|_{L^\frac{10}{3}_{T, x}}\notag\\
& \leq T^{0+} N_2^{-s} N_3^{-s} C(\|z_1\|_{B^s(T)}) \| v\|_{X^\s_T}\\
& \leq T^{0+} N_{\max}^{0-} C(\|z_1\|_{B^s(T)}) \| v\|_{X^\s_T},
\end{align*}

\noi
provided that  $ s > 0$. 

\noi
$\circ$ \underline{Subsubcase (F.1.b):} 
$\max(N_2, N_3) \ll N_1^\frac{1}{10}$.
\quad
In this case, we proceed as in Subsubcase (D.1.b).
It suffices to note that 
\begin{align*}
N_2^{1-s-} N_3^{-s}N_4^{-\frac{1}{2}+}
\les  N_{\max}^{0-}, 
\end{align*}

\noi
provided that  $ s \geq 0$.

\smallskip

\noi
{\bf $\bullet$ Subcase (F.2):} $N_2 \sim N_3 \gg N_1$.
\newline
\noi
$\circ$ \underline{Subsubcase (F.2.a):} $N_1\ll N_3^\frac{1}{2}$.
\quad 
By  Lemmas \ref{LEM:Str} and  \ref{LEM:biStr}, 
 we have
\begin{align*}
\text{LHS of } \eqref{XX1}
&  \les N_3^\s 
 \|v_{ N_1}  z_{1, N_2} \|_{L^2_{T, x}}
 \|z_{1, N_3}\|_{L^5_{T, x}}
 \|w_{N_4}\|_{L^\frac{10}{3}_{T, x}}
\\
& \le
T^{0+}
N_1^{1-\s-} N_2^{-s-\frac 12 + }
N_3^{\s-s}
C(\|z_1\|_{B^s(T)}) \| v\|_{X^\s_T}\\
& \le
T^{0+}
N_3^{\frac \s2-2s+}
C(\|z_1\|_{B^s(T)}) \| v\|_{X^\s_T}.
\end{align*}

\noi
This yields \eqref{XX1}, provided that $\s \leq \min (1, 4s-)$.

\smallskip

\noi
$\circ$ \underline{Subsubcase (F.2.b):} $N_1\ges N_3^\frac{1}{2} $.
\quad 
In this case, we proceed with 
$L^\frac{10}{3}_{T, x}, 
L^5_{T, x}, 
L^{5}_{T, x}, L^\frac{10}{3}_{T, x}$-H\"older's inequality.
It  suffices to note that 
\begin{align*}
N_1^{-\s} N_2^{-s}N_3^{\s-s} 
\les N_3^{ \frac \s2 - 2s }
\les  N_{\max}^{0-},
\end{align*}

\noi
provided that $0 \leq \s < 4s$.

\smallskip

\noi
{\bf $\bullet$  Subcase (F.3):} $N_3 \sim N_4 \sim N_{\max} \gg N_1, N_2$.
\newline
\noi
$\circ$ \underline{Subsubcase (F.3.a):} 
$N_1, N_2  \ll N_3^\frac{1}{2} $.
\quad 
By the bilinear Strichartz estimate, 
 we have
\begin{align*}
\text{LHS of }   \eqref{XX1}
&     \les N_3^\s\|v_{N_1}   z_{1, N_3} \|_{L^2_{T, x}} \|z_{1, N_2}  w_{N_4}\|_{L^2_{T, x}}\\
& \les T^{0+} N_1^{1- \s-}N_2^{1-s - }  N_3^{\s-  s -\frac 12 +}N_4^{-\frac{1}{2}+}
\|v\|_{X^\s_T}
\bigg(\prod_{j = 2}^3
 \|\P_{N_j}\phi^\o\|_{H^s}
\bigg)
\\
& \les T^{0+}  N_3^{\frac \s2 - \frac  32 s  +}
 C(\|z_1\|_{B^s(T)}) \| v\|_{X^\s_T}.
\end{align*}

\noi
This yields 
 \eqref{XX1}, 
provided that $ \s \leq \min(1, 3s-)$ and $ s < 1$.

\smallskip

\noi
$\circ$ \underline{Subcase (F.3.b):} $N_1, N_2\ges N_3^\frac{1}{2} $.
\quad 
In this case, we proceed  with 
$L^\frac{10}{3}_{T, x}, 
L^5_{T, x}, 
L^{5}_{T, x}, L^\frac{10}{3}_{T, x}$-H\"older's inequality.
It  suffices to note that 
\begin{align*}
N_1^{-\s} N_2^{-s}N_3^{\s-s} 
\les N_3^{ \frac{\s}2 - \frac{3}{2}s }
\les  N_{\max}^{0-},
\end{align*}

\noi
provided that $0 \leq \s < 3s$.

\smallskip

\noi
$\circ$ \underline{Subcase (F.3.c):} $N_1\ges N_3^\frac{1}{2} \gg N_2$.
\quad 
By Lemmas \ref{LEM:Str}  and  \ref{LEM:biStr},  
 we have
\begin{align*}
\text{LHS of } \eqref{XX1}
&  \les N_3^\s   \| v_{N_1} \|_{L^\frac{10}{3}_{T, x}} 
 \| z_{1, N_3} \|_{L^{5}_{T, x}} 
 \|z_{1, N_2}  w_{N_4} \|_{L^2_{T, x}}
 \notag\\
& \le T^{0+}
N_1^{-\s} N_2^{1-s-}
N_3^{\s- s}N_4^{-\frac{1}{2}+}
 C(\|z_1\|_{B^s(T)}) \| v\|_{X^\s_T}\\
& \le T^{0+}
N_3^{\frac \s2- \frac 32s+} 
 C(\|z_1\|_{B^s(T)}) \| v\|_{X^\s_T}.
\end{align*}

\noi
This
 yields \eqref{XX1}, 
provided that $0\leq \s < 3s$ and $s < 1$.

\smallskip

\noi
$\circ$ \underline{Subcase (F.3.d):} $N_2\ges N_3^\frac{1}{2} \gg N_1$.
\quad 
By Lemmas \ref{LEM:Str}  and  \ref{LEM:biStr},  
 we have
\begin{align*}
\text{LHS of } \eqref{XX1}
&  \les N_3^\s  
 \|v_{N_1}  z_{1, N_3} \|_{L^2_{T, x}}
  \| z_{1, N_2} \|_{L^{5}_{T, x}} 
 \| w_{N_4} \|_{L^\frac{10}{3}_{T, x}}
\\
& \les T^{0+}
N_1^{1-\s-} N_2^{-s}
N_3^{\s- s-\frac 12+}
 C(\|z_1\|_{B^s(T)}) \| v\|_{X^\s_T}\\
& \le T^{0+}
N_3^{\frac \s2- \frac 32s+} 
 C(\|z_1\|_{B^s(T)}) \| v\|_{X^\s_T}. 
\end{align*}

\noi
This yields \eqref{XX1}, 
provided that 
 $ \s \leq \min(1, 3s-)$ and $s\geq 0$.

\medskip
\noi
{\bf  Case (G):} $v z_3 z_3$ case.
\newline
\indent
By symmetry, we assume $N_3 \geq N_2$.

\smallskip

\noi
{\bf $\bullet$  Subcase (G.1):} $N_1 \ges N_3$.
\newline
\indent
In this case, we have $N_1 \sim N_{\max}$.
We can proceed as in Subcase (E.1)
with $v_{N_2}$ replaced by $z_{3, N_2}$.
More precisely, 
when $\max(N_2, N_3) \geq N_1^\frac{1}{10}$, 
we apply   H\"older's inequality  with~\eqref{XX2}
as in Subsubcase (E.1.a).
It suffices to note that 
\begin{align*}
 N_2^{-2s+}N_3^{-s+} 
 \les 
N_{\max}^{0-}, 
\end{align*}

\noi
provided that $ s> 0$.
When $\max(N_2, N_3) \ll N_1^\frac{1}{10}$, 
we proceed with 
Lemmas  \ref{LEM:Str}, \ref{LEM:biStr},  and \ref{LEM:Z3_2}
as in Subsubcase (E.1.b).
In this case, it suffices to note that 
\begin{align*}
 N_2^{1-2s-}
N_3^{-s}N_4^{-\frac{1}{2}+}
\les N_{\max}^{0-}, 
\end{align*}

\noi
provided that $s\geq 0$.

\smallskip

\noi
{\bf $\bullet$ Subcase (G.2):} $N_2 \sim N_3 \gg N_1$.
\newline
\noi
$\circ$ \underline{Subsubcase (G.2.a):} $N_1\ll N_3^\frac{1}{2}$.
\quad 
By Lemmas \ref{LEM:Str}  and  \ref{LEM:biStr},  
 we have
\begin{align*}
\text{LHS of } \eqref{XX1}
&  \les N_3^\s 
 \|v_{ N_1}  z_{3, N_2} \|_{L^2_{T, x}}
 \|z_{3, N_3}\|_{L^5_{T, x}}
 \|w_{N_4}\|_{L^\frac{10}{3}_{T, x}}
\\
& \le
T^{\frac 1{10}}
N_1^{1-\s-} N_2^{-2s-\frac 12 + }
N_3^{\s-s} \| v\|_{X^\s_T}
\|z_{3, N_2}\|_{X^{2s-}_T}
\| \jb{\nb}^s z_1\|_{L^{15}_T L^{\frac{15}{4}}_x}
\| z_1\|_{L^{15}_T L^{\frac{15}{4}}_x}^2 \\
& \le
T^{\frac1{10}}
N_3^{\frac\s2 -3 s+}
C(\|z_1\|_{B^s(T)}, \|z_3\|_{X^{2s-}_T}) \| v\|_{X^\s_T}\\
& \leq 
T^{0+} N_{\max}^{0-} C(\|z_1\|_{B^s(T)}, \|z_3\|_{X^{2s-}_T}) \| v\|_{X^\s_T},
\end{align*}

\noi
provided that $\s \leq \min (1, 6s-)$.

\smallskip

\noi
$\circ$ \underline{Subsubcase (G.2.b):} $N_1\ges N_3^\frac{1}{2} $.
\quad 
In this case, we proceed with
$L^\frac{10}{3}_{T, x}, 
L^5_{T, x}, 
L^{5}_{T, x}, L^\frac{10}{3}_{T, x}$-H\"older's inequality
with Lemma \ref{LEM:Z3_2}.
It  suffices to note that 
\begin{align*}
N_1^{-\s} N_2^{-s}N_3^{\s-s} 
\les N_3^{ \frac \s2 - 2s }
\les  N_{\max}^{0-},
\end{align*}

\noi
provided that $0 \leq \s < 4s$.

\smallskip

\noi
{\bf $\bullet$  Subcase (G.3):} $N_3 \sim N_4 \sim N_{\max} \gg N_1, N_2$.
\newline
\noi
$\circ$ \underline{Subsubcase (G.3.a):} 
$N_1, N_2  \ll N_3^\frac{1}{2} $.
\quad 
In this case, we proceed as in Subsubcase (F.3.a).
It suffices to note that 
\begin{align*}
 N_1^{1- \s-}N_2^{1-2s - }  N_3^{\s- 2 s -\frac 12 +}N_4^{-\frac{1}{2}+}
\les N_3^{\frac \s2 - 3s+}
\les N_{\max}^{0-}, 
\end{align*}

\noi
provided that $ \s \leq \min(1, 6s-)$ and $ s < \frac 12$.

\smallskip

\noi
$\circ$ \underline{Subsubcase (G.3.b):} $N_1, N_2\ges N_3^\frac{1}{2} $.
\quad 
In this case, we proceed with
$L^\frac{10}{3}_{T, x}, 
L^5_{T, x}, 
L^{5}_{T, x}, L^\frac{10}{3}_{T, x}$-H\"older's inequality
as in Subsubcase (F.3.b)
but we apply Lemma \ref{LEM:Z3_2}
to estimate the $L^{5}_{T, x}$-norm of $z_{3, N_j}$, $j = 2, 3$.
The regularity restriction is precisely the same as that in 
Subsubcase~(F.3.b).

\smallskip

\noi
$\circ$ \underline{Subsubcase (G.3.c):} $N_1\ges N_3^\frac{1}{2} \gg N_2$.
\quad 
By Lemmas \ref{LEM:Str},  \ref{LEM:biStr}, and~\ref{LEM:Z3_2},  
 we have
\begin{align*}
\text{LHS of } \eqref{XX1}
&  \les N_3^\s   \| v_{N_1} \|_{L^\frac{10}{3}_{T, x}} 
 \| z_{3, N_3} \|_{L^{5}_{T, x}} 
 \|z_{3, N_2}  w_{N_4} \|_{L^2_{T, x}}
 \notag\\
& \les T^{\frac{1}{10}+}
N_1^{-\s} N_2^{1-2s-}
N_3^{\s- s}N_4^{-\frac{1}{2}+}
 \| v\|_{X^\s_T}
 \| z_{3, N_2}\|_{X^{2s-}_T}
\| \jb{\nb}^s z_1\|_{L^{15}_T L^{\frac{15}{4}}_x}
\| z_1\|_{L^{15}_T L^{\frac{15}{4}}_x}^2\\
& \le T^{\frac{1}{10}+}
N_3^{\frac \s2- 2s+} 
 C(\|z_1\|_{B^s(T)}, \|z_3\|_{X^{2s-}_T}) \| v\|_{X^\s_T}.
\end{align*}

\noi
This yields \eqref{XX1}, 
provided that $0\leq \s < 4s$ and $s < \frac 12$.

\smallskip

\noi
$\circ$ \underline{Subcase (G.3.d):} $N_2\ges N_3^\frac{1}{2} \gg N_1$.
\quad 
By Lemmas \ref{LEM:Str},  \ref{LEM:biStr}, and~\ref{LEM:Z3_2},  
 we have
\begin{align*}
\text{LHS of } \eqref{XX1}
&  \les N_3^\s  
 \|v_{N_1}  z_{3, N_3} \|_{L^2_{T, x}}
  \| z_{3, N_2} \|_{L^{5}_{T, x}} 
 \| w_{N_4} \|_{L^\frac{10}{3}_{T, x}}
\\
& \les T^{\frac{1}{10}+}
N_1^{1-\s-} N_2^{-s}
N_3^{\s- 2s-\frac 12+}
 \| v\|_{X^\s_T}
\| \jb{\nb}^s z_1\|_{L^{15}_T L^{\frac{15}{4}}_x}
\| z_1\|_{L^{15}_T L^{\frac{15}{4}}_x}^2
\| z_{3, N_3} \|_{X^{2s-}_T}
 \\
& \le T^{\frac 1{10}+}
N_1^{\frac \s2- \frac 52s+} 
 C(\|z_1\|_{B^s(T)}, \|z_3\|_{X^{2s-}_T}) \| v\|_{X^\s_T}.
\end{align*}

\noi
This yields \eqref{XX1}, 
provided that 
 $ \s \leq \min(1, 5s-)$
 and $s\geq0$.

\medskip
\noi
{\bf  Case (H):} $v z_3 z_1$ case.
\newline
\noi
{\bf $\bullet$  Subcase (H.1):} $N_1 \sim N_{\max}$.
\newline
\indent
In this case, we can proceed as in Subcase (D.1)
by replacing $v_{N_2}$ with $z_{3, N_2}$.
When $\max(N_2, N_3) \geq N_1^\frac{1}{10}$, 
it suffices to note that 
\begin{align*}
 N_2^{-2s+}
N_3^{-s}
\les N_{\max}^{0-}, 
\end{align*}

\noi
provided that $s> 0$.
When $\max(N_2, N_3) \ll N_1^\frac{1}{10}$, 
it suffices to note that 
\begin{align*}
 N_2^{1-2s-}
N_3^{-s}N_4^{-\frac{1}{2}+}
\les N_{\max}^{0-}, 
\end{align*}

\noi
provided that $s\geq 0$.

\smallskip

\noi
{\bf $\bullet$  Subcase (H.2):} $N_3 \sim N_{\max} \gg N_1$.
\newline
\noi
$\circ$ \underline{Subsubcase (H.2.a):} 
$N_1, N_2  \ll N_3^\frac{1}{2} $.
\quad 
In this case, we have $N_3 \sim N_4 \sim N_{\max}$.
Then, proceeding as in Subsubcase (F.3.a), 
it suffices to note that 
\begin{align*}
 N_1^{1- \s-}N_2^{1-2s - }  N_3^{\s-  s -\frac 12 +}N_4^{-\frac{1}{2}+}
\les N_3^{\frac \s2 - 2s}
\les N_{\max}^{0-}, 
\end{align*}

\noi
provided that $ \s \leq \min(1, 4s-)$ and $ s < \frac 12$.

\smallskip

\noi
$\circ$ \underline{Subsubcase (H.2.b):} $N_1, N_2\ges N_3^\frac{1}{2} $.
\quad 
In this case, we proceed with
$L^\frac{10}{3}_{T, x}, 
L^\frac{10}{3}_{T, x}, 
L^{10}_{T, x}, L^\frac{10}{3}_{T, x}$-H\"older's inequality.
It suffices to note that
\begin{align*}
 N_1^{- \s}N_2^{-2s+}  N_3^{\s-  s}
\les N_3^{\frac \s2 - 2s+}
\les N_{\max}^{0-}, 
\end{align*}

\noi
provided that $0 \leq  \s < 4s$.

\smallskip

\noi
$\circ$ \underline{Subsubcase (H.2.c):} $N_1\ges N_3^\frac{1}{2} \gg N_2$.
\quad 
In this case, we have $N_3 \sim N_4 \sim N_{\max}$.
We can proceed as in Subsubcase~(G.3.c)
with $z_{3, N_3}$ replaced by $z_{1,N_3}$
(without applying  Lemma~\ref{LEM:Z3_2}).
The regularity restriction is precisely the same
as in Subsubcase (G.3.c).

\smallskip

\noi
$\circ$ \underline{Subsubcase (H.2.d):} $N_2\ges N_3^\frac{1}{2} \gg N_1$.
\quad 

We first consider the case $N_3 \sim N_4 \sim N_{\max}$.
By Lemmas \ref{LEM:Str} and \ref{LEM:biStr},  
 we have
\begin{align*}
\text{LHS of } \eqref{XX1}
&  \les N_3^\s  
 \|v_{N_1}  w_{N_4} \|_{L^2_{T, x}}
 \| z_{3, N_2} \|_{L^\frac{10}{3}_{T, x}}
  \| z_{1, N_3} \|_{L^{5}_{T, x}} 
\\
& \les T^{0+}
N_1^{1-\s-} N_2^{-2s+}
N_3^{\s- s} N_4^{-\frac 12+}
 C(\|z_1\|_{B^s(T)},  \| z_{3}\|_{X^{2s-}_T}) \| v\|_{X^\s_T}
 \\
& \le T^{0+}
N_3^{\frac \s2- 2s+} 
 C(\|z_1\|_{B^s(T)},  \| z_{3}\|_{X^{2s-}_T}) \| v\|_{X^\s_T}.
\end{align*}

\noi
This  yields \eqref{XX1}, 
provided that 
 $ \s \leq \min(1, 4s-)$ and $ s> 0$.

When  $N_4\ll N_3$, 
we have $N_2 \sim N_3  \sim N_{\max}$.
In this case, we can repeat the computation above 
with the roles of $z_{3, N_2}$ and $w_{N_4}$
switched.
Noting that 
\begin{align*}
N_1^{1-\s-} N_2^{-2s-\frac 12+}
N_3^{\s- s} 
\les N_3^{\frac \s2- 3s+} \les N_{\max}^{0-}, 
\end{align*}

\noi
we obtain 
 \eqref{XX1}, 
provided that 
 $ \s \leq \min(1, 6s-)$.

\smallskip

\noi
{\bf $\bullet$ Subcase (H.3):} $N_2 \sim N_4 \gg N_1, N_3$.
\newline
\noi
$\circ$ \underline{Subsubcase (H.3.a):} $N_1\ll N_2^\frac{1}{2}$.
\quad 
In this case, we can proceed as in Subsubcase (H.2.d).
It suffices to note that 
\begin{align*}
N_1^{1-\s-} N_2^{\s -2s + }
N_3^{-s} N_4^{-\frac 12 +}
\les 
N_2^{\frac\s2 -2 s+}
\les N_{\max}^{0-}, 
\end{align*}

\noi
provided that $\s \leq \min (1, 4s-)$ and $s \geq 0$.

\smallskip

\noi
$\circ$ \underline{Subsubcase (H.3.b):} $N_1\ges N_2^\frac{1}{2} $.
\quad 
In this case, we proceed with
$L^\frac{10}{3}_{T, x}, 
L^\frac{10}{3}_{T, x}, 
L^{10}_{T, x}, L^\frac{10}{3}_{T, x}$-H\"older's inequality.
It suffices to note that
\begin{align*}
 N_1^{- \s}N_2^{\s-2s+} N_3^{-s}
\les N_2^{\frac \s2 - 2s+}
\les N_{\max}^{0-}, 
\end{align*}

\noi
provided that $0 \leq  \s < 4s$.

\medskip
\noi
{\bf  Case (\,I\,):} $v vv$ case.
\newline
\indent 
In this case, there is no need to perform the dyadic decomposition.
By H\"older's inequality, Sobolev's inequality \eqref{Sobolev}, 
and Lemma \ref{LEM:Str}, 
 we have
\begin{align*}
\text{LHS of } \eqref{XX1}
&  \les 
  \|v \|_{L^{5}_{T, x}}^2
 \| \jb{\nb}^\s v \|_{L^\frac{10}{3}_{T, x}}   
 \| w \|_{L^\frac{10}{3}_{T, x}}
\\
&  \les 
  \|\jb{\nb}^\frac{1}{2} v \|_{L^{5}_{T} L^\frac{30}{11}_x}^2
 \| \jb{\nb}^\s v \|_{L^\frac{10}{3}_{T, x}}   
\\
&  \les 
\| v\|_{X^\s_T}^3,
\end{align*}

\noi
provided that 
 $ \s \geq \frac 12$.

\medskip

Putting all the cases
including 
Cases (A) and (B) treated in 
Lemmas \ref{LEM:Z5} and \ref{LEM:Z7}, 
we conclude that \eqref{XX1} holds, 
provided that  $ \frac 12 \leq  \s\leq 1$
and 
 $\frac 25 \s < s < \frac{1}{2}$.
This completes the proof of Proposition \ref{PROP:NL1}.
\end{proof}

\section{Proof of Theorem \ref{THM:2}}
\label{SEC:THM2}

In this section, we briefly discuss the 
proof of Theorem \ref{THM:2}.
Given $ \frac 16 < s < \frac 12$, 
let $k \in \NB $ such that $\frac{1}{2\al_{k+1}} < s \leq \frac{1}{2\al_{k}}$,
where $\al_k$ is defined in \eqref{al2}.
Our main goal is to  study the fixed point problem 
\eqref{NLS9}.
Define $\wt \G$ by 
\begin{equation*}
\wt \G v(t) = -i  \int_0^t S(t - t')
\bigg\{ \N\bigg(v +\sum_{\l= 1}^{k} \zeta_{2\l-1}\bigg) 
-  \sum_{\l = 2}^{k}
\sum_{\substack{j_1 + j_2 + j_3 = 2\l - 1\\j_1, j_2, j_3 \in \{1,  2\l-3\}}} 
 \zeta_{j_1} \cj{\zeta_{j_2}}\zeta_{j_3}\bigg\}(t') dt', 
%\label{NLSx2}
\end{equation*}

\noi
where $\zeta_{2\l-1}$, $\l = 1, \dots, k$, 
is as in \eqref{zeta2}.
 Theorem \ref{THM:2} follows
from a standard fixed point argument, 
once we prove the following proposition.

\begin{proposition}\label{PROP:NL2}
Given $ \frac 16 < s < \frac 12$, 
let $k \in \NB $ such that $\frac{1}{2\al_{k+1}} < s < \frac{1}{\al_{k}}$.
Then, 
  there exist $\theta > 0$,  $C_1, C_2> 0$, 
and an almost surely finite constant $R = R(\o) >0$ such that 
\begin{align}
\|\wt \G v\|_{X^\frac{1}{2}([0, T])}
& \leq C_1
\big(\|v\|_{X^\frac{1}{2}([0, T])} ^3 + T^\theta R(\o)\big),  
\label{nl2a}\\
\|\wt \G v_1 - \wt \G v_2  \|_{X^\frac{1}{2}([0, T])}
& \leq C_2
\Big(\sum_{j = 1}^2 \|v_j\|_{X^\frac{1}{2}([0, T])} ^2
+ T^\theta R(\o)\Big)
\|v_1 -v_2 \|_{X^\frac{1}{2}([0, T])},
\label{nl2b}
\end{align}

\noi
for all $v, v_1, v_2 \in X^\frac{1}{2}([0, T])$
and $0 < T\leq 1$.

\end{proposition}

On the one hand,  the upper bound $\frac{1}{\al_k}$ on the range of $s$
in Proposition \ref{PROP:NL2}
comes from the restriction in Proposition \ref{PROP:Zk}.
On the other hand, given $\frac 16 < s < \frac 12$, 
we fix $k \in \NB$ such that 
 $\frac{1}{2\al_{k+1}} < s \leq \frac{1}{2\al_{k}}$
 and 
 apply Proposition \ref{PROP:NL2}.
Namely, the upper bound on $s$ in 
 Proposition \ref{PROP:NL2}
 does not cause any problem in proving Theorem \ref{THM:2}.

\begin{proof}
We only discuss the proof of  \eqref{nl2a} since  \eqref{nl2b} follows in a  similar manner.
As in the proof of Proposition \ref{PROP:NL1}, 
it suffices to perform 
a case-by-case analysis of expressions of the form:
\begin{align*}
\bigg| \int_0^T \int_{\R^3}
\jb{\nb}^\frac{1}{2} ( w_1 w_2 w_3 )w dx dt\bigg|,
\end{align*}

\noi
where $\|w\|_{Y^{0}_T} \leq 1$
and $w_j=  v$, $\zeta_{2\l-1}$, $\l = 1, \dots, k$, 
but not of the form
$\zeta_1 \zeta_1\zeta_{2\l-3}$
for  $\l  \in \{ 2, 3, \dots, k\}$.
Here,  we have dropped  the complex conjugate sign
on $w_2$ 
since it does not play any  role in our analysis.
More concretely,  we need to consider the following cases:
\begin{center}
\begin{tabular}{clclcl}
(A)& $\zeta_1 \zeta_1 \zeta_{2k-1}$
\qquad\qquad  
& \text{(D)}& $v v \zeta_1$             
\qquad \qquad 
& \text{(G)}& $v \zeta_{j_2} \zeta_{j_3}$            
\rule[-3mm]{0pt}{0pt} 
 \\
 \text{(B)}& $\zeta_1 \zeta_{j_2} \zeta_{j_3}$     
\qquad \qquad
& \text{(E)}& $v v \zeta_{j_3}$
\qquad \qquad
& \text{(H)}& $v \zeta_{j_2} \zeta_1$
\rule[-3mm]{0pt}{0pt} 

\\
 \text{(C)}& $\zeta_{j_1} \zeta_{j_2} \zeta_{j_3}$
\qquad \qquad
& \text{(F)}& $v \zeta_1 \zeta_1$
\qquad \qquad
& \text{(\,I\,)}& $v v v$

\end{tabular}
\end{center}

\noi
where   $j_1, j_2, j_3$ can take any value in $\{3, 5, \dots, 2k-1\}$.
As we see below, the worst interaction appears in Case (A)
and all the other cases can be handled as in  the proof of Proposition~\ref{PROP:NL1}.

We first point out that, in the proof of Proposition \ref{PROP:NL1},
the regularity restriction 
coming from Cases (B) - (I) (with $\s = \frac 12$) is $s > \frac 16$.
Moreover, by comparing Proposition \ref{PROP:Z3}
and Lemma \ref{LEM:Z3_2}
with 
Proposition \ref{PROP:Zk}
and Lemma \ref{LEM:Zk_2}, 
we see that 
the unbalanced higher order term $\zeta_{2\l-1}$
for  $\l  \in \{ 2, 3, \dots, k\}$
enjoys (at least) the same regularity properties
as $z_3 = \zeta_3$
both in terms of differentiability and space-time integrability.
Therefore, we can simply  apply the estimates
in the proofs of Lemma \ref{LEM:Z7}\,(i)
for Case (B)
and Proposition \ref{PROP:NL1}
for Cases (C) - (I)
and conclude that the contributions
from Cases (B) - (I) can be bounded, provided $s > \frac 16$.

It remains to consider  Case (A).
In view of \eqref{zeta}, 
the contribution from Case (A) is nothing but $\zeta_{2k+1}$.
Hence, by applying 
Proposition \ref{PROP:Zk}
with $\s = \frac 12$, 
we conclude that 
the contribution in Case (A) can be estimated by 
$T^\theta R(\o)$, 
provided that 
$\frac{1}{2\al_{k+1}} < s < \frac{1}{\al_{k}}$.
This completes the proof of Proposition \ref{PROP:NL2}.
\end{proof}

\appendix

\section{On the deterministic non-smoothing
of the Duhamel integral operator}
\label{APP:A}

In this appendix, we 
show that there is no deterministic smoothing on the Duhamel integral operator
$\I$ defined in \eqref{Duhamel1} when $s \leq \frac 12$.

\begin{lemma}\label{LEM:Z3_1}

Let   $\s > 3s-1$.
Then, 
the estimate~\eqref{Z3_3}  does not hold.
In particular, there is no deterministic nonlinear smoothing when $s \leq  \frac 12$.

\end{lemma}

\begin{proof}

Given $N \geq  L \geq  \l \gg 1$,  
consider $u_j = S(t) \phi_j$,  $j = 1, 2, 3$,  
such that 
$\ft \phi_j( \xi)  \sim  \ind_{A_j}(\xi)$, 
 where 
 \begin{align*}
 A_1 = L e_1 + \l Q, \qquad 
  A_2 = \l Q, \qquad \text{and}\qquad
  A_3 = N  e_2 + \l Q.
\end{align*}

\noi
Then, we have 
\begin{align}
\prod_{j = 1}^3 \| u_j(0) \|_{H^s}
\sim \l^{\frac{9}{2} + s} L^s N^s.
\label{O12}
\end{align}

Let $A = Le_1 + N e_2 + \l Q$.
By writing
\begin{align}
\ind_{A}  (\xi)\F_x[\I(u_1, & u_2, u_3)](t, \xi) 
\notag \\
& 
= -i \ind_{A} (\xi) e^{-it |\xi|^2} \int_0^t \intt_{\xi = \xi_1 - \xi_2 + \xi_3} 
e^{i t' \Phi(\bar \xi)}  \prod_{j = 1}^3 \ind_{A_j} (\xi_j) d\xi_1 d\xi_2 dt',
\label{O13}
\end{align}

\noi
we see that 
 the non-trivial contribution to \eqref{O13}
comes from 
$\xi_j \in A_j$, $j = 1, 2, 3$.
Moreover, in this case, $\Phi(\bar \xi)$ defined \eqref{O7}
satisfies $|\Phi (\bar \xi) |\les N L$.
In particular, by choosing 
\begin{align}
t_* \sim \frac1{NL}
\label{O14}
\end{align}

\noi
such that $t_*|\Phi (\bar \xi) |\ll 1 $,
we have \eqref{O8} for all $t' \in [0, t_*]$.
Hence, by Lemma \ref{LEM:conv} with \eqref{O14}, 
we obtain 
\begin{align}
\| \I(u_1,  u_2, u_3) \|_{X^\s([0, 1])} \ges \| \I(u_1,  u_2, u_3)(t_*) \|_{  H^\s}
\ges t_*  \l^\frac{15}{2} N^\s 
\sim  \l^\frac{15}{2} L^{-1}N^{\s-1} .
\label{O15}
\end{align}

\noi
Therefore, it follows from \eqref{O12} and \eqref{O15} that 
by choosing $\l \sim L \sim N$, we have
\begin{align*}
\frac{\| \I(u_1,  u_2, u_3) \|_{X^\s([0, 1])}}
{\prod_{j=1}^3 \| u_j(0) \|_{H^s}}
\ges  \l^{3-s}L^{-1-s} N^{\s-1-s} 
\sim N^{\s - 3s+1} \too \infty, 
\end{align*}

\noi
as $N \to \infty$, provided that $\s > 3s-1$.
This shows that the estimate \eqref{Z3_3} can not hold
when $\s > 3s-1$.
\end{proof}

Define the quintilinear operator $\I^{(2)}$ by 
\begin{align*}
\I^{(2)}(u_1, \dots, u_5) (t) 
:= -i  \sum_{\l = 1}^3 
\int_0^t S(t - t') 
 v_{j_1} \cj{v_{j_2}}v_{j_3} (t') dt',
\end{align*}

\noi
where 
$v_{j_i} = u_i$ for $i \ne \l$
and $v_{j_\l} = \I(u_\l, u_4, u_5)$.
This basically corresponds to the second order term appearing
in the power series expansion for \eqref{NLS1}.
In particular, we have  that $z_5 = \I^{(2)}(z_1, \dots, z_1)$.

By a computation similar to the proof of Lemma \ref{LEM:Z3_1}, 
we can prove the following deterministic non-smoothing
on the second order term.
Let $\s > 5s - 2$.
Then, there is no finite constant $C> 0$ such that 
\begin{align*}
\big\|\I^{(2)}( S(t) \phi_1, \dots,  S(t) \phi_5)\big\|_{X^\s([0, 1])}
\leq  C \prod_{j = 1}^5 \| \phi_j  \|_{H^s}
\end{align*}

\noi
for all $\phi_j \in H^s(\R^3)$.
In particular, this shows that 
when $s \le \frac 12$,
there is no deterministic  smoothing on the second order term $\I^{(2)}$.
A similar comment applies to the higher order terms.

\begin{acknowledgment}

\rm 
This research was partially supported by 
Research in Groups at International Centre for Mathematical Sciences, Edinburgh, United  Kingdom.
%This work was partially supported by a grant from the Simons Foundation (No.~246024 to \'Arp\'ad B\'enyi).
% \'Arp\'ad B\'enyi 
 \'A.B.~was
 partially supported by a grant from the Simons Foundation (No.~246024).
T.O.~was supported by the European Research Council (grant no.~637995 ``ProbDynDispEq'').

\end{acknowledgment}

\end{document}